\documentclass[11pt]{article}
\usepackage{amsmath}
\usepackage{mathtools}
\usepackage{amsthm}
\usepackage{amssymb}
\usepackage{graphicx}
\usepackage{multirow} 
\usepackage{tabularx}

\newcommand{\Z}{\mathbb{Z}}
\newcolumntype{R}{>{\raggedleft\arraybackslash}X}
\newcolumntype{L}{>{\raggedright\arraybackslash}X}

\usepackage[a4paper,margin=3cm,right=3cm,includehead]{geometry}

\usepackage[colorlinks=true,citecolor=black,linkcolor=black,urlcolor=blue]{hyperref}

\newcommand{\arxiv}[1]{\href{http://arxiv.org/abs/#1}{\texttt{arXiv:#1}}}

\newtheorem{theorem}{Theorem}

\newtheorem{conjecture}[theorem]{Conjecture}

\title{ \bf The degree-diameter problem for circulant graphs of degree 10 and 11 - extended version}
\author{Robert R. Lewis\\[-0.8ex]
\small Department of Mathematics and Statistics\\[-0.8ex]
\small The Open University\\[-0.8ex]
\small Milton Keynes, UK\\
\small\tt{robert.lewis@open.ac.uk}
}

\date{17th March 2018}



\begin{document}
\maketitle

\begin{abstract}

This paper considers the degree-diameter problem for undirected circulant graphs. For degrees 10 and 11 newly discovered families of circulant graphs of arbitrary diameter are presented which are largest known and are conjectured to be extremal. They are also the largest-known Abelian Cayley graphs of these degrees. For each such family the order of every graph in the family is defined by a quintic polynomial function of the diameter which is specific to the family. The elements of the generating set for each graph are similarly defined by a set of polynomials in the diameter.  The existence of the graphs in the degree 10 families has been proved for all diameters. These graphs are consistent with a conjecture on the order of extremal Abelian Cayley and circulant graphs of any degree and diameter.

This is the extended version of the paper, including the proof steps for degree 10 graphs covering all diameter classes and an appendix listing additional tables of generating sets.

\end{abstract}


\section{Introduction}

The \emph{degree-diameter problem} is the problem of finding graphs with the largest possible number of vertices $n(d,k)$ for a given maximal degree $d$ and diameter $k$. We will call such graphs \emph{extremal} graphs. From the literature it is seen that this problem has been tackled for undirected, directed and mixed graphs. In addition to the general case various subproblems have also been explored, including vertex-transitive graphs and Cayley graphs. For a general background on the degree-diameter problem see the comprehensive survey by Miller and \v{S}ir\'{a}\v{n} \cite{Miller} and the tables of largest-known graphs on the CombinatoricsWiki website \cite{Wiki}.

Only in relatively few cases are the largest-known graphs believed to be extremal, typically restricted to degree 3 for small diameter or diameter 2 for small degree. Circulant graphs, which are  Cayley graphs of cyclic groups and therefore highly structured, are a noteworthy exception.
In this paper $CC(d,k)$ denotes the order of an extremal undirected circulant graph of degree $d$ and diameter $k$, and $AC(d,k)$ similarly for Abelian Cayley graphs. $L_{CC}(d,k)$ is the order of the largest-known circulant graph when the extremal order is unknown. For even degree $d$, if $L_{CC}(d,k)$ is even we additionally define $L_{OC}(d,k)$ to be the order of the largest-known graph of odd order.

Infinite families of extremal undirected circulant graphs have been identified and proven extremal by various authors for degrees $d= 2, 3, 4$ and $5$, with order $CC(d,k)$ defined by a polynomial in the diameter $k$ for any diameter \cite{Dougherty}. Similar families of largest-known circulant graphs of order $L_{CC}(d,k)$ were discovered for degree 6 by Monakhova in 2003 \cite{Monakhova} and independently by Dougherty and Faber, also for degree 7, in 2004 \cite{Dougherty}.

It happens that $CC(2,k), CC(4,k)$ and $L_{CC}(6,k)$ are all odd for any diameter $k$. Families of largest-known degree 8 graphs of odd order $L_{OC}(8,k)$ were discovered by Monakhova in 2013 and conjectured to be extremal \cite{Monakhova1}. However Monakhova had limited her search to odd order graphs because of a 1994 paper by Muga in which he mistakenly claimed that any extremal circulant graph of even degree and arbitrary diameter has odd order \cite{Muga}. The argument was flawed and the smallest counterexample is the extremal graph of degree 8 and diameter 3, which has order 104. Families of largest-known graphs of degree 8 and 9 were discovered by the author in 2014 \cite{Lewis8}. For degree 8, these graphs have even order for any diameter $k\ge 3$.

These circulant graph families of degree 6 to 9 all remain largest known to date. They have been proven by computer search to be extremal for small diameters, and they are conjectured to be extremal for all larger diameters. For details of their order, generating sets and range of proven extremality see \cite{Lewis8}. The main result of this paper is the construction of infinite families of largest-known circulant graphs of degree 10 and 11, with order and generating sets defined by polynomials in the diameter, which are conjectured to be extremal.

A circulant graph is so called because its adjacency matrix is a circulant matrix.  As mentioned, a circulant graph $X$ of order $n$ may also be viewed as a Cayley graph whose vertices are the elements of the cyclic group $\Z_n$. Two vertices $i,j$ are connected by an arc $(i, j)$ if and only if $j-i$ is an element of $C$, a subset of $\Z_n \setminus 0$, called the \emph {connection set}. We may also denote the graph by $X(\Z_n,C)$. If $C$ is closed under additive inverses then $X$ is an undirected graph, and in this paper we will only consider undirected graphs. By definition such a graph is regular, with the degree $d$ of each vertex equal to the cardinality of $C$. If $n$ is odd then $\Z _n \setminus 0 $ has no elements of order $2$. Therefore $C$ has even cardinality, say $d=2f$, and comprises $f$ complementary pairs of elements with one of each pair strictly between $0$ and $n/2$. Any set of size $f$ containing exactly one element from each pair is sufficient to uniquely determine the connection set and is called a \emph {generating set}. Without loss of generality in this paper we will choose the generating set which is comprised of the $f$ elements of $C$ between 0 and $n/2$ as the canonical generating set $G$ for $X$. If $n$ is even then $\Z_n \setminus 0$ has just one element of order $2$, namely $n/2$. In this case $C$ comprises $f$ complementary pairs of elements, as for odd $n$, with or without the addition of the involutory element $n/2$. If $C$ has odd cardinality, so that $d=2f+1$, then the value of its involutory element is defined by the value of $n$. Therefore for a circulant graph of given order and degree, its connection set $C$ is completely defined by specifying its generating set $G$. The cardinality of the connection set is equal to the degree $d$ of the graph, and the cardinality of the generating set, $f$, is defined to be the dimension of the graph.

Clearly if every element of a generating set is multiplied by a constant factor that is co-prime with the order of the graph then the resultant set will also be a generating set of a graph which is isomorphic to the first. Therefore the isomorphism class of a circulant graph could have a number of different generating sets. Not all isomorphism classes of circulant graphs have a primitive generating set (where one of the generators is 1). An example of an extremal circulant graph with no primitive generating set is the graph with degree 9, diameter 2, order 42 and generating set \{2, 7, 8, 10\}. For simplicity, where a graph has at least one primitive generating set then we will only consider the primitive sets. In fact it emerges that every family of extremal or largest-known circulant graphs so far discovered has at least one primitive generating set.


\section{Conjectured order of extremal Abelian Cayley and circulant graphs of any degree and diameter}
\label{section:Conjectured}

We briefly review general upper and lower bounds for the order of extremal Abelian Cayley and circulant graphs of arbitrary degree $d$ and diameter $k$. For Abelian Cayley graphs, and thus in particular for circulant graphs, an upper bound that is much sharper than the general Moore bound was established for even degree by Wong and Coppersmith in 1974 \cite{Wong} and further sharpened by Muga in 1994 \cite{Muga}. Dougherty and Faber developed an equivalent upper bound for odd degree in 2004 \cite{Dougherty}.

For positive integers $f, k$, we define $S_{f,k}$ to be the set of elements of $\Z^f$ (the $f$-dimensional direct product of $\Z$ with itself) which can be expressed as a word of length at most $k$ in the canonical generators $\textbf{e}_i$ of $\Z^f$, taken positive or negative. Equivalently, $S_{f,k}$ is the set of points in $\Z^f$ distant at most $k$ from the origin under the $\ell^1$ (Manhattan) metric: $S_{f,k}=\{(x_1,...,x_f) \in \Z^f : \vert x_1 \vert +...+\vert x_f \vert \leq k\}$. Within the literature on coding theory and tiling problems $S_{f,k}$ is often called the $f$-dimensional Lee sphere of radius $k$, although it appears more diamond-like than spherical, having the form of a regular dual $f$-cube.

For an Abelian Cayley graph of degree $d$ and diameter $k$, and corresponding dimension $f=\lfloor d/2 \rfloor$ , Muga's and Dougherty and Faber's upper bounds are defined by:
\[M_{AC}(d,k) =
\begin{cases}
\ |S_{f,k}| &\mbox{ for even } d\\
\ |S_{f,k}|+|S_{f,k-1}| &\mbox { for odd } d,
\end{cases}
\]
where, by \cite{Stanton},
\[|S_{f,k}|= \sum _{i=0}^f 2^i \binom {f}{i} \binom {k}{i}.\]

For even or odd degree, this is a polynomial in $k$ of order $f$:
\[M_{AC}(d,k) =
\begin{dcases}
\ {\frac{2^f}{ f!}} k^f+{\frac{2^{f-1}}{(f-1)!)}} k^{f-1}+O(k^{f-2}) \mbox{\quad\ \ \ for even } d \\
\ \frac{2^{f+1}}{f!} k^f \qquad \qquad \qquad \; +O(k^{f-2}) \mbox{\quad\ \ \ for odd } d.\\
\end{dcases}
\]

A lower bound for the order of extremal circulant graphs of even degree $d$ and arbitrary diameter $k$ was established by Chen and Jia in 1993 \cite{Chen}:
\[
CJ(d,k) = \frac{1}{2}\left ( \frac{4}{f}\right ) ^f k^f+\frac{b}{2}\left ( \frac{4}{f}\right ) ^{f-1}k^{f-1}+O(k^{f-2})
\]
where $f=d/2$, $k \ge f \ge 3$, and $b \le 13-4f$ with equality if and only if $k \equiv f-3 \pmod f$.

$CJ(d,k)$ is also a polynomial in $k$ of order $f$ but with an asymptotically smaller leading coefficient than $M_{AC}(d,k)$. For even degree $d$, the ratio of the leading coefficients of $CJ(d,k)$ and $M_{AC}(d,k)$, denoted by $R_f$, is given by

\[
R_f = 2^{f-1}\frac{f!}{f^f} .
\]

We see that $R_1=R_2=1$, and then $R_f$ decreases monotonically to zero with increasing $f$.

More generally, Dougherty and Faber also established a lower bound for the order of extremal Abelian Cayley graphs of even degree $d$ and diameter $k \ge (d-2)/4$, with corresponding dimension $f=d/2$ \cite{Dougherty}:
\[
DF_{AC}(d,k) = \frac{1}{2}\left ( \frac{4}{f}\right ) ^f k^f+\left ( \frac{4}{f}\right ) ^{f-1}k^{f-1}+O(k^{f-2}).
\]

It is noteworthy that these have the same leading coefficients as the Chen and Jia circulant graph lower bound, although their second coefficients are greater.

From a comparison of the formulae for the order of the families of extremal and largest-known circulant graphs of degree 2 to 9, common relationships may be discerned between the first and second terms and between the degrees \cite{Lewis8}.
For even degrees the first and second terms are the same as Dougherty and Faber's Abelian Cayley graph lower bound, $DF_{AC}(d,k)$.
The author conjectured in \cite {Lewis8} that these relationships might extend to all higher degrees (both odd and even) and apply to families of extremal circulant graphs of any degree. It is further conjectured that this is true not only for circulant graphs but also for extremal Abelian Cayley graphs in general.
This leads to the following formal statement of the extremal Abelian Cayley and circulant graph order conjecture.

\begin{conjecture}
Let the order of an extremal Abelian Cayley graph of degree $d$ and diameter $k$ be $AC(d,k)$ and the order of an extremal circulant graph of degree $d$ and diameter $k$ be $CC(d,k)$. For any degree $d \ge 2$ and corresponding dimension $f=\lfloor d/2 \rfloor$ there exist polynomials $AC^*_{d, k'}(k)$ and $CC^*_{d,k'}(k)$ in $k$ of degree $f$ for $0 \le k' < f$ such that $AC(d,k)=AC^*_{d,k'}(k)$ and $CC(d,k)=CC^*_{d,k'}(k)$ for any diameter $k \equiv k' \pmod f$, $k \ge k_d$ for some threshold value $k_d$ dependent on $d$. Moreover 

\[AC^*_{d,k'}(k) =
\begin{dcases}
\ \frac{1}{2}\left ( \frac{4}{f}\right ) ^f k^f+\left ( \frac{4}{f}\right ) ^{f-1}k^{f-1}+O(k^{f-2}) &  \mbox{\quad\ for even } d \\
\ \quad \left ( \frac{4}{f}\right ) ^f k^f \qquad \qquad \qquad \quad +O(k^{f-2}) &  \mbox{\quad\  for odd } d,\\
\end{dcases}
\]

\[CC^*_{d,k'}(k) =
\begin{dcases}
\ \frac{1}{2}\left ( \frac{4}{f}\right ) ^f k^f+\left ( \frac{4}{f}\right ) ^{f-1}k^{f-1}+O(k^{f-2}) & \mbox{\quad\ for even } d \\
\ \quad \left ( \frac{4}{f}\right ) ^f k^f \qquad \qquad \qquad \quad +O(k^{f-2}) & \mbox{\quad\ for odd } d.\\
\end{dcases}
\]

Note that the difference $AC^*_{d,k'}(k)-CC^*_{d,k'}(k)$ is $O(k^{f-2})$ for any $d$ and $k$.
\label {conjecture:ext}
\end{conjecture}

The conjecture is true for dimensions 1 and 2, with $k_d=1$. For dimensions 3 and 4 its conditions are satisfied by the largest-known circulant graphs, with $k_6=1$, $k_7=3$, $k_8=3$ and $k_9=5$, see \cite{Lewis8}, and also for largest-known Abelian Cayley graphs of dimension 3, which are just the largest-known circulant graphs.

For even degree both the leading and second terms are identical to those of the Abelian Cayley graph lower bound $DF_{AC}(d,k)$. Furthermore, the first two coefficients for both even and odd degree are the same multiple $R_f= 2^{f-1}(f!/f^f)$ of the corresponding terms for the Abelian Cayley upper bound $M_{AC}(d,k)$.

If Conjecture \ref{conjecture:ext} is true for dimension 5 then this would give the following formulae for the order of extremal circulant graphs of degree 10 and 11:
\[ CC(d,k) =
\begin {cases}
(512k^5+1280k^4)/3125 + O(k^3) & \mbox { for } d=10 \\
1024k^5/3125 \qquad \qquad \ \, + O(k^3) & \mbox { for } d=11 \\
\end {cases}
\]

As mentioned, this paper presents the construction of infinite families of largest-known circulant graphs of degree 10 and 11. The formulae for the order $L_{CC}(10,k)$ and $L_{CC}(11,k)$ of these graphs are quintic polynomials in the diameter with leading and second terms that match those from Conjecture \ref{conjecture:ext}. They are conjectured to be extremal for all diameters above defined thresholds.


\section {Largest-known circulant graphs of degree 10}

The process followed to discover the largest-known degree 10 and 11 graphs and the quintic polynomials in the diameter that define their orders and generating sets was an extension of the methods used by Dougherty and Faber for the degree 6 and 7 families and by the author for degrees 8 and 9. It involved a combination of four methods: application of the extremal order conjecture (Conjecture \ref{conjecture:ext}), analysis of the largest-known families of smaller degree to discover common features that may extrapolate, computer searches that needed to be increasingly sharply focused as the diameter increased, and a measure of inspiration and  pattern recognition.

The extremal order conjecture implied that the extremal graphs of degree 10 would have order defined by quintic polynomials in the diameter, one for each diameter class $k \pmod 5$, and also specified their common first two coefficients, reducing the degrees of freedom accordingly. As with the approach for degrees 6 to 9, for small diameter the extremality of the graphs was confirmed by conducting a computer search using feasible generating sets for graphs of every order up to the upper bound. However the number of possible permutations of elements for generating sets of dimension 5 increases rapidly with diameter, quickly exceeding available computing power. For degree 10 the graphs could only be proven extremal up to diameter 5, which provided a maximum of only one graph as the basis for each of the five presumed families.

\begin {table} [!b]
\small
\caption{Largest-known circulant graphs of degree 10, up to diameter 16.} 
\centering 
\setlength {\tabcolsep} {4pt}
\begin{tabular}{ @ { } c  r l l l}
\noalign {\vskip 2mm}  
\hline\hline 
\noalign {\vskip 1mm} 
Diameter & Order & Isomorphism & Generating set* & Status, extremality \\ 
$k$ & $n(10,k)$ & family &  & checked to $M_{AC}(10,k)$ \\
\hline 
\noalign {\vskip 1mm} 
2 & 51 & - &1, 2, 10, 16, 23 &  Extremal (checked to 61)\\
3 & 177 & - &1, 12, 19, 27, 87 &  Extremal (checked to 231) \\
4 & 457 & 4 &1, 20, 130, 147, 191 & Extremal (checked to 681) \\
5 & 1099 & 0 &1, 53, 207, 272, 536 & Extremal (checked to 1683) \\
6 & 2380 & 1 &1, 555, 860, 951, 970  &  Largest known \\
   & 2329 & 1(odd) &1, 75, 390, 453, 764  &  Largest-known odd order \\
7 & 4551 & 2 &1, 739, 1178, 1295, 1301 & Largest known  \\
8 & 8288 & 3 &1, 987, 2367, 2534, 3528 & Largest known  \\
   & 8183 & 3(odd)a &1, 286, 294, 1707, 3758 & Largest-known odd order \\
   &          & 3(odd)b &1, 112, 120, 953, 1504 & \\
9 & 14099 & 4 & 1, 247, 1766, 1983, 3494 &  Largest known  \\
10 & 22805 & 0 & 1, 313, 2495, 2846, 5662 &  Largest known  \\
11 & 35568 & 1 & 1, 4347, 7470, 7903, 11808 & Largest known  \\
     & 35243 & 1(odd) &1, 387, 3528, 3877, 7010  &  Largest-known odd order \\
12 & 53025 & 2 & 1, 5251, 19281, 19291, 19806 &  Largest known  \\
13 & 77572 & 3 & 1, 6347, 14103, 14740, 21098  & Largest known  \\
     & 77077 & 3(odd)a &1, 1594, 21165, 36774, 36784 & Largest-known odd order \\
     &            & 3(odd)b &1, 4344, 29303, 38093, 38103 & \\
14 & 110045 & 4 & 1, 827, 9176, 9935, 18272 &  Largest known  \\
15 & 152671 & 0 & 1, 973, 11663, 12716, 25364 &  Largest known  \\
16 & 208052 & 1 & 1, 17147, 30784, 32007, 47918  & Largest known  \\
     & 207037 & 1(odd) &1, 1131, 14794, 15845, 29496  &  Largest-known odd order \\
\hline
\noalign {\vskip 1mm} 
&& \multicolumn {3} {l} {\footnotesize * for each isomorphism class of graphs just one of the generating sets is listed} 
\end{tabular}
\label{table:10A}
\end{table}

From an analysis of the largest-known families of smaller degree, common factors were discovered that were tentatively assumed to remain valid, greatly reducing the search space for the computer runs. For example, as mentioned earlier, every graph in a largest-known circulant graph family up to degree 9 has a primitive generating set, so that the computer searches were set to fix one of the generators at 1, eliminating a degree of freedom. Although all Abelian Cayley graphs have girth 3 or 4 by definition, all the largest-known families have an odd girth that is maximal ($2k+1$), and so the computer searches were restricted to maximal odd girth, which reduced the run-time significantly. Also for each degree $d \le 9$ the order polynomials for the families for diameter $k \equiv 0 \pmod f$ have value 0 or 1 for $k=0$ depending on the parity of the order of the family, again reducing the degree of freedom. It was also observed that generating sets of largest-known families often include pairs of generators differing by a small value that increases linearly with diameter. Where this was found to occur for a single graph in a family, the subsequent search for other graphs in the prospective family was restricted to include such pairs, further reducing the degree of freedom.

Even utilising these and other similar techniques, the process remained complex and time-consuming. Each newly-discovered graph became a potential member of its family, restricting the freedom of the corresponding order polynomial and further sharpening the search for the next graph in the family. Any subsequent search failure would require backtracking to eliminate a candidate graph and initiate a search for the next candadiate. However the challenge eventually yielded to the effort as each graph family in turn was completed, for degree 10 and similarly for degree 11. The largest-known degree 10 circulant graphs up to diameter 16, as discovered by the author, are shown in Table \ref {table:10A}.

For each diameter there is a single isomorphism class of largest-known graphs, each with at least one primitive generating set. For diameter $k \ge 4$ it emerges that these isomorphism classes are partitioned into five families according to the value of $k \pmod 5$. Note that for diameter $k \equiv 1$ or $3 \pmod5$, when $k\ge 6$ the largest-known graphs have even order.  For completeness the largest-known degree 10 circulant graphs of odd order $L_{OC}(d,k)$ and diameter $k \equiv 1$ and $k \equiv 3 \pmod 5$ are also included in the table, as classes 1(odd), 3(odd)a and 3(odd)b.

The following quintic polynomials in $k$ determine the order of these families of largest-known graphs for arbitrary diameter $k \geq 4$. $L_{CC}(10,k)=$
\[
\begin{cases}
\ (512k^5+1280k^4+6400k^3+8000k^2+6250k+3125)/3125 &\mbox{ for } k \equiv 0 \pmod 5 \\
\ (512k^5+1280k^4+6560k^3+9520k^2+6100k+1028)/3125 &\mbox{ for } k \equiv 1 \pmod 5 \\
\ (512k^5+1280k^4+6080k^3+7840k^2+10010k+3741)/3125 &\mbox{ for } k \equiv 2 \pmod 5 \\
\ (512k^5+1280k^4+6560k^3+7600k^2+4180k+1344)/3125 &\mbox{ for } k \equiv 3 \pmod 5 \\
\ (512k^5+1280k^4+6400k^3+8640k^2+6890k+757)/3125 &\mbox{ for } k \equiv 4 \pmod 5 \\
\end{cases}
\]

The families of largest-known graphs of odd order for diameter $k \equiv 1$ and $k \equiv 3 \pmod 5$ have order given by the following quintics. $L_{OC}(10,k)=$
\[
\begin{cases}
\ (512k^5+1280k^4+5760k^3+9920k^2+6450k-2047)/3125 &\mbox{ for } k \equiv 1 \pmod 5 \\
\ (512k^5+1280k^4+5760k^3+8800k^2+4830k+819)/3125 &\mbox{ for } k \equiv 3 \pmod 5 \\
\end{cases}
\]

\begin {table} [!htb]
\caption{Order of largest-known circulant graphs of degree 10 for any diameter $k\ge4$.}
\centering 
\setlength {\tabcolsep} {5pt}
\begin{tabular}{ @ { } c c r r r r r r r r r l l c} 
\noalign {\vskip 2mm} 
\hline\hline 
\noalign {\vskip 1mm} 
Diameter & & \multicolumn {6} {l} {Order $L_{CC}(10,k)$} & & & & & & Isomorphism \\ 
$k \pmod 5$ & &\multicolumn {6} {l} {in $a$-format*} & & & & & where & family \\
\hline 
\noalign {\vskip 1mm} 
0 &&  ( 1 & 2 & 8 & 8 & 5 & 2 )&/2  & & & & $a=4k/5$ & 0 \\
1 &&  ( 1 & 1 & 7 & 5 & 2 & 0 )&/2 & & & & $a=(4k+1)/5$ & l \\
2 &&  ( 1 & 0 & 6 & 0 & 5 & 0 )&/2& & & & $a=(4k+2)/5$ & 2 \\
3 &&  ( 1 & -1 & 7 & -5 & 2 & 0 )&/2 & & & & $a=(4k+3)/5$ & 3 \\
4 &&  ( 1 & -2 & 8 & -8 & 5 & -2 )&/2 & & & & $a=(4k+4)/5$ & 4 \\

\hline
\noalign {\vskip 1mm} 
  & & \multicolumn {6} {l} {Order $L_{OC}(10,k)$} & & & & &  \\ 
\hline 
\noalign {\vskip 1mm} 
1 &&  ( 1 & 1 & 6 & 6 & 2 & -2 )&/2 & & & & $a=(4k+1)/5$ & 1(odd) \\
3 &&  ( 1 & -1 & 6 & -2 & 0 & 0 )&/2 & & & & $a=(4k+3)/5$ & 3(odd)a \& b \\

\hline
\noalign {\vskip 1mm} 
&&\multicolumn {12} {l} {\small * $( c_5\:  c_4\:  c_3\:  c_2\:  c_1\:  c_0)/b=(c_5a^5+c_4a^4+c_3a^3+c_2a^2+c_1a+c_0)/b$} 
\end{tabular}
\label{table:10B}
\end{table}

The graphs of all these families have orders that share common leading and second coefficients, $512/3125$ and $1280/3125$, consistent with Conjecture \ref{conjecture:ext}. These expressions may be greatly simplified by replacing the diameter $k$ with a parameter $a=(4k+c)/5$ where the value of $c$ depends on $k \pmod 5$ and is chosen to ensure $a$ remains integral, and presented in so-called \emph{$a$-format}, see Table \ref {table:10B}.

These formulae define the order $L_{CC}(10,k)$ of the largest-known circulant graphs of degree 10 for any diameter $k \ge 4$. For diameter $k\leq 3$, the graphs with order defined by these formulae are not extremal. For $k=2$ the formula gives a graph of order 45 whereas the extremal order is 51, and for $k=3$ the formula gives 156 instead of 177. The existence of these graph families  for all $k  \ge 4$ is proved in Section \ref{section:proof10}. They are the largest degree 10 circulant graphs discovered for any $k \geq 4$ and are conjectured to be extremal.


As mentioned, there is one family of graphs for each diameter class $k \pmod5$, so five in total. Within each family there is one isomorphism class of graphs with largest-known order for each diameter. In general for any circulant graph of order $n$ the maximum number of primitive generating sets for graphs in the same isomorphism class is equal to the dimension $f$ of the graph. This is because for each generator in the set there is a factor in $\Z_n$ such that the product equals 1 if and only if the generator is coprime with $n$. For example, consider the primitive generating sets for diameters $k=5$ and $k=10$, enumerated in Table \ref {table:100D}. For diameter $k=5$ there are five primitive generating sets for the isomorphism class, but only four for diameter $k=10$ because the generator 2495 in Set 1 is not coprime with the order 22805.

\begin {table} [!htbp]
\small
\caption{Example: enumeration of primitive generating sets for degree 10, class 0, for diameters $k=5$ and $k=10$.} 
\centering 
\setlength {\tabcolsep} {6pt}
\begin{tabular}{ @ { } c r r r r r r r r r r r r r } 
\noalign {\vskip 2mm} 
\hline\hline 
\noalign {\vskip 1mm}
 & &  \multicolumn {5} {l} {Diameter $k=5$}  & & & \multicolumn {5} {l} {Diameter $k=10$} \\
 & &  \multicolumn {5} {l} {Order $L_{CC}(10,5)=1099$}  & & & \multicolumn {5} {l} {Order $L_{CC}(10,10)=22805$} \\
\noalign {\vskip 1mm} 
\hline
\noalign {\vskip 1mm}
Generating & & \multicolumn {5} {l} {Generators} & & & \multicolumn {5} {l} {Generators} \\
set & & $g_1$ & $g_2$ & $g_3$ & $g_4$ & $g_5$ & & & $g_1$ & $g_2$ & $g_3$ & $g_4$ & $g_5$ \\
\hline 
\noalign {\vskip 1mm} 
1 & & 1 & 53 & 207 & 272 & 536  & & & 1 & 313 & 2495 & 2846 & 5662 \\
2 & & 1 & 351 & 281 & 285 & 509 & & & 1 & 1970 & 6819 & 6827 & 218 \\
3 & & 1 & 95 & 319 & 400 & 375 & & & 1 & 3775 & 6172 & 7396 & 11162 \\
4 & & 1 & 90 & 454 & 296 & 292 & & & - \\
5 & & 1 & 232 & 176 & 534 & 394 & & & 1 & 4663 & 1612 & 3635 & 6284 \\
\hline
\end{tabular}
\label{table:100D}
\end{table}
\begin {table} [!htb]
\small
\caption{Generating sets 1 to 3 and one example for set 5 for degree 10 graphs of class 0: diameter $k\equiv 0 \pmod 5$.} 
\centering 
\setlength {\tabcolsep} {5pt}
\begin{tabular}{ @ { } c l c c c c r l l l l c c c c r l} 
\noalign {\vskip 2mm} 
\hline\hline 
\noalign {\vskip 1mm} 
Generator & \multicolumn {8} {l} {Gen set 1, $k\equiv 0 \pmod {5}$} && \multicolumn {7} {l} {Gen set 2, $k\equiv 0 \pmod {5}$} \\ 
\hline 
\noalign {\vskip 1mm} 
$g_1$ &   ( 0 & 0 & 0 & 0 & 0 & 1 ) & & & & ( 0 & 0 & 0 & 0 & 0 & 1 ) \\
$g_2$ &  ( 0 & 0 & 1 & 1 & 6 & 2 ) & /2 & & & ( 0 & 4 & 5 & 28 & 11 & 11 ) \\
$g_3$ &  ( 0 &  1 & 1 & 6 & 0 & -2 ) & /2 & & & ( 0 & 10 & 13 & 71 & 30 & 29 ) \\
$g_4$ &  ( 0 & 1 & 2 & 8 & 7 & 4 )& /2 & & & ( 0 & 10 & 13 & 71 & 31 & 29 ) \\
$g_5$ &  ( 0 & 2 & 4 & 15 & 15 & 4 ) & /2 & & & ( 0 & 13 & 17 & 92 & 40 & 37 ) \\
\hline
\noalign {\vskip 1mm} 
 & \multicolumn {8} {l} {Gen set 3, $k\equiv 0 \pmod {15}$} && \multicolumn {7} {l} {Gen set 3, $k\equiv 5 \pmod {15}$} \\ 
\hline 
\noalign {\vskip 1mm} 
$g_1$ &   ( 0 & 0 & 0 & 0 & 0 & 1 ) & & & & ( 0 & 0 & 0 & 0 & 0 & 1 ) \\
$g_2$ &  ( 1 & -2 & 2 & -20 & -9 & 0 ) & /6 & & & ( 0 & 2 & 0 & 8 &-14 & -14 ) & /6 \\
$g_3$ &  ( 1 & 2 & 6 & 6 & -5 & 0 ) & /6 & & & ( 0 & 4 & 6 & 28 & 14 & 2 ) & /6 \\
$g_4$ &  ( 1 & 4 & 8 & 16 & -9 & -12 ) & /6 & & & ( 1 & 2 & 10 & 10 & 15 & 4 ) & /6 \\
$g_5$ &  ( 1 & 6 & 16 & 38 & 35 & 12 ) & /6 & & & ( 1 & 6 & 16 & 38 & 35 & 12 ) & /6 \\
\hline
\noalign {\vskip 1mm} 
 & \multicolumn {8} {l} {Gen set 3, $k\equiv 10 \pmod {15}$} && \multicolumn {7} {l} {Gen set 5, $k\equiv 0 \pmod {85}$} \\ 
\hline 
\noalign {\vskip 1mm} 
$g_1$ &   ( 0 & 0 & 0 & 0 & 0 & 1 ) & & & &  ( 0 & 0 & 0 & 0 & 0 & 1 ) \\
$g_2$ &  ( 0 & 4 & 8 & 30 & 30 & 10 ) & /6 & & &  ( 2 & -2 & 5 & -33 & -27 & -17 ) &/17  \\
$g_3$ &  ( 1 & 0 & 8 & 0 & 19 & 16 ) & /6 & & & ( 3 & 14 & 33 & 78 & 36 & 17 ) &/17  \\
$g_4$ &  ( 1 & 2 & 6 & 6 & -5 & 0 ) & /6  & & & ( 7 & 10 & 60 & 46 & 101 & 68 )& /34  \\
$g_5$ &  ( 1 & 6 & 14 & 36 & 19 & 4 ) & /6 & & & ( 4 & 13 & 44 & 70 & 65 & 17 ) & /17  \\
\hline
\noalign {\vskip 1mm} 
& \multicolumn {16} {l} {\footnotesize $( c_5\:  c_4\:  c_3\:  c_2\:  c_1\:  c_0)/b=(c_5a^5+c_4a^4+c_3a^3+c_2a^2+c_1a+c_0)/b$ where $a=4k/5$} 
\end{tabular}
\label{table:100A1}
\end{table}

For diameter $k \equiv 0 \pmod 5$, isomorphism class 0 has between 3 and 5 primitive generating sets, depending on the value of $k$. Generating sets 1 and 2 are defined by formulae valid for every $k\equiv 0 \pmod 5$, and set 3 is comprised of three subsets of formulae valid for $k\equiv 0, 5$ and $10 \pmod {15}$, see Table \ref {table:100A1}. Set 4 has four subsets valid for $k \equiv 0, 5, 15$ and $20 \pmod {25}$, omitting $k \equiv10 \pmod {25}$, see Table \ref {table:100A2}. Set 5 has 16 subsets valid for each $k \pmod {85}$ with the exception of $k\equiv 60 \pmod {85}$. An example of set 5 for $k \equiv0 \pmod {85}$ is included in Table \ref {table:100A1}, with the complete listing given in Tables \ref {table:100B} and \ref {table:100C} in the appendix.

\begin {table} [!htb]
\small
\caption{Generating set 4 for degree 10 graphs of class 0: diameter $k\equiv 0 \pmod 5$.} 
\centering 
\setlength {\tabcolsep} {5pt}
\begin{tabular}{ @ { } c l c c c c r l l l l c c c c r l} 
\noalign {\vskip 2mm} 
\hline\hline 
\noalign {\vskip 1mm} 
Generator & \multicolumn {8} {l} {Gen set 4, $k\equiv 0 \pmod {25}$} && \multicolumn {7} {l} {Gen set 4, $k\equiv 5 \pmod {25}$} \\ 
\hline 
\noalign {\vskip 1mm} 
$g_1$ &   ( 0 & 0 & 0 & 0 & 0 & 1 ) & & & & ( 0 & 0 & 0 & 0 & 0 & 1 ) \\
$g_2$ &  ( 1 & 10 & 18 & 62 & 21 & 20 ) & /10 & & & ( 1 & -4 & -2 & -40 & -27 & -24 ) & /10 \\
$g_3$ &  ( 1 & 10 & 18 & 62 & 31 & 20 ) & /10 & & & ( 1 & 6 & 18 & 40 & 43 & 16 ) & /10 \\
$g_4$ &  ( 1 & 0 & 3 & -8 & -14 & -5 ) & /5 & & & ( 1 & 6 & 13 & 35 & 13 & 11 ) & /5 \\
$g_5$ &  ( 1 & 5 & 13 & 32 & 21 & 15 ) & /5 & & & ( 1 & 6 & 13 & 35 & 18 & 11 ) & /5 \\
\hline
\noalign {\vskip 1mm} 
 & \multicolumn {8} {l} {Gen set 4, $k\equiv 15 \pmod {25}$} && \multicolumn {7} {l} {Gen set 4, $k\equiv 20 \pmod {25}$} \\ 
\hline 
\noalign {\vskip 1mm} 
$g_1$ &   ( 0 & 0 & 0 & 0 & 0 & 1 ) & & & & ( 0 & 0 & 0 & 0 & 0 & 1 ) \\
$g_2$ &  ( 1 & -2 & -2 & -24 & -33 & -12 ) & /10 & & & ( 1 & -6 & -2 & -46 & -21 & -16 ) & /10 \\
$g_3$ &  ( 1 & 8 & 18 & 56 & 37 & 28 ) & /10 & & & ( 1 & -6 & -2 & -46 & -11 & -16 ) & /10 \\
$g_4$ &  ( 1 & -2 & 3 & -19 & -8 & -7 ) & /5& & & (  1 & -1 & 3 & -16 & -11 & -11 ) & /5 \\
$g_5$ &  ( 1 & -2 & 3 & -19 & -3 & -7 ) & /5 & & & ( 1 & 4 & 13 & 24 & 24 & 9 ) & /5 \\
\hline
\noalign {\vskip 1mm} 
& \multicolumn {16} {l} {\footnotesize $( c_5\:  c_4\:  c_3\:  c_2\:  c_1\:  c_0)/b=(c_5a^5+c_4a^4+c_3a^3+c_2a^2+c_1a+c_0)/b$ where $a=4k/5$} 
\end{tabular}
\label{table:100A2}
\end{table}


\newpage

For diameter $k \equiv 1 \pmod 5$, isomorphism class 1 has just one primitive generating set, as shown in Table \ref {table:101A}, defined by formulae valid for every $k\equiv 1 \pmod 5$.

\begin {table} [!htb]
\small
\caption{Generating set 1 for degree 10 graphs of class 1: diameter $k\equiv 1 \pmod 5$.} 
\centering 
\setlength {\tabcolsep} {6pt}
\begin{tabular}{ @ { } c l c c c c r r l l c c c c r l} 
\noalign {\vskip 2mm} 
\hline\hline 
\noalign {\vskip 1mm} 
Generator && \multicolumn {8} {l} {Gen set 1, $k\equiv 1 \pmod {5}$}  \\ 
\hline 
\noalign {\vskip 1mm} 
$g_1$ &&   ( 0 & 0 & 0 & 0 & 0 & 1 )  \\
$g_2$ &&  ( 0 & 1 & 2 & 7 & 12 & 0 ) & /2 \\
$g_3$ &&  ( 0 & 2 & 1 & 13 & 4 & 0 ) & /2  \\
$g_4$ &&  ( 0 & 1 & 1 & 7 & 5 & 1 ) &   \\
$g_5$ &&  ( 0 & 3 & 3 & 20 & 14 & 0 ) & /2  \\
\hline 
\noalign {\vskip 1mm} 
\multicolumn {16} {l} {\footnotesize $( c_5\:  c_4\:  c_3\:  c_2\:  c_1\:  c_0)/b=(c_5a^5+c_4a^4+c_3a^3+c_2a^2+c_1a+c_0)/b$ where $a=(4k+1)/5$}  
\end{tabular}
\label{table:101A}
\end{table}


For diameter $k \equiv 2 \pmod 5$, isomorphism class 2 has 1 or 2 primitive generating sets, depending on the value of $k$, as shown in Table \ref {table:102A}. Generating set 1 is comprised of two subsets of formulae valid for $k\equiv 2$ and $7 \pmod {15}$ omitting $k \equiv12 \pmod {15}$. Set 2 also has two subsets valid for $k \equiv 7$ and $12 \pmod {15}$, omitting $k \equiv 2 \pmod {15}$.

\begin {table} [!htb]
\small
\caption{Generating sets for degree 10 graphs of class 2: diameter $k\equiv 2 \pmod 5$.} 
\centering 
\setlength {\tabcolsep} {6pt}
\begin{tabular}{ @ { } c l c c c c r l l l l c c c r r l } 
\noalign {\vskip 2mm} 
\hline\hline 
\noalign {\vskip 1mm} 
Generator & \multicolumn {8} {l} {Gen set 1, $k\equiv 2 \pmod {15}$} && \multicolumn {7} {l} {Gen set 1, $k\equiv 7 \pmod {15}$} \\ 
\hline 
\noalign {\vskip 1mm} 
$g_1$ &   ( 0 & 0 & 0 & 0 & 0 & 1 ) & & & & ( 0 & 0 & 0 & 0 & 0 & 1 ) \\
$g_2$ &  ( 0 & 1 & 0 & 5 & 0 & 2 ) & /2 & & & (  0 & 1 & 0 & 5 & 0 & 2 ) & /2 \\
$g_3$ &  ( 1 & -1 & 5 & -7 & -2 & -6 ) & /6 & & &      ( 1 & -1 & 4 & -7 & -3 & -6 ) & /6 \\
$g_4$ &  ( 1 & -1 & 5 & -7 & 4 & -6 ) & /6 & & &       ( 1 & -1 & 7 & -7 & 6 & -6  ) & /6 \\
$g_5$ &  ( 1 & -1 & 8 & -7 & 13 & -6 ) & /6 & & & ( 1 & -1 & 7 & -7 & 12 & -6  ) & /6 \\
\hline
\noalign {\vskip 1mm} 
 & \multicolumn {8} {l} {Gen set 2, $k\equiv 7 \pmod {15}$} && \multicolumn {7} {l} {Gen set 2, $k\equiv 12 \pmod {15}$} \\ 
\hline 
\noalign {\vskip 1mm} 
$g_1$ &   ( 0 & 0 & 0 & 0 & 0 & 1 ) & & & & ( 0 & 0 & 0 & 0 & 0 & 1 ) \\
$g_2$ &  ( 0 & 1 & 0 & 5 & 0 & 2 ) & /2 & & & (  0 & 1 & 0 & 5 & 0 & 2 ) & /2 \\
$g_3$ &  ( 1 & 1 & 4 & 7 & -3 & 6 ) & /6 & & &    ( 1 & 1 & 5 & 7 & -2 & 6 ) & /6 \\
$g_4$ &  ( 1 & 1 & 7 & 7 & 6 & 6  ) & /6 & & &     ( 1 & 1 & 5 & 7 & 4 & 6 ) & /6 \\
$g_5$ &  ( 1 & 1 & 7 & 7 & 12 & 6 ) & /6 & & & ( 1 & 1 & 8 & 7 & 13 & 6 ) & /6 \\
\hline
\noalign {\vskip 1mm} 
& \multicolumn {16} {l} {\footnotesize $( c_5\:  c_4\:  c_3\:  c_2\:  c_1\:  c_0)/b=(c_5a^5+c_4a^4+c_3a^3+c_2a^2+c_1a+c_0)/b$ where $a=(4k+2)/5$} 
\end{tabular}
\label{table:102A}
\end{table}


\newpage

\begin {table} [!htbp]
\small
\caption{Generating set 1 for degree 10 graphs of class 3: diameter $k\equiv 3 \pmod 5$.} 
\centering 
\setlength {\tabcolsep} {6pt}
\begin{tabular}{ @ { } c l l c c c c r l l l l c c c c r l} 
\noalign {\vskip 2mm} 
\hline\hline 
\noalign {\vskip 1mm} 
Generator && \multicolumn {8} {l} {Gen set 1, $k\equiv 3 \pmod {5}$}  \\ 
\hline 
\noalign {\vskip 1mm} 
$g_1$ &&   ( 0 & 0 & 0 & 0 & 0 & 1 )  \\
$g_2$ &&  ( 0 & 1 & -2 & 7 & -12 & 0 ) & /2 \\
$g_3$ &&  ( 0 & 1  & -1 & 7 & -5 & 1 )  \\
$g_4$ &&  ( 0 & 2 & -1 & 13 & -4 & 0 ) & /2  \\
$g_5$ &&  ( 0 & 3 & -3 & 20 & -14 & 0 ) & /2  \\
\hline
\noalign {\vskip 1mm} 
\multicolumn {16} {l} {\footnotesize $( c_5\:  c_4\:  c_3\:  c_2\:  c_1\:  c_0)/b=(c_5a^5+c_4a^4+c_3a^3+c_2a^2+c_1a+c_0)/b$ where $a=(4k+3)/5$} 
\end{tabular}
\label{table:103A}
\end{table}

For diameter $k \equiv 3 \pmod 5$, isomorphism class 3 has just one primitive generating set, as shown in Table \ref {table:103A}, defined by formulae valid for every $k\equiv 3 \pmod 5$.


For diameter $k \equiv 4 \pmod 5$, isomorphism class 4 has between 3 and 5 primitive generating sets, depending on the value of $k$.  This replicates exactly the pattern of generating sets for isomorphism class 0, and the sets are all tabulated in the appendix. Generating sets 1 and 2 are defined by formulae valid for every $k\equiv 4 \pmod 5$, and set 3 is comprised of three subsets of formulae valid for $k\equiv 4, 9$ and $14 \pmod {15}$, see Table \ref{table:104A}.
Set 4 has four subsets valid for $k \equiv 4, 9, 19$ and $24 \pmod {25}$, omitting $k \equiv14 \pmod {25}$, see Table \ref{table:104B}. Set 5 has 16 subsets valid for each $k \pmod {85}$ with the exception of $k\equiv 24 \pmod {85}$, see Tables \ref {table:104C} and \ref {table:104D}.


\section {Largest known circulant graphs of degree 11}

The same process was followed to discover the largest-known degree 11 graphs and the quintic polynomials in the diameter that define their orders and generating sets as for degree 10. The largest-known degree 11 circulant graphs up to diameter 16 are shown in Table \ref {table:5E}. For diameter $k=2$ there are five distinct isomorphism classes, one of which does not have a primitive generating set. For $k=3$ there are two classes, including one with no primitive generating set. For higher diameters there is a single isomorphism class, with a primitive generating set, except for diameters $k =6,11$ and $16$, where there are two isomorphism classes.

\begin {table} [!htb]
\small
\caption{Largest-known circulant graphs of degree 11, up to diameter 16.} 
\centering 
\setlength {\tabcolsep} {3pt}
\begin{tabular}{ @ { } c  r l l l} 
\noalign {\vskip 2mm} 
\hline\hline 
\noalign {\vskip 1mm} 
Diameter & Order & Isomorphism & Generating set* & Status, extremality \\ 
$k$ & $n(11,k)$ & family & & checked to $M_{AC}(11,k)$ \\
\hline 
\noalign {\vskip 1mm} 
2 & 56 & - &1, 2, 10, 15, 22  & Extremal (checked to 72) \\
 &  & - &1, 4, 6, 15, 24  &  \\
 &  & - &1, 6, 10, 15, 18  &   \\
 &  & - &1, 9, 14, 21, 25  &   \\
 &  & - &2, 6, 7, 18, 21 & \\
3 & 210 & - &1, 49, 59, 84, 89  & Extremal (checked to 292) \\
  &     & - &2, 32, 63, 92, 98 & \\
4 & 576 & - &1, 9, 75, 155, 179 &  Largest known \\
5 & 1428 & 0 &1, 169, 285, 289, 387 & Largest known \\
6 & 3200 & 1a & 1, 101, 925, 1031, 1429 &  Largest known \\
   &          & 1b & 1, 265, 851, 1111, 1321 &  \\
7 & 6652 & 2 & 1, 107, 647, 2235, 2769 &  Largest known \\
8 & 12416 & 3 & 1, 145, 863, 4163, 5177 &  Largest known  \\
9 & 21572 & 4 & 1, 189, 1517, 8113, 9435 &  Largest known  \\
10 & 35880 & 0 & 1, 2209, 5127, 5135, 12537  & Largest known  \\
11 & 56700 & 1a & 1, 1053, 1061, 10603, 17965  & Largest known  \\
     &            & 1b & 1, 4113, 4121, 13301, 23723  &  \\
12 & 87248 & 2 & 1, 479, 4799, 34947, 39257 &  Largest known  \\
13 & 128852 & 3 & 1, 581, 5799, 51599, 57989  & Largest known  \\
14 & 184424 & 4 & 1, 693, 8325, 76901, 84523 &  Largest known  \\
15 & 259260 & 0 & 1, 10729, 39875, 39887, 90637 &  Largest known  \\
16 & 355576 & 1a & 1, 22307, 131327, 136371, 153621  & Largest known  \\
     &              & 1b & 1, 8579, 75569, 75583, 111513   \\

\hline
\noalign {\vskip 1mm} 
\multicolumn {5} {l} {\footnotesize * excludes the involution; for each isomorphism class of graphs just one of the generating sets is listed} 
\end{tabular}
\label{table:5E}
\end{table}

For diameter $k$ from 5 to 16, these graphs may be partitioned into five families of isomorphism classes according to the value of $k \pmod 5$. The following quintic polynomials in $k$ determine the order of these families of graphs. $L_{CC}(11,k)=$
\[
\begin{cases}
\ (1024k^5+9600k^3+12500k)/3125 &\mbox{ for } k \equiv 0 \pmod 5 \\
\ (1024k^5+8960k^3+2880k^2-260k-104)/3125 &\mbox{ for } k \equiv 1 \pmod 5 \\
\ (1024k^5+10240k^3+640k^2+5140k-2528)/3125 &\mbox{ for } k \equiv 2 \pmod 5 \\
\ (1024k^5+10240k^3-640k^2+5140k+2528)/3125 &\mbox{ for } k \equiv 3 \pmod 5 \\
\ (1024k^5+8960k^3+5120k^2+740k-6896)/3125 &\mbox{ for } k \equiv 4 \pmod 5 \\
\end{cases}
\]

Graphs with these orders have been constructed for all diameters up to $k=100$. They are the largest degree 11 circulant graphs discovered for any $k \geq 5$ and are conjectured to be extremal. Above diameter $k=100$, the existence of graphs with these constructions has not been confirmed. However they are conjectured to exist and be extremal for all higher diameters. The graphs of all these families have orders that share common leading and second coefficients, $1024/3125$ and $0$, consistent with Conjecture \ref{conjecture:ext}. As for degree 10, these expressions may be simplified by replacing the diameter $k$ with a variable $a$ that depends on the value of $k \pmod 5$, see Table \ref {table:11A}.

\begin {table} [!htbp]
\caption{Order of largest known circulant graphs of degree 11 for any diameter $k\ge5$.} 
\centering 
\setlength {\tabcolsep} {6pt}
\begin{tabular}{ @ { } c l r r r r r r r l c} 
\noalign {\vskip 2mm} 
\hline\hline 
\noalign {\vskip 1mm} 
Diameter & \multicolumn {6} {l} {Order $L_{CC}(11,k)$} & & & & Isomorphism \\ 
$k \pmod 5$ &\multicolumn {6} {l} {in bracket notation*} & & &  where & classes \\
\hline 
\noalign {\vskip 1mm} 
0 &  (1 & 0 & 6 & 0 & 5 & 0) & & & $a=4k/5$ & 0 \\
1 &  (1 & -1 & 6 & -2 & 0 & 0) & & & $a=(4k+1)/5$ & 1a \\
2 &  (1 & -2 & 8 & -8 & 5 & -2) & & & $a=(4k+2)/5$ & 2 \\
3 &  (1 & 2 & 8 & 8 & 5 & 2) & & & $a=(4k-2)/5$ & 3 \\
4 &  (1 & 1 & 6 & 6 & 2 & -2) & & & $a=(4k-1)/5$ & 4 \\

\hline
\noalign {\vskip 1mm} 
&\multicolumn {10} {l} {\footnotesize * $( c_5\:  c_4\:  c_3\:  c_2\:  c_1\:  c_0)=c_5a^5+c_4a^4+c_3a^3+c_2a^2+c_1a+c_0$} 
\end{tabular}
\label{table:11A}
\end{table}

For diameter $k\leq 4$, the graphs with order defined by these formulae are not extremal. For $k=2, 3, 4$ respectively, the formulae give graphs of order 40, 172 and 544 whereas the extremal orders are 56, 210 and 576. For diameter $k\ge 5$ there is one isomorphism class of families of graphs of order $L_{CC}(11,k)$ for each $k \pmod5$ except for $k\equiv 1\pmod 5$ which has two isomorphism classes, labelled 1a and 1b in Table \ref{table:5E}.


For diameter $k \equiv 0 \pmod 5$, isomorphism class 0 has 1 or 2 primitive generating sets, depending on the value of $k$, as shown in Table \ref {table:110A}. Generating set 1 is comprised of two subsets of formulae valid for $k\equiv 0$ and $10 \pmod {15}$ omitting $k \equiv5 \pmod {15}$. Set 2 also has two subsets valid for $k \equiv 0$ and $5 \pmod {15}$, omitting $k \equiv 10 \pmod {15}$.

\begin {table} [!htbp]
\small
\caption{Generating sets for degree 11 graphs of class 0: diameter $k\equiv 0 \pmod 5$.} 
\centering 
\setlength {\tabcolsep} {6pt}
\begin{tabular}{ @ { } c l r r r r r r l l l  r r r r r l} 
\noalign {\vskip 2mm} 
\hline\hline 
\noalign {\vskip 1mm} 
Generator & \multicolumn {8} {l} {Gen set 1, $k\equiv 0 \pmod {15}$} && \multicolumn {7} {l} {Gen set 1, $k\equiv 10 \pmod {15}$} \\ 
\hline 
\noalign {\vskip 1mm} 
$g_1$ &   ( 0 & 0 & 0 & 0 & 0 & 1 ) & & & & ( 0 & 0 & 0 & 0 & 0 & 1 ) \\
$g_2$ & ( 0 & 1 & 0 & 5 & 0 & 2 ) & /2 & & & (  0 & 1 & 0 & 5 & 0 & 2 ) & /2 \\
$g_3$ & ( 1 & -1 & 7 & -7 & 6 & -6 ) & /6 & & & ( 1 & -1 & 5 & -7 & -2 & -6 ) & /6  \\
$g_4$ & ( 1 & -1 & 7 & -7 & 12 & -6 ) & /6 & & &  ( 1 & -1 & 5 & -7 & 4 & -6 ) & /6  \\
$g_5$ & ( 2 & 1 & 14 & 7 & 18 & 6 ) & /6 & & &  ( 2 & 1 & 10 & 7 & 2 & 6 ) & /6 \\
\hline
\noalign {\vskip 1mm} 
 & \multicolumn {8} {l} {Gen set 2, $k\equiv 0 \pmod {15}$} && \multicolumn {7} {l} {Gen set 2, $k\equiv 5 \pmod {15}$} \\ 
\hline 
\noalign {\vskip 1mm} 
$g_1$ &   ( 0 & 0 & 0 & 0 & 0 & 1 ) & & & & ( 0 & 0 & 0 & 0 & 0 & 1 ) \\
$g_2$ &  ( 0 & 1 & 0 & 5 & 0 & 2 ) & /2 & & & (  0 & 1 & 0 & 5 & 0 & 2 ) & /2 \\
$g_3$ &  ( 1 & 1 & 7 & 7 & 6 & 6 ) & /6 & & & ( 1 & 1 & 5 & 7 & -2 & 6 ) & /6 \\
$g_4$ &  ( 1 & 1 & 7 & 7 & 12 & 6 ) & /6 & & & ( 1 & 1 & 5 & 7 & 4 & 6 ) & /6 \\
$g_5$ &  ( 2 & -1 & 14 & -7 & 18 & -6 ) & /6 & & & ( 2 & -1 & 10 & -7 & 2 & -6 ) & /6 \\
\hline
\noalign {\vskip 1mm} 
& \multicolumn {16} {l} {\footnotesize $( c_5\:  c_4\:  c_3\:  c_2\:  c_1\:  c_0)/b=(c_5a^5+c_4a^4+c_3a^3+c_2a^2+c_1a+c_0)/b$ where $a=4k/5$} 
\end{tabular}
\label{table:110A}
\end{table}


For diameters $k \equiv 1$ to $4 \pmod5$ the tables of generating sets are found in the appendix.

For diameter $k \equiv 1 \pmod 5$, isomorphism class 1a has 1 to 4 primitive generating sets, depending on the value of $k$. Generating set 1 is comprised of five subsets of formulae valid for $k\equiv 1, 6, 11, 16$ and $21 \pmod {25}$, and generating set 2 has four subsets for $k\equiv 1, 6, 16$ and $21 \pmod {25}$, omitting $k\equiv 11 \pmod {25}$, see Table \ref {table:111aA}.   The formulae for generating sets 3 and 4 repeat with diameter increments of 65 and 85 respectively. Set 3 has 12 subsets valid for each $k \pmod {65}$ with the exception of $k\equiv 26 \pmod {65}$, and set 4 has 16 subsets valid for each $k \pmod {85}$ with the exception of $k\equiv 46 \pmod {85}$. For sets 3 and 4 just one example of each is given in Table \ref {table:111aB}.

Also for diameter $k \equiv 1 \pmod 5$, isomorphism class 1b has 2 to 4 primitive generating sets, depending on the value of $k$. Generating sets 1 and 2 are valid for every $k\equiv 1 \pmod 5$, see Table \ref {table:111bA}. Generating set 3 has six subsets valid for $k\equiv 1, 6, 16, 21, 26$ and $31 \pmod {35}$, omitting $k\equiv 11 \pmod {35}$, see Table \ref {table:111bB}.   The formulae for generating set 4 repeat with diameter increments of 65. There are 12 subsets valid for each $k \pmod {65}$ with the exception of $k\equiv 46 \pmod {65}$.


For diameter $k \equiv 2 \pmod 5$, isomorphism class 2 has between 3 and 5 primitive generating sets, depending on the value of $k$. Generating sets 1 and 2 are defined by formulae valid for every $k\equiv 2 \pmod 5$ and set 3 is comprised of three subsets of formulae valid for $k\equiv 2, 7$ and $12 \pmod {15}$, see  Table \ref {table:112A1}. Set 4 has four subsets valid for $k \equiv 7, 12, 17$ and $22 \pmod {25}$, omitting $k \equiv2 \pmod {25}$, see Table \ref {table:112A2}. Set 5 has 16 subsets valid for each $k \pmod {85}$ with the exception of $k\equiv 18 \pmod {85}$, see Tables \ref {table:112B} and \ref {table:112C}. This replicates exactly the pattern of degree 10 isomorphism classes 0 and 4.


For diameter $k \equiv 3 \pmod 5$, isomorphism class 3 has between 3 and 5 primitive generating sets, depending on the value of $k$. Generating sets 1 and 2 are defined by formulae valid for every $k\equiv 3 \pmod 5$ and set 3 is comprised of three subsets of formulae valid for $k\equiv 3, 8$ and $13 \pmod {15}$, see  Table \ref {table:113A1}. Set 4 has four subsets valid for $k \equiv 3, 8, 13$ and $18 \pmod {25}$, omitting $k \equiv23 \pmod {25}$, see Table \ref {table:113A2}. Set 5 has 16 subsets valid for each $k \pmod {85}$ with the exception of $k\equiv 18 \pmod {85}$, see Tables \ref {table:113B} and \ref {table:113C}. This replicates exactly the pattern of degree 10 isomorphism classes 0 and 4. and degree 11 isomorphism class 2.


For diameter $k \equiv 4 \pmod 5$, isomorphism class 4 has between 2 and 5 primitive generating sets, depending on the value of $k$. Generating sets 1 and 2 are defined by formulae valid for every $k\equiv 4 \pmod 5$, see Table \ref {table:114A}. The formulae for the final three sets repeat with diameter increments of 85, 115 and 235 respectively. Set 3 has 16 subsets valid for each $k \pmod {85}$ with the exception of $k\equiv 4 \pmod {85}$, see Tables \ref {table:114B} and \ref {table:114C}. Set 4 has 22 subsets valid for each $k \pmod {115}$ with the exception of $k\equiv 104 \pmod {115}$, and set 5 has 46 subsets valid for each $k \pmod {235}$ with the exception of $k\equiv 154 \pmod {235}$. For sets 4 and 5 just one example of each is given in Table \ref {table:114D}.


\section {Existence proof for degree 10 circulant graphs of order $L_{CC}(10,k)$ for all diameters $k$}
\label{section:proof10}

In this section we prove the existence of the degree 10 circulant graph of order $L_{CC}(10,k)$ for all diameters $k$. The method of proof closely follows the approach taken by Dougherty and Faber in their proof of the existence of the degree 6 graph of order $L_{CC}(6,k)$ \cite{Dougherty}. For diameter class $k \equiv 0 \pmod5$ all stages of the proof are presented. However resolution of the residual boundary exceptions is only included for the first orthant of the solution space, as an example. The full set of boundary exception resolutions, for all orthants and all diameter classes, runs to several hundred pages. They were resolved and confirmed using a tailored computer program and are available from the author upon request. For all subsequent cases just the specifics are included to avoid unnecessary repetition.

Beforehand we note two theorems by Dougherty and Faber \cite{Dougherty} which will be used in the proofs. For any dimension $f$ we consider $\Z^f$ with the canonical generators $\textbf{e}_i, 1\leq i\leq f$. For any Abelian group $G$ generated by $g_1, ..., g_f$ there is a unique epimorphism from $\Z^f$ onto $G$ which sends $\textbf{e}_i$ to $ g_i$ for all $i$. If $N$ is its kernel, then $G$ is isomorphic to $\Z^f/N$, and the Cayley graph of $G$ with the given generators is isomorphic to the Cayley graph of $\Z^f/N$ with the canonical generators for $\Z^f$. When $G$ is finite the structure of the kernel $N$ is an $f$-dimensional lattice in $\Z^f$. We are interested in the case when $G$ is the cyclic group $\Z_n$ for some $n \in \mathbb{N}$, and so rename the kernel of the corresponding epimorphism as $L_f$. As defined in Section \ref{section:Conjectured}, $S_{f,k}=\{(x_1,...,x_f) \in \Z^f : \vert x_1 \vert +...+\vert x_f \vert \leq k\}$. We have this theorem for circulant graphs of even degree.

\begin{theorem}
(Dougherty and Faber). Let $L_f$ and $S_{f,k}$ be as defined above. If an undirected circulant graph $X$ of order $n$ and degree $2f$ is the Cayley graph for $\Z_n$ with generating set $g_1,...,g_f$, then $X$ has diameter at most $k$ if and only if $S_{f,k}+L_f=\Z^f$.
\label {theoremEvenlattice}
\end{theorem}

Let $\{\textbf{v}_1,\ldots,\textbf{v}_f\}$ be a set of linearly independent vectors in $\Z^f$ that generates the lattice $L_f$, and let $M$ be the $f\times f$ matrix $(\textbf{v}_1,\ldots,\textbf{v}_f)^\textnormal{T}$. Then $M$ is called the lattice generating vectors matrix for the lattice $L_f$, and we have $|G|=|\Z^f/L_f|=|\textnormal{det}M|\le|S_{f,k}|$.

Now consider the case where $X$ is a circulant graph of odd degree $2f+1$, so that its connection set includes the involution $g_m=n/2$. Again let $\textbf{v}_1, \dots, \textbf{v}_f$ be the lattice generating vectors for $L_f$, the lattice corresponding to $X$. Also let $\textbf{v}_m$ be an extra vector associated with the involutory generator $g_m$, defined by $\textbf{v}_m=\frac{1}{2}\sum^f_{i=1}\textbf{v}_i$, so that $2\times\textbf{v}_m\in L_f$. The shortest path from an arbitrary vertex to any other vertex either excludes or includes a single occurrence of the involution $g_m$. This leads to an equivalent theorem for circulant graphs of odd degree.

\begin{theorem}
(Dougherty and Faber) Let $L_f$, $S_{f,k}$ and $\textbf{v}_i$ be as defined above. If an undirected circulant graph $X$ of order $n$ and degree $2f+1$ is the Cayley graph for $\Z_n$ with generating set $g_1,...,g_f$, then $X$ has diameter at most $k$ if and only if
$(S_{f,k}+L_f) \cup (S_{f,k-1}+\textbf{v}_m+L_f)=\Z^f$.
\label {theoremOddlattice}
\end{theorem}

In this case, with $M=(\textbf{v}_1,\ldots,\textbf{v}_f)^\textnormal{T}$, we have $|G|=|\Z^f/L_f|=|\textnormal{det}M|\le|S_{f,k}|+|S_{f,k-1}|$.


\subsection {Existence proof for degree 10 circulant graphs of order $L_{CC}(10,k)$ for all diameters $k \equiv 0 \pmod 5$}

First we prove the existence of the degree 10 circulant graph of order $L_{CC}(10,k)$ for all diameters $k \equiv 0\pmod 5$, with generating set 1 of Table \ref {table:100A1}.

\begin{theorem}
For all $k\equiv 0 \pmod 5$, there is an undirected Cayley graph on five generators of a cyclic group which has diameter k and order $L_{CC}(10,k)$, where
\[ L_{CC}(10,k)=(512k^5+1280k^4+6400k^3+8000k^2+6250k+3125)/3125. \\
\]

Moreover, a generating set is
$\{1,
(32k^3+40k^2+300k+125)/125,
(128k^4+160k^3+1200k^2-625)/625,
(128k^4+320k^3+1600k^2+1750k+1250)/625,
(256k^4+640k^3+3000k^2+3750k+1250)/625\}$.

\label{theorem:100A}
\end{theorem}

\begin{proof}
We will show the existence of five-dimensional lattices $L_k \subseteq \Z^5$ such that $\Z^5/L_k$ is cyclic, $S_{5,k}+L_k=\Z^5$, where $S_{5,k}$ is the set of points in $\Z^5$ at a distance at most $k$ from the origin under the $\ell^1$ (Manhattan) metric, and $\vert \Z^5 : L_k\vert = L_{CC}(10,k) $ as specified in the theorem. Then, by Theorem \ref {theoremD}, the resultant Cayley graph has diameter at most $k$.

Let $a= 2k/5$, and let $L_k$ be defined by five generating vectors as follows:
\[
\begin{array}{rcl}
\textbf{v}_1&=&(a+1,-a,-a,a,a)\\
\textbf{v}_2&=&(a,-a-1,-a,a,-a)\\
\textbf{v}_3&=&(-a,a,-a,-a-1,-a)\\
\textbf{v}_4&=&(a+1,-a,-a+1,-a-1,-a)\\
\textbf{v}_5&=&(a-1,a+1,a+1,-a+1,-a-1)\\
\end{array}
\] 
Then the following vectors are in $L_k$:
\newline
$(2a^2+2a+1)\textbf{v}_1 - (2a-1)\textbf{v}_2 + 2a\textbf{v}_4 + 2a^2\textbf{v}_5 = (4a^3+2a^2+6a+1, -1, 0, 0, 0),$
\newline
$(4a^3+4a^2+2a)\textbf{v}_1 - (4a^2-2a)\textbf{v}_2 + \textbf{v}_3 +(4a^2-1)\textbf{v}_4 +4a^3\textbf{v}_5 = (8a^4+4a^3+12a^2-1, 0, -1, 0, 0),$
\newline
$(4a^3+6a^2+5a+2)\textbf{v}_1 - 4a^2\textbf{v}_2 + \textbf{v}_3 +(4a^2+2a)\textbf{v}_4 + (4a^3+2a^2+a)\textbf{v}_5 = (8a^4+8a^3+16a^2+7a+2, 0, 0, -1, 0),$
\newline
$(8a^3+12a^2+9a+4)\textbf{v}_1 - (8a^2-1)\textbf{v}_2 + 2\textbf{v}_3 +(8a^2+4a-1)\textbf{v}_4 + (8a^3+4a^2+a+1)\textbf{v}_5 = (16a^4+16a^3+30a^2+15a+2, 0, 0, 0, -1)$
\newline

Hence we have $\textbf{e}_2 =  (4a^3+2a^2+6a+1)\textbf{e}_1,  \textbf{e}_3 =  (8a^4+4a^3+12a^2-1)\textbf{e}_1, \textbf{e}_4 =  (8a^4+8a^3+16a^2+7a+2)\textbf{e}_1$,     and $\textbf{e}_5 =  (16a^4+16a^3+30a^2+15a+2)\textbf{e}_1$ in $\Z^5/L_k$, and so $\textbf{e}_1$ generates $\Z^5/L_k$.

Also det $\left ( \begin{array} {c} \textbf{v}_1 \\ \textbf{v}_2 \\ \textbf{v}_3 \\ \textbf{v}_4 \\ \textbf{v}_5 \end {array} \right ) = $ det
$\left (
\begin{array} {l r r r r}
16a^5+16a^4+32a^3+16a^2+5a+1 & 0 & 0 & 0 & 0 \\
4a^3+2a^2+6a+1 & -1 & 0 & 0 & 0 \\
8a^4+4a^3+12a^2-1 & 0 & -1 & 0 & 0 \\
8a^4+8a^3+16a^2+7a+2 & 0 & 0 & -1 & 0 \\
16a^4+16a^3+30a^2+15a+2 & 0 & 0 & 0 & -1  
\end {array} \right )$ 
\newline
\newline
$= 16a^5+16a^4+32a^3+16a^2+5a+1 = (512k^5+1280k^4+6400k^3+8000k^2+6250k+3125)/3125= L_{CC}(10,k)$, as in the statement of the theorem. 

Thus $\Z^5/L_k$ is isomorphic to $\Z_{L_{CC}(10,k)}$ via an isomorphism taking $\textbf{e}_1, \textbf{e}_2, \textbf{e}_3, \textbf{e}_4, \textbf{e}_5$ respectively to $1, 4a^3+2a^2+6a+1, 8a^4+4a^3+12a^2-1,  8a^4+8a^3+16a^2+7a+2, 16a^4+16a^3+30a^2+15a+2 $. As $a=2k/5$ this gives the generating set specified in the theorem: $\{1,
(32k^3+40k^2+300k+125)/125,
(128k^4+160k^3+1200k^2-625)/625,
(128k^4+320k^3+1600k^2+1750k+1250)/625,
(256k^4+640k^3+3000k^2+3750k+1250)/625\}$.

It remains to show that $S_{5,k}+L_k=\Z^5$. For $k=5$, it is straightforward to show directly that $\Z_{1099}$ with generators $1, 53, 207, 272, 536$ has diameter 5. So we assume $k\geq10$, so that $a \geq 4$. Now let
\[
\begin{array}{rll}
\textbf{v}_6&=\textbf{v}_1-\textbf{v}_2+\textbf{v}_3&=(-a+1,a+1,-a,-a-1,a)\\
\textbf{v}_7&=\textbf{v}_1-\textbf{v}_2+\textbf{v}_4&=(a+2,-a+1,-a+1,-a-1,a)\\
\textbf{v}_8&=\textbf{v}_1-\textbf{v}_2+\textbf{v}_5&=(a,a+2,a+1,-a+1,a-1)\\
\textbf{v}_9&=\textbf{v}_1+\textbf{v}_3-\textbf{v}_4&=(-a,a,-a-1,a,a)\\
\textbf{v}_{10}&=\textbf{v}_1+\textbf{v}_3+\textbf{v}_5&=(a,a+1,-a+1,-a,-a-1)\\
\textbf{v}_{11}&=\textbf{v}_1-\textbf{v}_4+\textbf{v}_5&=(a-1,a+1,a,a+2,a-1)\\
\textbf{v}_{12}&=\textbf{v}_2+\textbf{v}_3-\textbf{v}_4&=(-a-1,a-1,-a-1,a,-a)\\
\textbf{v}_{13}&=\textbf{v}_2-\textbf{v}_4+\textbf{v}_5&=(a-2,a,a,a+2,-a-1)\\
\textbf{v}_{14}&=\textbf{v}_2+\textbf{v}_5+\textbf{v}_9&=(a-1,a,-a,a+1,-a-1)\\
\textbf{v}_{15}&=\textbf{v}_1+\textbf{v}_5+\textbf{v}_9&=(a,a+1,-a,a+1,a-1)\\
\textbf{v}_{16}&=\textbf{v}_1+\textbf{v}_5+\textbf{v}_6&=(a+1,a+2,-a+1,-a,a-1).
\end{array}
\]
Then the 32 vectors $\pm\textbf{v}_i$ for $i=1,...,16$ provide one element of $L_k$ lying strictly within each of the 32 orthants of $\Z^5$. Many of the coordinates of these vectors have absolute value at most $a+1$. Only $\pm \textbf{v}_7,\pm \textbf{v}_8,  \pm \textbf{v}_{11}, \pm \textbf{v}_{13}, $ and $\pm \textbf{v}_{16}$ each have one coordinate with absolute value equal to $a+2$.

We must show that each $\textbf{x}\in \Z^5$ is in $S_{5,k}+L_k$, which means that for any $\textbf{x} \in \Z^5$ we need to find a $\textbf{w} \in L_k$ such that $\textbf{x}-\textbf{w} \in S_{5,k}$. However $\textbf{x}-\textbf{w} \in S_{5,k}$ if and only if $\delta (\textbf{x},\textbf{w})\leq k$, where $\delta$ is the $l^1$ metric on $\Z^5$.
If $\textbf{x}, \textbf{y}, \textbf{z} \in \Z^5$ and each coordinate of $\textbf{y}$ lies between the corresponding coordinate of $\textbf{x}$ and $\textbf{z}$ or is equal to one of them, then $\delta (\textbf{x},\textbf{y})+\delta(\textbf{y},\textbf{z})=\delta(\textbf{x},\textbf{z})$. In such a case we say that  \textquotedblleft $\textbf{y}$ lies between $\textbf{x}$ and $\textbf{z}$\textquotedblright.

For any $\textbf{x}\in \Z^5$, we reduce $\textbf{x}$ by adding appropriate elements of $L_k$ until the resulting vector lies within $l^1$-distance $k$ of $\textbf{0}$ or some other element of $L_k$.
The first stage is to reduce $\textbf{x}$ to a vector whose coordinates all have absolute value at most $a+1$. If $\textbf{x}$ has a coordinate with absolute value above $a+1$, then let $\textbf{v}$ be one of the vectors $\pm \textbf{v}_i(1 \leq i \leq 16)$ such that the coordinates of  $\textbf{v}$ have the same sign as the corresponding coordinates of $\textbf{x}$. If a coordinate of $\textbf{x}$ is 0 then either sign is allowed for $\textbf{v}$ as long as the corresponding coordinate of $\textbf{v}$ has absolute value $\leq a+1$.
So if $\textbf{x}$ lies in the orthant of $\textbf{v}_8$ and its $\textbf{e}_2$ coordinate is 0 then we take $\textbf{v}_{12}$ instead.
If $\textbf{x}$ lies in the orthant of $\textbf{v}_{11}$ and its $\textbf{e}_4$ coordinate is 0 then we take $\textbf{v}_8$ instead, unless its $\textbf{e}_2$ coordinate is also 0, in which case we take $\textbf{v}_{12}$ instead.
If $\textbf{x}$ lies in the orthant of $\textbf{v}_{13}$ and its $\textbf{e}_4$ coordinate is 0 then we take $\textbf{v}_5$ instead.
If $\textbf{x}$ lies in the orthant of $\textbf{v}_7$ and its $\textbf{e}_1$ coordinate is 0 then we take $-\textbf{v}_{13}$ instead, unless its $\textbf{e}_4$ coordinate is also 0, in which case we take $-\textbf{v}_5$ instead, as above.
If $\textbf{x}$ lies in the orthant of $\textbf{v}_{16}$ and its $\textbf{e}_2$ coordinate is 0 then we consider $\textbf{v}_7$ instead, with the provisos stated above.

Now consider $\textbf{x}'=\textbf{x}-\textbf{v}$. If a coordinate of $\textbf{x}$ has absolute value $s, 1\leq s\leq a+1$, then the corresponding coordinate of $\textbf{x}'$ will have absolute value $s'\leq a+1$ because of the sign matching and the fact that the coordinates of $\textbf{v}$ have absolute value $\leq a+2$. If a coordinate of $\textbf{x}$ has absolute value $s=0$, then as indicated above, the corresponding value of $\textbf{x}'$ will have absolute value $s' \leq a+1$ because $\textbf{v}$ is chosen such that the corresponding coordinate has absolute value $\leq a+1$. If a coordinate of $\textbf{x}$ has absolute value $s>a+1$, then the corresponding coordinate of $\textbf{x}'$ will be strictly smaller in absolute value. Therefore repeating this procedure will result in a vector whose coordinates all have absolute value at most $a+1$.

If the resulting vector $\textbf{x}'$ lies between $\textbf{0}$ and $\textbf{v}$, where $\textbf{v}=\pm\textbf{v}_i$ for some $i$, then we have $\delta(\textbf{0},\textbf{x}')+\delta(\textbf{x}',\textbf{v})=\delta(\textbf{0},\textbf{v})$. All of the vectors $\textbf{v}$ satisfy $\delta(\textbf{0},\textbf{v})=5a+1$, so we have $\delta(\textbf{0},\textbf{v})=2k+1$. Since $\delta(\textbf{0},\textbf{x}')$ and $\delta(\textbf{x}',\textbf{v})$ are both non-negative integers, one of them must be at most $k$, so that $\textbf{x}' \in S_{5,k}+L_k$. Hence we also have $\textbf{x} \in S_{5,k}+L_k$ as required.

Now we are left with the case where the absolute value of each coordinate of the reduced $\textbf{x}$ is at most $a+1$, and $\textbf{x}$ is in the orthant of $\textbf{v}$, where $\textbf{v} = \pm \textbf{v}_i$ for some $i \leq 16$ but does not lie between $\textbf{0}$ and $\textbf{v}$.
Since $L_k$ is centrosymmetric we only need to consider the 16 orthants containing $\textbf{v}_1, ..., \textbf{v}_{16}$.
To avoid this paper being unduly long just one set of exceptions, for the orthant of $\textbf{v}_1$, are included here as an example. A full listing of the resolution of the exceptions for all 16 orthants is available from the author on request.


Suppose that $\textbf{x}$ lies within the orthant of $\textbf{v}_1$, but not between $\textbf{0}$ and $\textbf{v}_1$. Then as $\textbf{v}_1 = (a+1,-a,-a,a,a)$, the second or third coordinate of $\textbf{x}$ is equal to $-a-1$ or the fourth or fifth coordinate equals $a+1$ or any combination of these. We now distinguish 15 cases.

Case 1: $\textbf{x}=(r,-a-1,-t,u,v)$ where $0\leq s\leq a+1$ and $0 \leq t,u,v \leq a$.
Let $\textbf{x}' = \textbf{x} - \textbf{v}_1 =( r-a-1, -1, a-t, u-a, v-a)$, which lies between $\textbf{0}$ and $-\textbf{v}_{15}$ unless $r = 0$, in which case let $\textbf{x}'' = \textbf{x}'+\textbf{v}_{15}=(-1, a, -t, u+1, v-1)$, which lies between $\textbf{0}$ and $\textbf{v}_9$ unless $u=a$ or $v=0$.
If $v=0$ then $\textbf{x}$ lies between $\textbf{0}$ and $\textbf{v}_2$.
If $u=a$ and $v \geq 1$ then $\textbf{x}''' = \textbf{x}''-\textbf{v}_9=(a-1, 0, a+1-t, 1, v-a-1)$
which lies between $\textbf{0}$ and $-\textbf{v}_6$ unless $t\leq 1$, in which case
let $\textbf{x}'''' = \textbf{x}'''+\textbf{v}_6=(0, a+1, 1-t, -a, v-1)$
which lies between $\textbf{0}$ and $-\textbf{v}_2$.

Case 2: $\textbf{x}=(r,-s,-a-1,u,v)$ where $0\leq r\leq a+1$ and $0 \leq s,u,v \leq a$.
Let $\textbf{x}' = \textbf{x} - \textbf{v}_1 = (r-a-1,a-s, -1, u-a, v-a)$, which lies between $\textbf{0}$ and $\textbf{v}_3$ unless $r = 0$, in which case
\textbf{x} lies between $\textbf{0}$ and $-\textbf{v}_5$ unless $u=a$, in which case
let $\textbf{x}'' = \textbf{x}+\textbf{v}_5=( a-1, a+1-s, 0, 1, v-a-1)$
which lies between $\textbf{0}$ and $\textbf{v}_{14}$, unless $s=0$, in which case
\textbf{x} lies between $\textbf{0}$ and $\textbf{v}_9$.

Case 3: $\textbf{x}=(r,-s,-t,a+1,v)$ where $0\leq r\leq a+1$ and $0 \leq s,t,v \leq a$.
Let $\textbf{x}' = \textbf{x} - \textbf{v}_1 = (r-a-1,a-s,a-t,1,v-a)$, which lies between $\textbf{0}$ and $-\textbf{v}_7$ unless $s = 0$ or $t = 0$.
Let $\textbf{x}'' = \textbf{x}'+\textbf{v}_7=(1+r,1-s,1-t,-a,v)$.
If $t=0$ then $\textbf{x}$ lies between $\textbf{0}$ and $-\textbf{v}_3$ unless $r \geq a$, in which case
let $\textbf{x}'''=\textbf{x}+\textbf{v}_3 = (r-a,a-s,-a,0,v-a)$
which lies between $\textbf{0}$ and $\textbf{v}_{14}$.
If $s=0$ and $t \geq 1$ then $\textbf{x}''$ lies between $\textbf{0}$ and $\textbf{v}_{16}$
unless $r=a+1$ or $v=a$.
If $v=a$ then $\textbf{x}'$ lies between $\textbf{0}$ and $-\textbf{v}_4$.
If $r=a+1$ and $v \leq a$ then
let $\textbf{x}''''=\textbf{x}''-\textbf{v}_{16} = (r-a,-a-1,a-t,0,v-a+1)$
which lies between $\textbf{0}$ and $-\textbf{v}_6$.

Case 4: $\textbf{x}=(r,-s,-t,u,a+1)$ where $0\leq r\leq a+1$ and $0 \leq s,t,u \leq a$.
Let $\textbf{x}' = \textbf{x} - \textbf{v}_1 = (r-a-1,a-s,a-t,u-a,1)$, which lies between $\textbf{0}$ and $-\textbf{v}_2$ unless $r = 0$, in which case
let $\textbf{x}''=\textbf{x}+\textbf{v}_2 = (1,-s-1,-t,u,1-a)$
which lies between $\textbf{0}$ and $-\textbf{v}_8$
unless $u=a$, in which case 
let $\textbf{x}'''=\textbf{x}''+\textbf{v}_8 = (a-1,a+1-s,a+1-t,1,0)$
which lies between $\textbf{0}$ and $\textbf{v}_{11}$
unless $t=0$, in which case 
$\textbf{x}$ lies between $\textbf{0}$ and $-\textbf{v}_{10}$.

Case 5: $\textbf{x}=(r,-a-1,-a-1,u,v)$ where $0\leq r\leq a+1$ and $0 \leq u,v \leq a$.
Let $\textbf{x}' = \textbf{x} - \textbf{v}_1 = (r-a-1,-1,-1,u-a,v-a)$, which lies between $\textbf{0}$ and $-\textbf{v}_{11}$ unless $r \leq 1$ or $v = 0$.
Let $\textbf{x}''=\textbf{x}'+\textbf{v}_{11} = (r-2,a,a-1,u+2,v-1)$.
If $r \leq 1$ and $v \geq 1$ then $\textbf{x}''$ lies between $\textbf{0}$ and $-\textbf{v}_4$
unless $u=a$, in which case 
let $\textbf{x}'''=\textbf{x}''+\textbf{v}_4 = (r+a-1,0,0,1,v-a-1)$
which lies between $\textbf{0}$ and $\textbf{v}_2$.
If $r \leq 1$ and $v=0$ then
let $\textbf{x}''''=\textbf{x}''+\textbf{v}_7 = (r+a,1,0,u-a+1,a-1)$
which lies between $\textbf{0}$ and $-\textbf{v}_4$
unless $u=a$, in which case 
let $\textbf{x}'''=\textbf{x}''+\textbf{v}_4 = (r+a-1,0,0,1,v-a-1)$
which lies between $\textbf{0}$ and $\textbf{v}_{16}$.
unless $u=a$, in which case
let $\textbf{x}^v=\textbf{x}''''-\textbf{v}_{16} = (r-1,-a-1,a-1,a+1,0)$
which lies between $\textbf{0}$ and $-\textbf{v}_6$
unless $r=0$, in which case
let $\textbf{x}^{vi}=\textbf{x}^v+\textbf{v}_6 = (-a,0,-1,0,a)$
which lies between $\textbf{0}$ and $\textbf{v}_9$.
If $r \geq 2$ and $v=0$ then
$\textbf{x}''$ lies between $\textbf{0}$ and $\textbf{v}_{13}$
unless $r=a+1$, in which case 
$\textbf{x}'$ lies between $\textbf{0}$ and $\textbf{v}_4$.

Case 6: $\textbf{x}=(r,-a-1,-t,a+1,v)$ where $0\leq r\leq a+1$ and $0 \leq t,v \leq a$.
Let $\textbf{x}' = \textbf{x} - \textbf{v}_1 = (r-a-1,-1,a-t,1,v-a)$, which lies between $\textbf{0}$ and $-\textbf{v}_{16}$ unless $t=0$ or $v=0$.
Let $\textbf{x}''=\textbf{x}'+\textbf{v}_{16} = (r,a+1,1-t,1-a,v-1)$.
If $t=0$ and $v=0$ then $\textbf{x}''$ lies between $\textbf{0}$ and $\textbf{v}_5$
unless $r\geq a$, in which case 
let $\textbf{x}'''=\textbf{x}''-\textbf{v}_5 = (r-a+1,0,-a,0,a)$
which lies between $\textbf{0}$ and $\textbf{v}_1$.
If $t \geq 1$ and $v=0$ then $\textbf{x}''$ lies between $\textbf{0}$ and $\textbf{v}_{10}$
unless $r=a+1$, in which case 
let $\textbf{x}'''=\textbf{x}''-\textbf{v}_{10} = (1,0,a-t,1,a)$
which lies between $\textbf{0}$ and $-\textbf{v}_4$.
If $t=0$ and $v \geq 1$ then $\textbf{x}''$ lies between $\textbf{0}$ and $\textbf{v}_8$
unless $r=a+1$, in which case 
let $\textbf{x}^v=\textbf{x}''-\textbf{v}_8 = (1,-1,-a,0,v-a)$
which lies between $\textbf{0}$ and $\textbf{v}_2$.

Case 7: $\textbf{x}=(r,-a-1,-t,u,a+1)$ where $0\leq r\leq a+1$ and $0 \leq t,u \leq a$.
Let $\textbf{x}' = \textbf{x} - \textbf{v}_1 = (r-a-1,-1,a-t,u-a,1)$, which lies between $\textbf{0}$ and $-\textbf{v}_{14}$ unless $r \leq 1$, in which case
let $\textbf{x}''=\textbf{x}'+\textbf{v}_{14} = (r-2,a-1,-t,u+1,-a)$
which lies between $\textbf{0}$ and $-\textbf{v}_{12}$
unless $u=a$, in which case 
let $\textbf{x}'''=\textbf{x}''-\textbf{v}_{12} = (r+a-1,0,a+1-t,1,0)$
which lies between $\textbf{0}$ and $-\textbf{v}_3$
unless $t=0$, in which case 
let $\textbf{x}''''=\textbf{x}'''+\textbf{v}_3 = (r-1,a,1,-a,-a)$
which lies between $\textbf{0}$ and $-\textbf{v}_2$.

Case 8: $\textbf{x}=(r,-s,-a-1,a+1,v)$ where $0\leq r\leq a+1$ and $0 \leq s,v \leq a$.
Let $\textbf{x}' = \textbf{x} - \textbf{v}_1 = (r-a-1,a-s,-1,1,v-a)$, which lies between $\textbf{0}$ and $\textbf{v}_{12}$ unless $s=0$, in which case
let $\textbf{x}''=\textbf{x}'-\textbf{v}_{12} = (r,1,a,1-a,v)$
which lies between $\textbf{0}$ and $\textbf{v}_8$
unless $r=a+1$ or $v=a$.
Let $\textbf{x}''' = \textbf{x}'' - \textbf{v}_8 = (r-a,-a-1,-1,0,v-a+1)$.
If $r=a+1$ and $v=a$ then
let $\textbf{x}''''=\textbf{x}'''-\textbf{v}_1 = (-a,-1,a-1,-a,1-a)$
which lies between $\textbf{0}$ and $-\textbf{v}_{15}$.
If $r=a+1$ and $v\leq a-1$ then
$\textbf{x}'''$ lies between $\textbf{0}$ and $\textbf{v}_2$.
If $r\leq a$ and $v=a$ then
$\textbf{x}'''$ lies between $\textbf{0}$ and $-\textbf{v}_5$.
unless $r=0$, in which case 
let $\textbf{x}^v=\textbf{x}'''+\textbf{v}_5 = (-1,0,a,1-a,-a)$
which lies between $\textbf{0}$ and $-\textbf{v}_1$.

Case 9: $\textbf{x}=(r,-s,-a-1,u,a+1)$ where $0\leq r\leq a+1$ and $0 \leq s,u \leq a$.
Let $\textbf{x}' = \textbf{x} - \textbf{v}_1 = (r-a-1,a-s,-1,u-a,1)$, which lies between $\textbf{0}$ and $\textbf{v}_6$ unless $r \leq 1$, in which case
let $\textbf{x}''=\textbf{x}'-\textbf{v}_6 = (r-2,-s-1,a-1,u+1,1-a)$
which lies between $\textbf{0}$ and $-\textbf{v}_{16}$
unless $u=a$  in which case 
let $\textbf{x}'''=\textbf{x}''+\textbf{v}_{16} = (r+a-1,a+1-s,0,1,0)$
which lies between $\textbf{0}$ and $\textbf{v}_{15}$.

Case 10: $\textbf{x}=(r,-s,-t,a+1,a+1)$ where $0\leq r\leq a+1$ and $0 \leq s,t \leq a$.
Let $\textbf{x}' = \textbf{x} - \textbf{v}_1 = (r-a-1,a-s,a-t,1,1)$, which lies between $\textbf{0}$ and $-\textbf{v}_4$ unless $s=0$, in which case
let $\textbf{x}''=\textbf{x}'+\textbf{v}_4 = (r,0,1-t,-a,1-a)$.
If $t=0$ then
$\textbf{x}''$ lies between $\textbf{0}$ and $-\textbf{v}_9$
unless $r=a+1$, in which case 
let $\textbf{x}'''=\textbf{x}''+\textbf{v}_9 = (1,a,-a,0,1)$
which lies between $\textbf{0}$ and $\textbf{v}_{15}$.
If $t\geq 1$ then
$\textbf{x}''$ lies between $\textbf{0}$ and $\textbf{v}_4$.

Case 11: $\textbf{x}=(r,-a-1,-a-1,a+1,v)$ where $0\leq r\leq a+1$ and $0 \leq v \leq a$.
Let $\textbf{x}' = \textbf{x} - \textbf{v}_1 = (r-a-1,-1,-1,1,v-a)$, which lies between $\textbf{0}$ and $-\textbf{v}_8$ unless $r=0$ or $v=0$.
Let $\textbf{x}''=\textbf{x}'+\textbf{v}_8 = (r-1,a+1,a,2-a,v-1)$.
If $r=0$ and $v=0$ then
let $\textbf{x}'''=\textbf{x}''+\textbf{v}_8 = (a,1,0,2,a-1)$
which lies between $\textbf{0}$ and $\textbf{v}_{15}$.
If $r=0$ and $v\geq 1$ then $\textbf{x}''$ lies between $\textbf{0}$ and $-\textbf{v}_2$.
If $r\geq 1$ and $v=0$ then $\textbf{x}''$ lies between $\textbf{0}$ and $\textbf{v}_5$
unless $r=a+1$, in which case 
let $\textbf{x}'''=\textbf{x}''-\textbf{v}_5 = (1,0,-1,1,a)$
which lies between $\textbf{0}$ and $\textbf{v}_1$.

Case 12: $\textbf{x}=(r,-a-1,-a-1,u,a+1)$ where $0\leq r\leq a+1$ and $0 \leq u \leq a$.
Let $\textbf{x}' = \textbf{x} - \textbf{v}_1 = (r-a-1,-1,-1,u-a,1)$, which lies between $\textbf{0}$ and $-\textbf{v}_{13}$ unless $r \leq 1$, in which case
let $\textbf{x}''=\textbf{x}'+\textbf{v}_{13} = (r-3,a-1,a-1,u+2,-a)$
which lies between $\textbf{0}$ and $-\textbf{v}_7$
unless $u=a$  in which case 
let $\textbf{x}'''=\textbf{x}''+\textbf{v}_7 = (r+a-1,0,0,1,0)$
which lies between $\textbf{0}$ and $-\textbf{v}_3$.

Case 13: $\textbf{x}=(r,-a-1,-t,a+1,a+1)$ where $0\leq r\leq a+1$ and $0 \leq t \leq a$.
Let $\textbf{x}' = \textbf{x} - \textbf{v}_1 = (r-a-1,-1,a-t,1,1)$, which lies between $\textbf{0}$ and $-\textbf{v}_{10}$ unless $r=0$ or $t=0$.
Let $\textbf{x}''=\textbf{x}'+\textbf{v}_{10} = (r-1,a,1-t,1-a,-a)$.
If $r=0$ and $t=0$ then $\textbf{x}''$ lies between $\textbf{0}$ and $-\textbf{v}_1$.
If $r=0$ and $t\geq 1$ then $\textbf{x}''$ lies between $\textbf{0}$ and $\textbf{v}_3$.
If $r\geq 1$ and $t=0$ then $\textbf{x}''$ lies between $\textbf{0}$ and $\textbf{v}_5$
unless $r=a+1$, in which case 
let $\textbf{x}'''=\textbf{x}''-\textbf{v}_5 = (1,-1,-a,0,1)$
which lies between $\textbf{0}$ and $\textbf{v}_1$.

Case 14: $\textbf{x}=(r,-s,-a-1,a+1,a+1)$ where $0\leq r\leq a+1$ and $0 \leq s \leq a$.
Let $\textbf{x}' = \textbf{x} - \textbf{v}_1 = (r-a-1,a-s,-1,1,1)$, which lies between $\textbf{0}$ and $\textbf{v}_9$ unless $r=0$, in which case
let $\textbf{x}''=\textbf{x}'-\textbf{v}_9 = (-1,-s,a,1-a,1-a)$
which lies between $\textbf{0}$ and $-\textbf{v}_{15}$.

Case 15: $\textbf{x}=(r,-a-1,-a-1,a+1,a+1)$ where $0\leq r\leq a+1$.
Let $\textbf{x}' = \textbf{x} - \textbf{v}_1 = (r-a-1,-1,-1,1,1)$, which lies between $\textbf{0}$ and $-\textbf{v}_5$ unless $r\leq 1$, in which case
let $\textbf{x}''=\textbf{x}'+\textbf{v}_5 = (r-2,a,a,2-a,-a)$
which lies between $\textbf{0}$ and $-\textbf{v}_1$.

This completes the cases for the orthant of $\textbf{v}_1$ for $k \equiv 0 \pmod 5$.  A full listing of the resolution of the exceptions for all 16 orthants is available from the author on request.
\end{proof}


\subsection {Existence proof for degree 10 circulant graphs of order $L_{CC}(10,k)$ for all diameters $k \equiv 1 \pmod 5$}

The existence of the degree 10 circulant graph of order $L_{CC}(10,k)$ for all diameters $k \equiv 1\pmod 5$, with generating set 1 of Table \ref {table:101A}, is proved following the same method as for Theorem \ref{theorem:100A}.

\begin{theorem}
For all $k\equiv 1 \pmod 5$, there is an undirected Cayley graph on five generators of a cyclic group which has diameter k and order $L_{CC}(10,k)$, where
\[ L_{CC}(10,k)=(512k^5+1280k^4+6560k^3+9520k^2+6100k+1028)/3125. \\
\]

Moreover, a generating set is
$\{1,
(128k^4+448k^3+1688k^2+3768k+843)/625,
(256k^4+416k^3+2816k^2+2346k+416)/625,
(256k^4+576k^3+3136k^2+3976k+1431)/625,
(384k^4+864k^3+4504k^2+5614k+1134)/625\}$.

\label{theorem:101A}
\end{theorem}

\begin{proof}
We will show the existence of five-dimensional lattices $L_k \subseteq \Z^5$ such that $\Z^5/L_k$ is cyclic, $S_{5,k}+L_k=\Z^5$, where $S_{5,k}$ is the set of points in $\Z^5$ at a distance of at most $k$ from the origin under the $\ell^1$ (Manhattan) metric, and $\vert \Z^5 : L_k\vert = L_{CC}(10,k) $ as specified in the theorem. Then, by Theorem \ref {theoremD}, the resultant Cayley graph has diameter at most $k$. 

Let $a= (2k+3)/5$, and let $L_k$ be defined by five generating vectors as follows:
\[
\begin{array}{rcl}
\textbf{v}_1&=&(a-1,a,-a,a,a-1)\\
\textbf{v}_2&=&(a-2,-a-1,-a-1,-a+2,a)\\
\textbf{v}_3&=&(a-1,-a,-a+1,a,-a)\\
\textbf{v}_4&=&(a-1,a-1,-a-1,-a+1,a)\\
\textbf{v}_5&=&(a,-a-1,a-1,-a,-a+2)\\
\end{array}
\] 
Then the following vectors are in $L_k$:
\newline
$(4a^3-2a^2+4a+2)\textbf{v}_1 - 4a^2\textbf{v}_2 + (2a+1)\textbf{v}_3 + (2a^2-2a)\textbf{v}_4 + (4a^3+2a+1)\textbf{v}_5 = (8a^4-8a^3+14a^2-3, -1, 0, 0, 0),$
\newline
$(8a^3-10a^2+13a-4)\textbf{v}_1 - (8a^2-6a+2)\textbf{v}_2 + (4a-1)\textbf{v}_3 + (4a^2-7a+4)\textbf{v}_4 + (8a^3-6a^2+6a-2)\textbf{v}_5 = (16a^4-28a^3+44a^2-27a+5, 0, -1, 0, 0),$
\newline
$(8a^3-8a^2+12a-2)\textbf{v}_1 - (8a^2-4a+2)\textbf{v}_2 + 4a\textbf{v}_3 + (4a^2-6a+3)\textbf{v}_4 + (8a^3-4a^2+6a-1)\textbf{v}_5 = (16a^4-24a^3+40a^2-20a+3, 0, 0, -1, 0),$
\newline
$(12a^3-12a^2+17a-3)\textbf{v}_1 - (12a^2-6a+2)\textbf{v}_2 + 6a\textbf{v}_3 + (6a^2-9a+4)\textbf{v}_4 + (12a^3-6a^2+8a-2)\textbf{v}_5 = (24a^4-36a^3+58a^2-29a+3, 0, 0, 0, -1)$
\newline

Hence we have $\textbf{e}_2 =  (8a^4-8a^3+14a^2-3)\textbf{e}_1,  \textbf{e}_3 =  (16a^4-28a^3+44a^2-27a+5)\textbf{e}_1, \textbf{e}_4 =  (16a^4-24a^3+40a^2-20a+3)\textbf{e}_1$,     and $\textbf{e}_5 =  (24a^4-36a^3+58a^2-29a+3)\textbf{e}_1$ in $\Z^5/L_k$, and so $\textbf{e}_1$ generates $\Z^5/L_k$.

Also det $\left ( \begin{array} {c} \textbf{v}_1 \\ \textbf{v}_2 \\ \textbf{v}_3 \\ \textbf{v}_4 \\ \textbf{v}_5 \end {array} \right ) = $ det
$\left (
\begin{array} {l r r r r}
16a^5-32a^4+52a^3-40a^2+14a-2 & 0 & 0 & 0 & 0 \\
8a^4-8a^3+14a^2-3 & -1 & 0 & 0 & 0 \\
16a^4-28a^3+44a^2-27a+5 & 0 & -1 & 0 & 0 \\
16a^4-24a^3+40a^2-20a+3 & 0 & 0 & -1 & 0 \\
24a^4-36a^3+58a^2-29a+3 & 0 & 0 & 0 & -1  
\end {array} \right )$ 
\newline
\newline
$= 16a^5-32a^4+52a^3-40a^2+14a-2 = (512k^5+1280k^4+6560k^3+9520k^2+6100k+1028)/3125= L_{CC}(10,k)$, as in the statement of the theorem. 

Thus $\Z^5/L_k$ is isomorphic to $\Z_{L_{CC}(10,k)}$ via an isomorphism taking $\textbf{e}_1, \textbf{e}_2, \textbf{e}_3, \textbf{e}_4, \textbf{e}_5$ respectively to $1, 8a^4-8a^3+14a^2-3, 16a^4-28a^3+44a^2-27a+5, 16a^4-24a^3+40a^2-20a+3, 24a^4-36a^3+58a^2-29a+3 $. As $a=(2k+3)/5$ this gives the generating set specified in the theorem: $\{1,
(128k^4+448k^3+1688k^2+3768k+843)/625,
(256k^4+416k^3+2816k^2+2346k+416)/625,
(256k^4+576k^3+3136k^2+3976k+1431)/625,
(384k^4+864k^3+4504k^2+5614k+1134)/625\}$.

It remains to show that $S_{5,k}+L_k=\Z^5$. For $k=6$, it is straightforward to show directly that $\Z_{2380}$ with generators $1, 555, 860, 951, 970$ has diameter 6. So we assume $k\geq11$, so that $a \geq 5$. Now let
\[
\begin{array}{rll}
\textbf{v}_6&=\textbf{v}_1+\textbf{v}_2-\textbf{v}_4&=(a-2,-a,-a,a+1,a-1)\\
\textbf{v}_7&=\textbf{v}_1-\textbf{v}_2+\textbf{v}_5&=(a+1,a,a,a-2,-a+1)\\
\textbf{v}_8&=-\textbf{v}_1+\textbf{v}_3+\textbf{v}_4&=(a-1,-a-1,-a,-a+1,-a+1)\\
\textbf{v}_9&=\textbf{v}_1-\textbf{v}_3+\textbf{v}_5&=(a,a-1,a-2,-a,a+1)\\
\textbf{v}_{10}&=\textbf{v}_1-\textbf{v}_4+\textbf{v}_5&=(a,-a,a,a-1,-a+1)\\
\textbf{v}_{11}&=-\textbf{v}_2+\textbf{v}_3+\textbf{v}_4&=(a,a,-a+1,a-1,-a)\\
\textbf{v}_{12}&=-\textbf{v}_2+\textbf{v}_4+\textbf{v}_5&=(a+1,a-1,a-1,-a-1,-a+2)\\
\textbf{v}_{13}&=-\textbf{v}_3+\textbf{v}_5+\textbf{v}_6&=(a-1,-a-1,a-2,-a+1,a+1)\\
\textbf{v}_{14}&=\textbf{v}_1+\textbf{v}_5-\textbf{v}_8&=(a,a,a-1,a-1,a)\\
\textbf{v}_{15}&=-\textbf{v}_2+\textbf{v}_4+\textbf{v}_8&=(a,a-1,-a,-a,-a+1)\\
\textbf{v}_{16}&=\textbf{v}_1-\textbf{v}_4+\textbf{v}_{13}&=(a-1,-a,a-1,a,a).
\end{array}
\]
Then the 32 vectors $\pm\textbf{v}_i$ for $i=1,...,16$ provide one element of $L_k$ lying strictly within each of the 32 orthants of $\Z^5$. All of the coordinates of these vectors have absolute value at most $a+1$. As in the proof of Theorem \ref {theorem:100A} any vector $\textbf{x}$ is reduced to a vector $\textbf{x}'$ with coordinates with absolute value at most $a+1$ by successively subtracting appropriate vectors $\textbf{v}=\pm \textbf{v}_i$. If a coordinate of $\textbf{x}$ is 0 then either sign is allowed for $\textbf{v}$ as the corresponding coordinate of $\textbf{v}$ has absolute value $\leq a+1$.

If the resulting vector $\textbf{x}'$ lies between $\textbf{0}$ and $\textbf{v}$, where $\textbf{v}=\pm\textbf{v}_i$ for some $i$, then we have $\delta(\textbf{0},\textbf{x}')+\delta(\textbf{x}',\textbf{v})=\delta(\textbf{0},\textbf{v})$. All of the vectors $\textbf{v}$ satisfy $\delta(\textbf{0},\textbf{v})=5a-2$, so we have $\delta(\textbf{0},\textbf{v})=2k+1$. Hence $\textbf{x}' \in S_{5,k}+L_k$, so that $\textbf{x} \in S_{5,k}+L_k$ as required. Now we are left with the case where the absolute value of each coordinate of the reduced $\textbf{x}$ is at most $a+1$, and $\textbf{x}$ is in the orthant of $\textbf{v}$, where $\textbf{v} = \pm \textbf{v}_i$ for some $i \leq 16$ but does not lie between $\textbf{0}$ and $\textbf{v}$.
\end{proof}


\subsection {Existence proof for degree 10 circulant graphs of order $L_{CC}(10,k)$ for all diameters $k \equiv 2 \pmod {15}$}

For graphs of diameter $k \equiv 2\pmod {5}$ there is no single generating set valid for all diameters. Therefore we must consider in turn the three subcases $k \equiv 2, 7$ and $ 12 \pmod {15}$. First the existence of the degree 10 circulant graph of order $L_{CC}(10,k)$ for all diameters $k \equiv 2\pmod {15}$, with generating set 1 of Table \ref {table:102A}, is proved.

\begin{theorem}
For all $k\equiv 2 \pmod {15}$, there is an undirected Cayley graph on five generators of a cyclic group which has diameter k and order $L_{CC}(10,k)$, where
\[ L_{CC}(10,k)=(512k^5+1280k^4+6080k^3+7840k^2+10010k+3741)/3125. \\
\]

Moreover, a generating set is
$\{1,
(128k^4+256k^3+1192k^2+1064k+883)/625,
(512k^5+640k^4+4000k^3-1320k^2-6660k-11899)/9375,
(512k^5+640k^4+4000k^3-1320k^2+840k-8149)/9375,
(512k^5+640k^4+6400k^3+2280k^2+13890k-2224)/9375\}$.

\label{theorem:102A}
\end{theorem}

\begin{proof}
We will show the existence of five-dimensional lattices $L_k \subseteq \Z^5$ such that $\Z^5/L_k$ is cyclic, $S_{5,k}+L_k=\Z^5$, where $S_{5,k}$ is the set of points in $\Z^5$ at a distance of at most $k$ from the origin under the $\ell^1$ (Manhattan) metric, and $\vert \Z^5 : L_k\vert = L_{CC}(10,k) $ as specified in the theorem. Then, by Theorem \ref {theoremD}, the resultant Cayley graph has diameter at most $k$. 

Let $a= (2k-4)/15$, and let $L_k$ be defined by five generating vectors as follows:
\[
\begin{array}{rcl}
\textbf{v}_1&=&(3a+1,-3a-1,-3a-1,-3a,-3a-2)\\
\textbf{v}_2&=&(3a+1,3a+1,-3a-1,-3a-1,-3a-1)\\
\textbf{v}_3&=&(3a,3a,3a+1,-3a-3,3a+1)\\
\textbf{v}_4&=&(3a+2,-3a,-3a,3a+2,-3a-1)\\
\textbf{v}_5&=&(3a+1,3a,-3a-1,-3a-2,3a+1)\\
\end{array}
\] 
Then the following vectors are in $L_k$:
\newline
$-(18a^2+6a)\textbf{v}_1 - \textbf{v}_2 + (108a^3+90a^2+30a+3)\textbf{v}_3 + (108a^3+126a^2+54a+8)\textbf{v}_4 + (18a^2+12a+4)\textbf{v}_5 = (648a^4+864a^3+522a^2+156a+19, -1, 0, 0, 0),$
\newline
$-(36a^3+18a^2+2a)\textbf{v}_1 - 2a\textbf{v}_2 + (216a^4+216a^3+90a^2+14a)\textbf{v}_3 + (216a^4+288a^3+150a^2+32a+1)\textbf{v}_4 + (36a^3+30a^2+12a+1)\textbf{v}_5 = (1296a^5+1944a^4+1332a^3+474a^2+78a+3, 0, -1, 0, 0),$
\newline
$-(36a^3+18a^2+2a)\textbf{v}_1 - 2a\textbf{v}_2 + (216a^4+216a^3+90a^2+14a+1)\textbf{v}_3 + (216a^4+288a^3+150a^2+32a+2)\textbf{v}_4 + (36a^3+30a^2+12a+1)\textbf{v}_5 = (1296a^5+1944a^4+1332a^3+474a^2+84a+5, 0, 0, -1, 0),$
\newline
$-(36a^3+18a^2+5a)\textbf{v}_1 - 2a\textbf{v}_2 + (216a^4+216a^3+108a^2+23a+2)\textbf{v}_3 + (216a^4+288a^3+168a^2+47a+5)\textbf{v}_4 + (36a^3+30a^2+15a+2)\textbf{v}_5 = (1296a^5+1944a^4+1440a^3+582a^2+129a+12, 0, 0, 0, -1).$
\newline

Hence we have $\textbf{e}_2 =  (648a^4+864a^3+522a^2+156a+19)\textbf{e}_1,  \textbf{e}_3 =  (1296a^5+1944a^4+1332a^3+474a^2+78a+3)\textbf{e}_1, \textbf{e}_4 =  (1296a^5+1944a^4+1332a^3+474a^2+84a+5)\textbf{e}_1$,  and $\textbf{e}_5 =  (1296a^5+1944a^4+1440a^3+582a^2+129a+12)\textbf{e}_1$ in $\Z^5/L_k$, and so $\textbf{e}_1$ generates $\Z^5/L_k$.

Also det $\left ( \begin{array} {c} \textbf{v}_1 \\ \textbf{v}_2 \\ \textbf{v}_3 \\ \textbf{v}_4 \\ \textbf{v}_5 \end {array} \right )$

$ = $ det
$\left (
\begin{array} {l r r r r}
3888a^5+6480a^4+4968a^3+2088a^2+471a+45 & 0 & 0 & 0 & 0 \\
648a^4+864a^3+522a^2+156a+19 & -1 & 0 & 0 & 0 \\
1296a^5+1944a^4+1332a^3+474a^2+78a+3 & 0 & -1 & 0 & 0 \\
1296a^5+1944a^4+1332a^3+474a^2+84a+5 & 0 & 0 & -1 & 0 \\
1296a^5+1944a^4+1440a^3+582a^2+129a+12 & 0 & 0 & 0 & -1  
\end {array} \right )$ 
\newline
\newline
$= 3888a^5+6480a^4+4968a^3+2088a^2+471a+45 = (512k^5+1280k^4+6080k^3+7840k^2+10010k+3741)/3125= L_{CC}(10,k)$, as in the statement of the theorem. 

Thus $\Z^5/L_k$ is isomorphic to $\Z_{L_{CC}(10,k)}$ via an isomorphism taking $\textbf{e}_1, \textbf{e}_2, \textbf{e}_3, \textbf{e}_4, \textbf{e}_5$ respectively to $1, 648a^4+864a^3+522a^2+156a+19, 1296a^5+1944a^4+1332a^3+474a^2+78a+3, 1296a^5+1944a^4+1332a^3+474a^2+84a+5, 1296a^5+1944a^4+1440a^3+582a^2+129a+12 $. As $a=(2k-4)/15$ this gives the generating set specified in the theorem: $\{1,
(128k^4+256k^3+1192k^2+1064k+883)/625,
(512k^5+640k^4+4000k^3-1320k^2-6660k-11899)/9375,
(512k^5+640k^4+4000k^3-1320k^2+840k-8149)/9375,
(512k^5+640k^4+6400k^3+2280k^2+13890k-2224)/9375\}$.

It remains to show that $S_{5,k}+L_k=\Z^5$. For $k=2$ it is straightforward to show directly that $\Z_{45}$ with generators $1, 3, 5, 12, 19$ has diameter 2, and also that $\Z_{277179}$ with generators $1, 19699, 85287, 85301, 86694$ has diameter 17. So we assume $k\geq32$, so that $a \geq 4$. Now let
\[
\begin{array}{rll}
\textbf{v}_6&=       \textbf{v}_1  -\textbf{v}_2   +\textbf{v}_3&=     (3a,-3a-2,3a+1,-3a-2,3a)\\
\textbf{v}_7&=      -\textbf{v}_1  +\textbf{v}_2   +\textbf{v}_4&=    (3a+2,3a+2,-3a,3a+1,-3a)\\
\textbf{v}_8&=       \textbf{v}_1  -\textbf{v}_2  +\textbf{v}_5&=     (3a+1,-3a-2,-3a-1,-3a-1,3a)\\
\textbf{v}_9&=      -\textbf{v}_2  +\textbf{v}_3    +\textbf{v}_4&=   (3a+1,-3a-1,3a+2,3a,3a+1)\\
\textbf{v}_{10}&=  -\textbf{v}_2  +\textbf{v}_4   +\textbf{v}_5&=   (3a+2,-3a-1,-3a,3a+1,3a+1)\\
\textbf{v}_{11}&=  \textbf{v}_3  +\textbf{v}_4  -\textbf{v}_5&=     (3a+1,-3a,3a+2,3a+1,-3a-1)\\
\textbf{v}_{12}&=  \textbf{v}_2  -\textbf{v}_5  +\textbf{v}_6&=    (3a,-3a-1,3a+1,-3a-1,-3a-2)\\
\textbf{v}_{13}&= -\textbf{v}_2  +\textbf{v}_3  +\textbf{v}_7&=    (3a+1,3a+1,3a+2,3a-1,3a+2)\\
\textbf{v}_{14}&=  -\textbf{v}_2  +\textbf{v}_5  +\textbf{v}_7&=     (3a+2,3a+1,-3a,3a,3a+2)\\
\textbf{v}_{15}&=  \textbf{v}_3  -\textbf{v}_5  +\textbf{v}_7&=      (3a+1,3a+2,3a+2,3a,-3a)\\
\textbf{v}_{16}&=  \textbf{v}_1  +\textbf{v}_3   -\textbf{v}_8&=    (3a,3a+1,3a+1,-3a-2,-3a-1).
\end{array}
\]
Then the 32 vectors $\pm\textbf{v}_i$ for $i=1,...,16$ provide one element of $L_k$ lying strictly within each of the 32 orthants of $\Z^5$. Most of the coordinates of these vectors have absolute value at most $3a+2$. Only $\pm \textbf{v}_3$ have one coordinate with absolute value equal to $3a+3$.
As in the proof of Theorem \ref {theorem:100A} any vector $\textbf{x}$ is reduced to a vector $\textbf{x}'$ with coordinates with absolute value at most $3a+2$ by successively subtracting appropriate vectors $\textbf{v}=\pm \textbf{v}_i$. If a coordinate of $\textbf{x}$ is 0 then either sign is allowed for $\textbf{v}$ as the corresponding coordinate of $\textbf{v}$ has absolute value $\leq 3a+2$.
So if the $\textbf{e}_4$ coordinate of $\textbf{x}$ is 0 then we avoid $\textbf{v}_3$ and take $\textbf{v}_{13}$ instead.

If the resulting vector $\textbf{x}'$ lies between $\textbf{0}$ and $\textbf{v}$, where $\textbf{v}=\pm\textbf{v}_i$ for some $i$, then we have $\delta(\textbf{0},\textbf{x}')+\delta(\textbf{x}',\textbf{v})=\delta(\textbf{0},\textbf{v})$. All of the vectors $\textbf{v}$ satisfy $\delta(\textbf{0},\textbf{v})=15a+5$, so we have $\delta(\textbf{0},\textbf{v})=2k+1$. Hence $\textbf{x}' \in S_{5,k}+L_k$, so that $\textbf{x} \in S_{5,k}+L_k$ as required. Now we are left with the case where the absolute value of each coordinate of the reduced $\textbf{x}$ is at most $3a+2$, and $\textbf{x}$ is in the orthant of $\textbf{v}$, where $\textbf{v} = \pm \textbf{v}_i$ for some $i \leq 16$ but does not lie between $\textbf{0}$ and $\textbf{v}$.
\end{proof}


\subsection {Existence proof for degree 10 circulant graphs of order $L_{CC}(10,k)$ for all diameters $k \equiv 7 \pmod {15}$}

For graphs of diameter $k \equiv 2\pmod {5}$ we now consider the second subcase where $k \equiv 7 \pmod {15}$. The existence of the degree 10 circulant graph of order $L_{CC}(10,k)$ for all diameters $k \equiv 7\pmod {15}$, with generating set 1 of Table \ref {table:102A}, is proved in the following theorem.

\begin{theorem}
For all $k\equiv 7 \pmod {15}$, there is an undirected Cayley graph on five generators of a cyclic group which has diameter k and order $L_{CC}(10,k)$, where
\[ L_{CC}(10,k)=(512k^5+1280k^4+6080k^3+7840k^2+10010k+3741)/3125. \\
\]

Moreover, a generating set is
$\{1,
(128k^4+256k^3+1192k^2+1064k+883)/625,
(512k^5+640k^4+3200k^3-2520k^2-8510k-12624)/9375,
(512k^5+640k^4+5600k^3+1080k^2+4540k-6699)/9375,
(512k^5+640k^4+5600k^3+1080k^2+12040k-2949)/9375\}$.

\label{theorem:102B}
\end{theorem}

\begin{proof}
We will show the existence of five-dimensional lattices $L_k \subseteq \Z^5$ such that $\Z^5/L_k$ is cyclic, $S_{5,k}+L_k=\Z^5$, where $S_{5,k}$ is the set of points in $\Z^5$ at a distance of at most $k$ from the origin under the $\ell^1$ (Manhattan) metric, and $\vert \Z^5 : L_k\vert = L_{CC}(10,k) $ as specified in the theorem. Then, by Theorem \ref {theoremD}, the resultant Cayley graph has diameter at most $k$. 

Let $a= (2k+1)/15$, and let $L_k$ be defined by five generating vectors as follows:
\[
\begin{array}{rcl}
\textbf{v}_1&=&(3a+1,3a+1,3a,-3a+2,3a)\\
\textbf{v}_2&=&(3a,3a-1,-3a-1,3a+1,3a-1)\\
\textbf{v}_3&=&(3a,-3a,-3a+1,-3a-1,-3a)\\
\textbf{v}_4&=&(3a,3a,3a-1,3a+2,3a-1)\\
\textbf{v}_5&=&(3a-1,3a,3a-1,3a+1,-3a-1)\\
\end{array}
\] 
Then the following vectors are in $L_k$:
\newline
$18a^2\textbf{v}_1 + \textbf{v}_2 + (108a^3+12a)\textbf{v}_3 + (108a^3+6a)\textbf{v}_4 - (18a^2+1)\textbf{v}_5 = (648a^4+90a^2+1, -1, 0, 0, 0),$
\newline
$(36a^3-6a^2-a)\textbf{v}_1 + 2a\textbf{v}_2 + (216a^4-36a^3+18a^2-6a-1)\textbf{v}_3 + (216a^4-36a^3+6a^2-4a-1)\textbf{v}_4 - (36a^3-6a^2+a-1)\textbf{v}_5 = (1296a^5-216a^4+144a^3-42a^2-3a-1, 0, -1, 0, 0),$
\newline
$(36a^3-6a^2+2a)\textbf{v}_1 + 2a\textbf{v}_2 + (216a^4-36a^3+36a^2-6a)\textbf{v}_3 + (216a^4+-36a^3+24a^2-4a-1)\textbf{v}_4 - (36a^3-6a^2+4a-1)\textbf{v}_5 = (1296a^5-216a^4+252a^3-42a^2+6a-1, 0, 0, -1, 0),$
\newline
$(36a^3-6a^2+2a)\textbf{v}_1 + 2a\textbf{v}_2 + (216a^4-36a^3+36a^2-6a+1)\textbf{v}_3 + (216a^4-36a^3+24a^2-4a)\textbf{v}_4 - (36a^3-6a^2+4a-1)\textbf{v}_5 = (1296a^5-216a^4+252a^3-42a^2+12a-1, 0, 0, 0, -1).$
\newline

Hence we have $\textbf{e}_2 =  (648a^4+90a^2+1)\textbf{e}_1,  \textbf{e}_3 =  (1296a^5-216a^4+144a^3-42a^2-3a-1)\textbf{e}_1, \textbf{e}_4 =  (1296a^5-216a^4+252a^3-42a^2+6a-1)\textbf{e}_1$,  and $\textbf{e}_5 =  (1296a^5-216a^4+252a^3-42a^2+12a-1)\textbf{e}_1$ in $\Z^5/L_k$, and so $\textbf{e}_1$ generates $\Z^5/L_k$.

Also det $\left ( \begin{array} {c} \textbf{v}_1 \\ \textbf{v}_2 \\ \textbf{v}_3 \\ \textbf{v}_4 \\ \textbf{v}_5 \end {array} \right )$

$ = $ det
$\left (
\begin{array} {l r r r r}
3888a^5+648a^3+15a & 0 & 0 & 0 & 0 \\
648a^4+90a^2+1 & -1 & 0 & 0 & 0 \\
1296a^5-216a^4+144a^3-42a^2-3a-1 & 0 & -1 & 0 & 0 \\
1296a^5-216a^4+252a^3-42a^2+6a-1 & 0 & 0 & -1 & 0 \\
1296a^5-216a^4+252a^3-42a^2+12a-1 & 0 & 0 & 0 & -1  
\end {array} \right )$ 
\newline
\newline
$= 3888a^5+648a^3+15a = (512k^5+1280k^4+6080k^3+7840k^2+10010k+3741)/3125= L_{CC}(10,k)$, as in the statement of the theorem. 

Thus $\Z^5/L_k$ is isomorphic to $\Z_{L_{CC}(10,k)}$ via an isomorphism taking $\textbf{e}_1, \textbf{e}_2, \textbf{e}_3, \textbf{e}_4, \textbf{e}_5$ respectively to $1, 648a^4+90a^2+1, 1296a^5-216a^4+144a^3-42a^2-3a-1, 1296a^5-216a^4+252a^3-42a^2+6a-1, 1296a^5-216a^4+252a^3-42a^2+12a-1 $. As $a=(2k+1)/15$ this gives the generating set specified in the theorem: $\{1,
(128k^4+256k^3+1192k^2+1064k+883)/625,
(512k^5+640k^4+3200k^3-2520k^2-8510k-12624)/9375,
(512k^5+640k^4+5600k^3+1080k^2+4540k-6699)/9375,
(512k^5+640k^4+5600k^3+1080k^2+12040k-2949)/9375\}$.

It remains to show that $S_{5,k}+L_k=\Z^5$. For $k=7$ it is straightforward to show directly that $\Z_{4551}$ with generators $1, 739, 1178, 1295, 1301$ has diameter 7, and also that $\Z_{962325}$ with generators $1, 53299, 300932, 303875, 303893$ has diameter 22. So we assume $k\geq37$, so that $a \geq 5$. Now let
\[
\begin{array}{rll}
\textbf{v}_6&=       \textbf{v}_1  +\textbf{v}_2   -\textbf{v}_4&=     (3a+1,3a,-3a,-3a+1,3a)\\
\textbf{v}_7&=       \textbf{v}_1  -\textbf{v}_4   +\textbf{v}_5&=    (3a,3a+1,3a,-3a+1,-3a)\\
\textbf{v}_8&=       \textbf{v}_2  -\textbf{v}_4  +\textbf{v}_5&=     (3a-1,3a-1,-3a-1,3a,-3a-1)\\
\textbf{v}_9&=       \textbf{v}_3  +\textbf{v}_4    -\textbf{v}_5&=   (3a+1,-3a,-3a+1,-3a,3a)\\
\textbf{v}_{10}&=  \textbf{v}_1  +\textbf{v}_3   -\textbf{v}_6&=   (3a,-3a+1,3a+1,-3a,-3a)\\
\textbf{v}_{11}&=  \textbf{v}_2  +\textbf{v}_3  -\textbf{v}_6&=     (3a-1,-3a-1,-3a,3a-1,-3a-1)\\
\textbf{v}_{12}&=  \textbf{v}_3  +\textbf{v}_4  -\textbf{v}_6&=    (3a-1,-3a,3a,3a,-3a-1)\\
\textbf{v}_{13}&= -\textbf{v}_4  +\textbf{v}_5  +\textbf{v}_6&=    (3a,3a,-3a,-3a,-3a)\\
\textbf{v}_{14}&=  \textbf{v}_3  +\textbf{v}_4  -\textbf{v}_7&=     (3a,-3a-1,-3a,3a,3a-1)\\
\textbf{v}_{15}&=  \textbf{v}_3  +\textbf{v}_4  -\textbf{v}_8&=      (3a+1,-3a+1,3a+1,-3a+1,3a)\\
\textbf{v}_{16}&=  \textbf{v}_4  -\textbf{v}_6   +\textbf{v}_9&=    (3a,-3a,3a,3a+1,3a-1).
\end{array}
\]
Then the 32 vectors $\pm\textbf{v}_i$ for $i=1,...,16$ provide one element of $L_k$ lying strictly within each of the 32 orthants of $\Z^5$. Most of the coordinates of these vectors have absolute value at most $3a+1$. Only $\pm \textbf{v}_4$ have one coordinate with absolute value equal to $3a+2$.
As in the proof of Theorem \ref {theorem:100A} any vector $\textbf{x}$ is reduced to a vector $\textbf{x}'$ with coordinates with absolute value at most $3a+2$ by successively subtracting appropriate vectors $\textbf{v}=\pm \textbf{v}_i$. If a coordinate of $\textbf{x}$ is 0 then either sign is allowed for $\textbf{v}$ as long as the corresponding coordinate of $\textbf{v}$ has absolute value $\leq 3a+1$. So if the $\textbf{e}_4$ coordinate of $\textbf{x}$ is 0 then we avoid $\textbf{v}_4$ and take $\textbf{v}_1$ instead.

If the resulting vector $\textbf{x}'$ lies between $\textbf{0}$ and $\textbf{v}$, where $\textbf{v}=\pm\textbf{v}_i$ for some $i$, then we have $\delta(\textbf{0},\textbf{x}')+\delta(\textbf{x}',\textbf{v})=\delta(\textbf{0},\textbf{v})$. All of the vectors $\textbf{v}$ satisfy $\delta(\textbf{0},\textbf{v})=15a$, so we have $\delta(\textbf{0},\textbf{v})=2k+1$. Hence $\textbf{x}' \in S_{5,k}+L_k$, so that $\textbf{x} \in S_{5,k}+L_k$ as required. 
Now we are left with the case where the absolute value of each coordinate of the reduced $\textbf{x}$ is at most $3a+1$, and $\textbf{x}$ is in the orthant of $\textbf{v}$, where $\textbf{v} = \pm \textbf{v}_i$ for some $i \leq 16$ but does not lie between $\textbf{0}$ and $\textbf{v}$.
\end{proof}


\subsection {Existence proof for degree 10 circulant graphs of order $L_{CC}(10,k)$ for all diameters $k \equiv 12 \pmod {15}$}

For graphs of diameter $k \equiv 2\pmod {5}$ we now consider the third subcase where $k \equiv 12 \pmod {15}$. The existence of the degree 10 circulant graph of order $L_{CC}(10,k)$ for all diameters $k \equiv 12\pmod {15}$, with generating set 2 of Table \ref {table:102A}, is proved following the same method.

\begin{theorem}
For all $k\equiv 12 \pmod {15}$, there is an undirected Cayley graph on five generators of a cyclic group which has diameter k and order $L_{CC}(10,k)$, where
\[ L_{CC}(10,k)=(512k^5+1280k^4+6080k^3+7840k^2+10010k+3741)/3125. \\
\]

Moreover, a generating set is
$\{1,
(128k^4+256k^3+1192k^2+1064k+883)/625,
(512k^5+1920k^4+6560k^3+14600k^2+7980k+10431)/9375,
(512k^5+1920k^4+6560k^3+14600k^2+15480k+14181)/9375,
(512k^5+1920k^4+8960k^3+18200k^2+28530k+20106)/9375\}$.

\label{theorem:102C}
\end{theorem}

\begin{proof}
We will show the existence of five-dimensional lattices $L_k \subseteq \Z^5$ such that $\Z^5/L_k$ is cyclic, $S_{5,k}+L_k=\Z^5$, where $S_{5,k}$ is the set of points in $\Z^5$ at a distance of at most $k$ from the origin under the $\ell^1$ (Manhattan) metric, and $\vert \Z^5 : L_k\vert = L_{CC}(10,k) $ as specified in the theorem. Then, by Theorem \ref {theoremD}, the resultant Cayley graph has diameter at most $k$. 

Let $a= (2k+6)/15$, and let $L_k$ be defined by five generating vectors as follows:
\[
\begin{array}{rcl}
\textbf{v}_1&=&(3a-1,-3a+1,3a-1,3a,3a-2)\\
\textbf{v}_2&=&(3a-1,3a-1,3a-1,3a-1,3a-1)\\
\textbf{v}_3&=&(3a,3a,-3a+1,3a-3,-3a+1)\\
\textbf{v}_4&=&(3a-2,-3a,3a,-3a+2,3a-1)\\
\textbf{v}_5&=&(3a,3a-1,-3a+1,3a-2,3a-1)\\
\end{array}
\] 
Then the following vectors are in $L_k$:
\newline
$(18a^2-6a)\textbf{v}_1  - (18a^2-12a+3)\textbf{v}_2 + (108a^3-108a^2+42a-7)\textbf{v}_3 + (108a^3-126a^2+54a-8)\textbf{v}_4 + (10a^2-12a+4) \textbf{v}_5= (648a^4-864a^3+522a^2-156a+19, -1, 0, 0, 0),$
\newline
$(36a^3-18a^2+2a)\textbf{v}_1 - (36a^3-30a^2+10a-1)\textbf{v}_2 + (216a^4-252a^3+120a^2-26a+1)\textbf{v}_3 + (216a^4-288a^3+150a^2-32a+1)\textbf{v}_4 + (36a^3-30a^2+12a-1)\textbf{v}_5 = (1296a^5-1944a^4+1332a^3-474a^2+78a-3, 0, -1, 0, 0),$
\newline
$(36a^3-18a^2+2a)\textbf{v}_1 - (36a^3-30a^2+10a-1)\textbf{v}_2 + (216a^4-252a^3+120a^2-26a+2)\textbf{v}_3 + (216a^4-288a^3+150a^2-32a+2)\textbf{v}_4 + (36a^3-30a^2+12a-1)\textbf{v}_5 = (1296a^5-1944a^4+1332a^3-474a^2+84a-5, 0, 0, -1, 0),$
\newline
$(36a^3-18a^2+5a)\textbf{v}_1 - (36a^3-30a^2+13a-2)\textbf{v}_2 + (216a^4-252a^3+138a^2-38a+4)\textbf{v}_3 + (216a^4-288a^3+168a^2-47a+5)\textbf{v}_4 + (36a^3-30a^2+15a-2)\textbf{v}_5 = (1296a^5-1944a^4+1440a^3-582a^2+129a-12, 0, 0, 0, -1).$

Hence we have $\textbf{e}_2 =  (648a^4-864a^3+522a^2-156a+19)\textbf{e}_1,  \textbf{e}_3 =  (1296a^5-1944a^4+1332a^3-474a^2+78a-3)\textbf{e}_1, \textbf{e}_4 =  (1296a^5-1944a^4+1332a^3-474a^2+84a-5)\textbf{e}_1$,  and $\textbf{e}_5 =  (1296a^5-1944a^4+1440a^3-582a^2+129a-12)\textbf{e}_1$ in $\Z^5/L_k$, and so $\textbf{e}_1$ generates $\Z^5/L_k$.

Also det $\left ( \begin{array} {c} \textbf{v}_1 \\ \textbf{v}_2 \\ \textbf{v}_3 \\ \textbf{v}_4 \\ \textbf{v}_5 \end {array} \right )$

$ = $ det
$\left (
\begin{array} {l r r r r}
3888a^5-6480a^4+4968a^3-2088a^2+471a-45 & 0 & 0 & 0 & 0 \\
648a^4-864a^3+522a^2-156a+19 & -1 & 0 & 0 & 0 \\
1296a^5-1944a^4+1332a^3-474a^2+78a-3 & 0 & -1 & 0 & 0 \\
1296a^5-1944a^4+1332a^3-474a^2+84a-5 & 0 & 0 & -1 & 0 \\
1296a^5-1944a^4+1440a^3-582a^2+129a-12 & 0 & 0 & 0 & -1  
\end {array} \right )$ 
\newline
\newline
$= 3888a^5-6480a^4+4968a^3-2088a^2+471a-45 = (512k^5+1280k^4+6080k^3+7840k^2+10010k+3741)/3125= L_{CC}(10,k)$, as in the statement of the theorem. 

Thus $\Z^5/L_k$ is isomorphic to $\Z_{L_{CC}(10,k)}$ via an isomorphism taking $\textbf{e}_1, \textbf{e}_2, \textbf{e}_3, \textbf{e}_4, \textbf{e}_5$ respectively to $1, 648a^4-864a^3+522a^2-156a+19, 1296a^5-1944a^4+1332a^3-474a^2+78a-3, 1296a^5-1944a^4+1332a^3-474a^2+84a-5, 1296a^5-1944a^4+1440a^3-582a^2+129a-12 $. As $a=(2k+6)/15$ this gives the generating set specified in the theorem: $\{1,
(128k^4+256k^3+1192k^2+1064k+883)/625,
(512k^5+1920k^4+6560k^3+14600k^2+7980k+10431)/9375,
(512k^5+1920k^4+6560k^3+14600k^2+15480k+14181)/9375,
(512k^5+1920k^4+8960k^3+18200k^2+28530k+20106)/9375\}$.

It remains to show that $S_{5,k}+L_k=\Z^5$. For $k=12$, it is straightforward to show directly that $\Z_{53025}$ with generators $1, 5251, 19281, 19291, 19806$ has diameter 12. So we assume $k\geq27$, so that $a \geq 4$. Now let
\[
\begin{array}{rll}
\textbf{v}_6&=       \textbf{v}_1  -\textbf{v}_2   +\textbf{v}_3&=     (3a,-3a+2,-3a+1,3a-2,-3a)\\
\textbf{v}_7&=      -\textbf{v}_1  +\textbf{v}_2   +\textbf{v}_4&=    (3a-2,3a-2,3a,-3a+1,3a)\\
\textbf{v}_8&=       \textbf{v}_1  -\textbf{v}_2  +\textbf{v}_5&=     (3a,-3a+1,-3a+1,3a-1,3a-2)\\
\textbf{v}_9&=      -\textbf{v}_2  +\textbf{v}_3    +\textbf{v}_4&=   (3a-1,-3a+1,-3a+2,-3a,-3a+1)\\
\textbf{v}_{10}&=  -\textbf{v}_2  +\textbf{v}_4   +\textbf{v}_5&=   (3a-1,-3a,-3a+2,-3a+1,3a-1)\\
\textbf{v}_{11}&=  \textbf{v}_3  +\textbf{v}_4  -\textbf{v}_5&=     (3a-2,-3a+1,3a,-3a+1,-3a+1)\\
\textbf{v}_{12}&=  \textbf{v}_2  -\textbf{v}_5  +\textbf{v}_6&=    (3a-1,-3a+2,3a-1,3a-1,-3a)\\
\textbf{v}_{13}&= -\textbf{v}_2  +\textbf{v}_3  +\textbf{v}_7&=    (3a-1,3a-1,-3a+2,-3a-1,-3a+2)\\
\textbf{v}_{14}&=  -\textbf{v}_2  +\textbf{v}_5  +\textbf{v}_7&=     (3a-1,3a-2,-3a+2,-3a,3a)\\
\textbf{v}_{15}&=  \textbf{v}_3  -\textbf{v}_5  +\textbf{v}_7&=      (3a-2,3a-1,3a,-3a,-3a+2)\\
\textbf{v}_{16}&=  \textbf{v}_1  +\textbf{v}_3   -\textbf{v}_8&=    (3a-1,3a,3a-1,3a-2,-3a+1).
\end{array}
\]
Then the 32 vectors $\pm\textbf{v}_i$ for $i=1,...,16$ provide one element of $L_k$ lying strictly within each of the 32 orthants of $\Z^5$. Most of the coordinates of these vectors have absolute value at most $3a$. Only $\pm \textbf{v}_{13}$ have one coordinate with absolute value equal to $3a+1$.
As in the proof of Theorem \ref {theorem:100A} any vector $\textbf{x}$ is reduced to a vector $\textbf{x}'$ with coordinates with absolute value at most $3a$ by successively subtracting appropriate vectors $\textbf{v}=\pm \textbf{v}_i$.
If a coordinate of $\textbf{x}$ is 0 then either sign is allowed for $\textbf{v}$ as long as the corresponding coordinate of $\textbf{v}$ has absolute value $\leq 3a$. So if the $\textbf{e}_4$ coordinate of $\textbf{x}$ is 0 then we avoid $\textbf{v}_{13}$ and take $\textbf{v}_3$ instead.

If the resulting vector $\textbf{x}'$ lies between $\textbf{0}$ and $\textbf{v}$, where $\textbf{v}=\pm\textbf{v}_i$ for some $i$, then we have $\delta(\textbf{0},\textbf{x}')+\delta(\textbf{x}',\textbf{v})=\delta(\textbf{0},\textbf{v})$. All of the vectors $\textbf{v}$ satisfy $\delta(\textbf{0},\textbf{v})=15a-5$, so we have $\delta(\textbf{0},\textbf{v})=2k+1$. Hence $\textbf{x}' \in S_{5,k}+L_k$, so that $\textbf{x} \in S_{5,k}+L_k$ as required. 
Now we are left with the case where the absolute value of each coordinate of the reduced $\textbf{x}$ is at most $3a$, and $\textbf{x}$ is in the orthant of $\textbf{v}$, where $\textbf{v} = \pm \textbf{v}_i$ for some $i \leq 16$ but does not lie between $\textbf{0}$ and $\textbf{v}$.
\end{proof}


\subsection {Existence proof for degree 10 circulant graphs of order $L_{CC}(10,k)$ for all diameters $k \equiv 3 \pmod 5$}

The existence of the degree 10 circulant graph of order $L_{CC}(10,k)$ for all diameters $k \equiv 3\pmod 5$, with generating set 1 of Table \ref {table:103A}, is proved following the same method.

\begin{theorem}
For all $k\equiv 3 \pmod 5$, there is an undirected Cayley graph on five generators of a cyclic group which has diameter k and order $L_{CC}(10,k)$, where
\[ L_{CC}(10,k)=(512k^5+1280k^4+6560k^3+7600k^2+4180k+1344)/3125. \\
\]

Moreover, a generating set is
$\{1,
(128k^4+64k^3+1112k^2-1224k-1557)/625,
(256k^4+448k^3+2944k^2+1592k+271)/625,
(256k^4+608k^3+3104k^2+3062k+726)/625,
(384k^4+672k^3+4216k^2+2338k-456)/625\}$.

\label{theorem:103A}
\end{theorem}

\begin{proof}
We will show the existence of five-dimensional lattices $L_k \subseteq \Z^5$ such that $\Z^5/L_k$ is cyclic, $S_{5,k}+L_k=\Z^5$, where $S_{5,k}$ is the set of points in $\Z^5$ at a distance of at most $k$ from the origin under the $\ell^1$ (Manhattan) metric, and $\vert \Z^5 : L_k\vert = L_{CC}(10,k) $ as specified in the theorem. Then, by Theorem \ref {theoremD}, the resultant Cayley graph has diameter at most $k$. 

Let $a= (2k-1)/5$, and let $L_k$ be defined by five generating vectors as follows:
\[
\begin{array}{rcl}
\textbf{v}_1&=&(a+2,-a,a-1,-a,a+1)\\
\textbf{v}_2&=&(a-1,a,a+2,a,-a-1)\\
\textbf{v}_3&=&(a+1,a+1,-a-1,-a+1,a)\\
\textbf{v}_4&=&(a,a,a+1,-a-1,-a)\\
\textbf{v}_5&=&(a+1,-a+1,-a-1,a+2,a-1)\\
\end{array}
\] 
Then the following vectors are in $L_k$:
\newline
$(2a^2+2a-1)\textbf{v}_1 - 2a^2\textbf{v}_2 - 2a\textbf{v}_3 + (4a^3+2a^2+4a-2)\textbf{v}_4 + (4a^3+2a^2+2a-1)\textbf{v}_5 = (8a^4+8a^3+14a^2-3, -1, 0, 0, 0),$
\newline
$(4a^2+6a+1)\textbf{v}_1 - (4a^2+2a+1)\textbf{v}_2 - (4a+2)\textbf{v}_3 + (8a^3+8a^2+12a+2)\textbf{v}_4 + (8a^3+8a^2+8a+2)\textbf{v}_5 = (16a^4+24a^3+40a^2+20a+3, 0, -1, 0, 0),$
\newline
$(4a^2+7a+2)\textbf{v}_1 - (4a^2+3a+1)\textbf{v}_2 - (4a+3)\textbf{v}_3 + (8a^3+10a^2+13a+4)\textbf{v}_4 + (8a^3+10a^2+9a+3)\textbf{v}_5 = (16a^4+28a^3+44a^2+27a+5, 0, 0, -1, 0),$
\newline
$(6a^2+9a+1)\textbf{v}_1 - (6a^2+3a+1)\textbf{v}_2 - (6a+3)\textbf{v}_3 + (12a^3+12a^2+17a+3)\textbf{v}_4 + (12a^3+12a^2+11a+3)\textbf{v}_5 = (24a^4+36a^3+58a^2+29a+3, 0, 0, 0, -1)$
\newline

Hence we have $\textbf{e}_2 =  (8a^4+8a^3+14a^2-3)\textbf{e}_1,  \textbf{e}_3 =  (16a^4+24a^3+40a^2+20a+3)\textbf{e}_1, \textbf{e}_4 =  (16a^4+28a^3+44a^2+27a+5)\textbf{e}_1$,     and $\textbf{e}_5 =  (24a^4+36a^3+58a^2+29a+3)\textbf{e}_1$ in $\Z^5/L_k$, and so $\textbf{e}_1$ generates $\Z^5/L_k$.

Also det $\left ( \begin{array} {c} \textbf{v}_1 \\ \textbf{v}_2 \\ \textbf{v}_3 \\ \textbf{v}_4 \\ \textbf{v}_5 \end {array} \right ) = $ det
$\left (
\begin{array} {l r r r r}
16a^5+32a^4+52a^3+40a^2+14a+2 & 0 & 0 & 0 & 0 \\
8a^4+8a^3+14a^2-3 & -1 & 0 & 0 & 0 \\
16a^4+24a^3+40a^2+20a+3 & 0 & -1 & 0 & 0 \\
16a^4+28a^3+44a^2+27a+5 & 0 & 0 & -1 & 0 \\
24a^4+36a^3+58a^2+29a+3 & 0 & 0 & 0 & -1  
\end {array} \right )$ 
\newline
\newline
$= 16a^5+32a^4+52a^3+40a^2+14a+2 = (512k^5+1280k^4+6560k^3+7600k^2+4180k+1344)/3125= L_{CC}(10,k)$, as in the statement of the theorem. 

Thus $\Z^5/L_k$ is isomorphic to $\Z_{L_{CC}(10,k)}$ via an isomorphism taking $\textbf{e}_1, \textbf{e}_2, \textbf{e}_3, \textbf{e}_4, \textbf{e}_5$ respectively to $1, 8a^4+8a^3+14a^2-3, 16a^4+24a^3+40a^2+20a+3, 16a^4+28a^3+44a^2+27a+5, 24a^4+36a^3+58a^2+29a+3 $. As $a=(2k-1)/5$ this gives the generating set specified in the theorem: $\{1,
(128k^4+64k^3+1112k^2-1224k-1557)/625,
(256k^4+448k^3+2944k^2+1592k+271)/625,
(256k^4+608k^3+3104k^2+3062k+726)/625,
(384k^4+672k^3+4216k^2+2338k-456)/625\}$.

It remains to show that $S_{5,k}+L_k=\Z^5$. For $k=8$ it is straightforward to show directly that $\Z_{8288}$ with generators $1, 987, 2367, 2534, 3528$ has diameter 8. So we assume $k\geq13$, so that $a \geq 5$. Now let
\[
\begin{array}{rll}
\textbf{v}_6&=      \textbf{v}_1  +\textbf{v}_2  -\textbf{v}_4&=     (a+1,-a,a,a+1,a)\\
\textbf{v}_7&=      \textbf{v}_2  +\textbf{v}_3   -\textbf{v}_4&=    (a,a+1,-a,a+2,a-1)\\
\textbf{v}_8&=     -\textbf{v}_2  +\textbf{v}_4  +\textbf{v}_5&=    (a+2,-a+1,-a-2,-a+1,a)\\
\textbf{v}_9&=      -\textbf{v}_3  +\textbf{v}_4   +\textbf{v}_5&=   (a,-a,a+1,a,-a-1)\\
\textbf{v}_{10}&= \textbf{v}_3  -\textbf{v}_5   +\textbf{v}_6&=    (a+1,a,a,-a,a+1)\\
\textbf{v}_{11}&= \textbf{v}_4 +\textbf{v}_5  -\textbf{v}_6&=     (a,a+1,-a,-a,-a-1)\\
\textbf{v}_{12}&= \textbf{v}_4 +\textbf{v}_5   -\textbf{v}_7&=    (a+1,-a,a,-a-1,-a)\\
\textbf{v}_{13}&=-\textbf{v}_1 +\textbf{v}_2   +\textbf{v}_8&=    (a-1,a+1,-a+1,a+1,-a-2)\\
\textbf{v}_{14}&= \textbf{v}_1  +\textbf{v}_7  -\textbf{v}_8&=     (a,a,a+1,a+1,a)\\
\textbf{v}_{15}&=-\textbf{v}_1 +\textbf{v}_8  +\textbf{v}_9&=     (a,-a+1,-a,a+1,-a-2)\\
\textbf{v}_{16}&=-\textbf{v}_6  +\textbf{v}_8   +\textbf{v}_9&=  (a+1,-a+1,-a-1,-a,-a-1).
\end{array}
\]
Then the 32 vectors $\pm\textbf{v}_i$ for $i=1,...,16$ provide one element of $L_k$ lying strictly within each of the 32 orthants of $\Z^5$. Most of the coordinates of these vectors have absolute value at most $a+1$. However vectors $\pm \textbf{v}_1, \pm \textbf{v}_2, \pm \textbf{v}_5, \pm \textbf{v}_7, \pm \textbf{v}_8$ and $\pm \textbf{v}_{13}$ each have one coordinate with absolute value equal to $a+2$.
As in the proof of Theorem \ref {theorem:100A} any vector $\textbf{x}$ is reduced to a vector $\textbf{x}'$ with coordinates with absolute value at most $3a$ by successively subtracting appropriate vectors $\textbf{v}=\pm \textbf{v}_i$.
If a coordinate of $\textbf{x}$ is 0 then either sign is allowed for $\textbf{v}$ as long as the corresponding coordinate of $\textbf{v}$ has absolute value $\leq a+1$.
So if $\textbf{x}$ lies in the orthant of $\textbf{v}_1$ and its $\textbf{e}_1$ coordinate is 0 then we take $-\textbf{v}_{11}$ instead.
If $\textbf{x}$ lies in the orthant of $\textbf{v}_2$ and its $\textbf{e}_3$ coordinate is 0 then we take $\textbf{v}_{11}$ instead.
If $\textbf{x}$ lies in the orthant of $\textbf{v}_7$ and its $\textbf{e}_4$ coordinate is 0 then we take $\textbf{v}_3$ instead.
If $\textbf{x}$ lies in the orthant of $\textbf{v}_{13}$ and its $\textbf{e}_5$ coordinate is 0 then we take $\textbf{v}_7$ instead, unless its $\textbf{e}_4$ coordinate is also 0, in which case we take $\textbf{v}_3$ instead, as above.
If $\textbf{x}$ lies in the orthant of $\textbf{v}_8$ and its $\textbf{e}_3$ coordinate is 0 then we consider $\textbf{v}_1$ instead, with the provisos stated above.
Finally if $\textbf{x}$ lies in the orthant of $\textbf{v}_5$ and its $\textbf{e}_4$ coordinate is 0 then we consider $\textbf{v}_8$ instead, with the provisos stated above.

If the resulting vector $\textbf{x}'$ lies between $\textbf{0}$ and $\textbf{v}$, where $\textbf{v}=\pm\textbf{v}_i$ for some $i$, then we have $\delta(\textbf{0},\textbf{x}')+\delta(\textbf{x}',\textbf{v})=\delta(\textbf{0},\textbf{v})$. All of the vectors $\textbf{v}$ satisfy $\delta(\textbf{0},\textbf{v})=5a+2$, so we have $\delta(\textbf{0},\textbf{v})=2k+1$. Hence $\textbf{x}' \in S_{5,k}+L_k$, so that $\textbf{x} \in S_{5,k}+L_k$ as required. 
Now we are left with the case where the absolute value of each coordinate of the reduced $\textbf{x}$ is at most $a+1$, and $\textbf{x}$ is in the orthant of $\textbf{v}$, where $\textbf{v} = \pm \textbf{v}_i$ for some $i \leq 16$ but does not lie between $\textbf{0}$ and $\textbf{v}$.
\end{proof}


\subsection {Existence proof for degree 10 circulant graphs of order $L_{CC}(10,k)$ for all diameters $k \equiv 4 \pmod 5$}

The existence of the degree 10 circulant graph of order $L_{CC}(10,k)$ for all diameters $k \equiv 4\pmod 5$, with generating set 1 of Table \ref {table:104A}, is proved following the same method.

\begin{theorem}
For all $k\equiv 4 \pmod 5$, there is an undirected Cayley graph on five generators of a cyclic group which has diameter k and order $L_{CC}(10,k)$, where
\[ L_{CC}(10,k)=(512k^5+1280k^4+6400k^3+8640k^2+6890k+757)/3125. \\
\]

Moreover, a generating set is
$\{1,
(32k^3+56k^2+316k+167)/125,
(128k^4+192k^3+1408k^2+1002k+908)/625,
(128k^4+352k^3+1488k^2+2432k+543)/625,
(256k^4+384k^3+2616k^2+1354k+116)/625\}$.

\label{theorem:104A}
\end{theorem}

\begin{proof}
We will show the existence of five-dimensional lattices $L_k \subseteq \Z^5$ such that $\Z^5/L_k$ is cyclic, $S_{5,k}+L_k=\Z^5$, where $S_{5,k}$ is the set of points in $\Z^5$ at a distance of at most $k$ from the origin under the $\ell^1$ (Manhattan) metric, and $\vert \Z^5 : L_k\vert = L_{CC}(10,k) $ as specified in the theorem. Then, by Theorem \ref {theoremD}, the resultant Cayley graph has diameter at most $k$. 

Let $a= (2k+2)/5$, and let $L_k$ be defined by five generating vectors as follows:
\[
\begin{array}{rcl}
\textbf{v}_1&=&(a-1,a,-a,a,a)\\
\textbf{v}_2&=&(a,a-1,-a,a,-a)\\
\textbf{v}_3&=&(a,a,a,a-1,a)\\
\textbf{v}_4&=&(-a+1,-a,a+1,a-1,a)\\
\textbf{v}_5&=&(-a-1,a-1,-a+1,a+1,a-1)\\
\end{array}
\] 
Then the following vectors are in $L_k$:
\newline
$(2a^2-2a+1)\textbf{v}_1 + (2a+1)\textbf{v}_2 + 2a\textbf{v}_4 - 2a^2\textbf{v}_5 = (4a^3-2a^2+6a-1, -1, 0, 0, 0),$
\newline
$(4a^3-4a^2+2a)\textbf{v}_1 + (4a^2+2a)\textbf{v}_2 + \textbf{v}_3 +(4a^2-1)\textbf{v}_4 -4a^3\textbf{v}_5 = (8a^4-4a^3+12a^2-1, 0, -1, 0, 0),$
\newline
$(4a^3-6a^2+5a-2)\textbf{v}_1 + 4a^2\textbf{v}_2 + \textbf{v}_3 +(4a^2-2a)\textbf{v}_4 - (4a^3-2a^2+a)\textbf{v}_5 = (8a^4-8a^3+16a^2-7a+2, 0, 0, -1, 0),$
\newline
$(8a^3-12a^2+9a-4)\textbf{v}_1 + (8a^2-1)\textbf{v}_2 + 2\textbf{v}_3 +(8a^2-4a-1)\textbf{v}_4 - (8a^3-4a^2+a-1)\textbf{v}_5 = (16a^4-16a^3+30a^2-15a+2, 0, 0, 0, -1)$
\newline

Hence we have $\textbf{e}_2 =  (4a^3-2a^2+6a-1)\textbf{e}_1,  \textbf{e}_3 =  (8a^4-4a^3+12a^2-1)\textbf{e}_1, \textbf{e}_4 =  (8a^4-8a^3+16a^2-7a+2)\textbf{e}_1$,     and $\textbf{e}_5 =  (16a^4-16a^3+30a^2-15a+2)\textbf{e}_1$ in $\Z^5/L_k$, and so $\textbf{e}_1$ generates $\Z^5/L_k$.

Also det $\left ( \begin{array} {c} \textbf{v}_1 \\ \textbf{v}_2 \\ \textbf{v}_3 \\ \textbf{v}_4 \\ \textbf{v}_5 \end {array} \right ) = $ det
$\left (
\begin{array} {l r r r r}
16a^5-16a^4+32a^3-16a^2+5a-1 & 0 & 0 & 0 & 0 \\
4a^3-2a^2+6a-1 & -1 & 0 & 0 & 0 \\
8a^4-4a^3+12a^2-1 & 0 & -1 & 0 & 0 \\
8a^4-8a^3+16a^2-7a+2 & 0 & 0 & -1 & 0 \\
16a^4-16a^3+30a^2-15a+2 & 0 & 0 & 0 & -1  
\end {array} \right )$ 
\newline
\newline
$= 16a^5-16a^4+32a^3-16a^2+5a-1 = (512k^5+1280k^4+6400k^3+8640k^2+6890k+757)/3125= L_{CC}(10,k)$, as in the statement of the theorem. 

Thus $\Z^5/L_k$ is isomorphic to $\Z_{L_{CC}(10,k)}$ via an isomorphism taking $\textbf{e}_1, \textbf{e}_2, \textbf{e}_3, \textbf{e}_4, \textbf{e}_5$ respectively to $1, 4a^3-2a^2+6a-1, 8a^4-4a^3+12a^2-1,  8a^4-8a^3+16a^2-7a+2, 16a^4-16a^3+30a^2-15a+2 $. As $a=(2k+2)/5$ this gives the generating set specified in the theorem: $\{1,
(32k^3+56k^2+316k+167)/125,
(128k^4+192k^3+1408k^2+1002k+908)/625,
(128k^4+352k^3+1488k^2+2432k+543)/625,
(256k^4+384k^3+2616k^2+1354k+116)/625\}$.

It remains to show that $S_{5,k}+L_k=\Z^5$. For $k=4$, it is straightforward to show directly that $\Z_{457}$ with generators $1, 20, 130, 147, 191$ has diameter 4. So we assume $k\geq9$, so that $a \geq 4$. Now let
\[
\begin{array}{rll}
\textbf{v}_6&=-\textbf{v}_1+\textbf{v}_2+\textbf{v}_3&=(a+1,a-1,a,a-1,-a)\\
\textbf{v}_7&=\textbf{v}_1-\textbf{v}_2-\textbf{v}_4&=(a-2,a+1,-a-1,-a+1,a)\\
\textbf{v}_8&=\textbf{v}_1-\textbf{v}_2-\textbf{v}_5&=(a,-a+2,a-1,-a-1,a+1)\\
\textbf{v}_9&=\textbf{v}_1-\textbf{v}_3+\textbf{v}_4&=(a,a,a-1,-a,-a)\\
\textbf{v}_{10}&=\textbf{v}_1-\textbf{v}_3-\textbf{v}_5&=(a,-a+1,-a-1,-a,-a+1)\\
\textbf{v}_{11}&=\textbf{v}_1+\textbf{v}_4-\textbf{v}_5&=(a+1,-a+1,a,a-2,a+1)\\
\textbf{v}_{12}&=-\textbf{v}_2+\textbf{v}_3-\textbf{v}_4&=(a-1,a+1,a-1,-a,a)\\
\textbf{v}_{13}&=\textbf{v}_2+\textbf{v}_4-\textbf{v}_5&=(a+2,-a,a,a-2,-a+1)\\
\textbf{v}_{14}&=\textbf{v}_2-\textbf{v}_5+\textbf{v}_9&=(a+1,-a,-a,a-1,-a+1)\\
\textbf{v}_{15}&=\textbf{v}_1-\textbf{v}_5+\textbf{v}_9&=(a,-a+1,-a,a-1,a+1)\\
\textbf{v}_{16}&=\textbf{v}_1-\textbf{v}_3+\textbf{v}_8&=(a-1,-a+2,-a-1,-a,a+1).
\end{array}
\]
Then the 32 vectors $\pm\textbf{v}_i$ for $i=1,...,16$ provide one element of $L_k$ lying strictly within each of the 32 orthants of $\Z^5$. Most of the coordinates of these vectors have absolute value at most $a+1$. Only $\pm \textbf{v}_{13}$ have one coordinate with absolute value equal to $a+2$.
As in the proof of Theorem \ref {theorem:100A} any vector $\textbf{x}$ is reduced to a vector $\textbf{x}'$ with coordinates with absolute value at most $3a$ by successively subtracting appropriate vectors $\textbf{v}=\pm \textbf{v}_i$.
If a coordinate of $\textbf{x}$ is 0 then either sign is allowed for $\textbf{v}$ as long as the corresponding coordinate of $\textbf{v}$ has absolute value $\leq a+1$. So if the $\textbf{e}_1$ coordinate of $\textbf{x}$ is 0 then we avoid $\textbf{v}_{13}$ and take $-\textbf{v}_7$ instead.

If the resulting vector $\textbf{x}'$ lies between $\textbf{0}$ and $\textbf{v}$, where $\textbf{v}=\pm\textbf{v}_i$ for some $i$, then we have $\delta(\textbf{0},\textbf{x}')+\delta(\textbf{x}',\textbf{v})=\delta(\textbf{0},\textbf{v})$. All of the vectors $\textbf{v}$ satisfy $\delta(\textbf{0},\textbf{v})=5a-1$, so we have $\delta(\textbf{0},\textbf{v})=2k+1$. Hence $\textbf{x}' \in S_{5,k}+L_k$, so that $\textbf{x} \in S_{5,k}+L_k$ as required. 
Now we are left with the case where the absolute value of each coordinate of the reduced $\textbf{x}$ is at most $a+1$, and $\textbf{x}$ is in the orthant of $\textbf{v}$, where $\textbf{v} = \pm \textbf{v}_i$ for some $i \leq 16$ but does not lie between $\textbf{0}$ and $\textbf{v}$.
\end{proof}


\section {Conclusion}

As we have seen, the formulae for the order of the largest-known degree 10 and 11 circulant graphs are quintic polynomials in the diameter $k$ which depend on the value of $k \pmod 5$. This divides the solution space into five families for degree 10 and another five for degree 11. For each of these 10 families there is a unique isomorphism class with the exception of degree 11, diameter $k \equiv 1 \pmod 5$ for which there are two. For eight of these eleven cases there is also at least one family of primitive generating sets which is valid for every diameter, so that it repeats every 5. The other families require at least two sets or subsets to cover all diameters. For example the degree 10 familiy for diameter 2 (mod 5) requires both generating sets 1 and 2, and degree 11 family 1a requires all five subsets of generating set 1. Table \ref {table:gensets} summarises the primitive generating set families for each isomorphism class of graph families. Generating set families that are valid for every diameter of the class are represented by 1/5 where $c/b$ means that there are $c$ sets of formulae each valid for a particular $k \pmod b$.

\begin {table} [!htbp]
\small
\caption{Characterisation of primitive generating sets for each isomorphism family.} 
\centering 
\setlength {\tabcolsep} {11pt}
\begin{tabular}{ @ { } c c c c c c c } 
\noalign {\vskip 2mm} 
\hline\hline 
\noalign {\vskip 1mm} 
Isomorphism & Generating sets & \multicolumn {5} {l} {Characterisation of generating set*}  \\ 
family & per diameter & Set 1 & Set 2 & Set 3 & Set 4 & Set 5   \\ 
\hline 
\noalign {\vskip 1mm} 
\multicolumn {6} {l} {Degree 10}  \\ 
0 &   3 - 5 & 1/5 & 1/5 & 3/15 & 4/25 & 16/85  \\
1 & 1 & 1/5  \\
2 & 1 - 2 & 2/15 & 2/15  \\
3 & 1 & 1/5  \\
4 & 3 - 5 & 1/5 & 1/5 & 3/15 & 4/25 & 16/85 \\
\hline
\noalign {\vskip 1mm} 
\multicolumn {6} {l} {Degree 11}  \\ 
0 &   1 - 2 & 2/15 & 2/15   \\
1a & 1 - 4 & 5/25 & 4/25 & 12/65 & 16/85  \\
1b & 2 - 4 & 1/5 & 1/5 & 6/35 & 12/65  \\
2 & 3 - 5 & 1/5 & 1/5 & 3/15 & 4/25 & 16/85  \\
3 & 3 - 5 & 1/5 & 1/5 & 3/15 & 4/25 & 16/85  \\
4 & 2 - 5 & 1/5 & 1/5 & 16/85 & 22/115 & 46/235 \\
\hline
\noalign {\vskip 1mm} 
\multicolumn {7} {l} {\footnotesize * $c/b$: a generating set with $c$ sets of formulae each valid for a particular diameter $k\pmod b$} 
\end{tabular}
\label{table:gensets}
\end{table}

These families of undirected circulant graphs of dimension 5, degrees 10 and 11, are conjectured to be extremal for all diameters above a threshold. The existence of the degree 10 graphs has been proved for all diameters, and the existence of the degree 11 graphs may be proved similarly. They are consistent with Conjecture \ref {conjecture:ext}, that for any degree $d \ge 2$ and corresponding dimension $f$, the first two coefficients of the polynomial formula in the diameter $k$ for the order of an extremal family of Abelian Cayley or circulant graphs are a multiple $R_f= 2^{f-1}(f!/f^f)$ of the respective coefficients of the Abelian Cayley graph upper bound $M_{AC}(d,k)$.

If the conjecture is valid, then for dimension 6 and 7 this would give the following formulae for extremal circulant graph order for degree $d$ and arbitrary diameter $k$ above some threshold value:
\[ CC(d,k) =
\begin {cases}
(32k^6+96k^5)/729 \ + O(k^4) & \mbox { for } d=12 \\
\ 64k^6/729 \qquad \qquad + O(k^4) & \mbox { for } d=13 \\
(8192k^7+28672k^6)/823543 + O(k^5) & \mbox { for } d=14 \\
16384k^7/823543 \qquad \qquad \quad + O(k^5) & \mbox { for } d=15 \\
\end {cases}
\]

It remains to be confirmed that any such graph families exist. Establishing the extremal Abelian Cayley and circulant graph order conjecture, Conjecture \ref{conjecture:ext}, as either a lower bound or an upper bound would also represent an important advance in this field.



\newpage

\section{Appendix: Additional tables of generating sets}


Degree 10, diameter $k\equiv 0$

\begin {table} [!htbp]
\small
\caption{Generating set 5 for degree 10 graphs of class 0: diameter $k\equiv 0 \pmod 5$, part one.} 
\centering 
\setlength {\tabcolsep} {5pt}
\begin{tabular}{ @ { } c l c c c c r r l l l c c c r r l} 
\noalign {\vskip 2mm} 
\hline\hline 
\noalign {\vskip 1mm} 
Generator & \multicolumn {8} {l} {Gen set 5, $k\equiv 0 \pmod {85}$} && \multicolumn {7} {l} {Gen set 5, $k\equiv 5 \pmod {85}$} \\ 
\hline 
\noalign {\vskip 1mm} 
$g_1$ &   ( 0 & 0 & 0 & 0 & 0 & 1 ) & & & & ( 0 & 0 & 0 & 0 & 0 & 1 ) \\
$g_2$ &  ( 2 & -2 & 5 & -33 & -27 & -17 ) &/17 & & & ( 1 & 12 & 32 & 84 & 95 & 20  ) &/34 \\
$g_3$ &  ( 3 & 14 & 33 & 78 & 36 & 17 ) &/17 & & & ( 3 & 2 & 28 & 14 & 81 & 60  ) & /34 \\
$g_4$ &  ( 7 & 10 & 60 & 46 & 101 & 68 )& /34 & & & ( 5 & 26 & 58 & 148 & 67 & 32 ) & /34 \\
$g_5$ &  ( 4 & 13 & 44 & 70 & 65 & 17 ) & /17 & & & ( 4 & 14 & 43 & 81 & 57 & 29 ) & /17 \\
\hline 
\noalign {\vskip 1mm} 
 & \multicolumn {8} {l} {Gen set 5, $k\equiv 10 \pmod {85}$} && \multicolumn {7} {l} {Gen set 5, $k\equiv 15 \pmod {85}$} \\ 
\hline 
\noalign {\vskip 1mm} 
$g_1$ &   ( 0 & 0 & 0 & 0 & 0 & 1 ) & & & & ( 0 & 0 & 0 & 0 & 0 & 1 ) \\
\end{tabular}
\label{table:100B}
\end{table}

\begin {table} [!htbp]
\small
\caption{Generating set 5 for degree 10 graphs of class 0: diameter $k\equiv 0 \pmod 5$, part two.} 
\centering 
\setlength {\tabcolsep} {4pt}
\begin{tabular}{ @ { } c l c c c c r r l l l c c c r r l} 
\noalign {\vskip 2mm} 
\hline\hline 

\noalign {\vskip 1mm} 
Generator & \multicolumn {8} {l} {Gen set 5, $k\equiv 65 \pmod {85}$} && \multicolumn {7} {l} {Gen set 5, $k\equiv 70 \pmod {85}$} \\ 
\hline 
\noalign {\vskip 1mm} 
$g_1$ &   ( 0 & 0 & 0 & 0 & 0 & 1 ) & & & & ( 0 & 0 & 0 & 0 & 0 & 1 ) \\
$g_2$ &  ( 1 & -2 & 12 & -2 & 71 & 56 ) & /34 & & & ( 1 & -6 & -1 & -46 & -16 & -9 ) & /17 \\
$g_3$ &  ( 3 & -6 & 2 & -74 & -59 & -36 ) & /34 & & & ( 3 & 16 & 48 & 100 & 105 & 24 ) & /34 \\
$g_4$ &  ( 2 & -4 & 7 & -38 & -11 & -7 )& /17 & & & ( 7 & 26 & 78 & 154 & 109 & 56 ) & /34 \\
$g_5$ &  ( 3 & -11 & 36 & 62 & 60 & 15 ) & /17 & & & ( 4 & 10 & 30 & 37 & -13 & -19 ) & /17 \\
\hline

\noalign {\vskip 1mm} 
 & \multicolumn {8} {l} {Gen set 5, $k\equiv 75 \pmod {85}$} && \multicolumn {7} {l} {Gen set 5, $k\equiv 80 \pmod {85}$} \\ 
\hline 
\noalign {\vskip 1mm} 
$g_1$ &   ( 0 & 0 & 0 & 0 & 0 & 1 ) & & & & ( 0 & 0 & 0 & 0 & 0 & 1 ) \\
$g_2$ &  ( 1 & -10 & -14 & -90 & -69 & -40 ) & /34 & & & ( 1 & -14 & -10 & -100 & -37 & -20 ) & /34 \\
$g_3$ &  ( 1 &  7 & 20 & 46 & 50 & 11 ) & /17 & & & ( 2 & 6 & 14 & 21 & -23 & -23 ) & /17 \\
$g_4$ &  ( 3 &  4 & 26 & 19 & 48 & 33 ) & /17 & & & ( 5 & -2 & 18 & -58 & -49 & -32 ) & /34 \\
$g_5$ &  ( 7 &  -2 & 38 & -52 & -14 & -8 ) & /34 & & & ( 7 & 4 & 32 & -20 & -55 & -4 ) & /34 \\
\hline

\noalign {\vskip 1mm} 
&\multicolumn {16} {l} {\footnotesize $( c_5\:  c_4\:  c_3\:  c_2\:  c_1\:  c_0)/b=(c_5a^5+c_4a^4+c_3a^3+c_2a^2+c_1a+c_0)/b$ where $a=4k/5$} 
\end{tabular}
\label{table:100C}
\end{table}


\newpage
Degree 10, diameter $k\equiv 4$

\begin {table} [!htbp]
\small
\caption{Generating sets 1 to 3 for degree 10 graphs of class 4: diameter $k\equiv 4 \pmod 5$.} 
\centering 
\setlength {\tabcolsep} {5pt}
\begin{tabular}{ @ { } c l c c c c r l l l l c c c c r l} 
\noalign {\vskip 2mm} 
\hline\hline 
\noalign {\vskip 1mm} 
Generator & \multicolumn {8} {l} {Gen set 1, $k\equiv 4 \pmod {5}$} && \multicolumn {7} {l} {Gen set 2, $k\equiv 4 \pmod {5}$} \\ 
\hline 
\noalign {\vskip 1mm} 
$g_1$ &  ( 0 & 0 & 0 & 0 & 0 & 1 ) & & & & ( 0 & 0 & 0 & 0 & 0 & 1 ) \\
$g_2$ &  ( 0 & 0 & 1 & -1 & 6 & -2 ) & /2 & & & ( 0 & 4 & -5 & 28 & -11 & 11 ) \\
$g_3$ &  ( 0 &  1 & -2 & 8 & -7 & 4 ) & /2 & & & ( 0 & 10 & -13 & 71 & -31 & 29 ) \\
$g_4$ &  ( 0 & 1 & -1 & 6 & 0 & -2 ) & /2 & & & ( 0 & 10 & -13 & 71 & -30 & 29 ) \\
$g_5$ &  ( 0 & 2 & -4 & 15 & -15 & 4 ) & /2 & & & ( 0 & 13 & -17 & 92 & -40 & 37 ) \\
\hline
\noalign {\vskip 1mm} 
 & \multicolumn {8} {l} {Gen set 3, $k\equiv 4 \pmod {15}$} && \multicolumn {7} {l} {Gen set 3, $k\equiv 9 \pmod {15}$} \\ 
\hline 
\noalign {\vskip 1mm} 
$g_1$ &   ( 0 & 0 & 0 & 0 & 0 & 1 ) & & & & ( 0 & 0 & 0 & 0 & 0 & 1 ) \\
\end{tabular}
\label{table:104A}
\end{table}

\begin {table} [!htbp]
\small
\caption{Generating set 4 for degree 10 graphs of class 4: diameter $k\equiv 4 \pmod 5$.} 
\centering 
\setlength {\tabcolsep} {5pt}
\begin{tabular}{ @ { } c l c c c c r l l l l c c c c r l} 
\noalign {\vskip 2mm} 
\hline\hline 
\noalign {\vskip 1mm} 
Generator & \multicolumn {8} {l} {Gen set 4, $k\equiv 4 \pmod {25}$} && \multicolumn {7} {l} {Gen set 4, $k\equiv 9 \pmod {25}$} \\ 
\hline 
\noalign {\vskip 1mm} 
$g_1$ &   ( 0 & 0 & 0 & 0 & 0 & 1 ) & & & & ( 0 & 0 & 0 & 0 & 0 & 1 ) \\
$g_2$ &  ( 1 & 6 & -2 & 46 & -21 & 16 ) & /10 & & & ( 1 & -8 & 18 & -56 & 37 & -28 ) & /10 \\
$g_3$ &  ( 1 & 6 & -2 & 46 & -11 & 16 ) & /10 & & & ( 1 & 2 & -2 & 24 & -33 & 12 ) & /10 \\
$g_4$ &  ( 1 & -4 & 13 & -24 & 24 & -9 )& /5 & & & ( 1 & 2 & 3 & 19 & -8 & 7 ) & /5 \\
$g_5$ &  ( 1 & 1 & 3 & 16 & -11 & 11 ) & /5 & & & ( 1 & 2 & 3 & 19 & -3 & 7 ) & /5 \\
\hline
\noalign {\vskip 1mm} 
 & \multicolumn {8} {l} {Gen set 4, $k\equiv 19 \pmod {25}$} && \multicolumn {7} {l} {Gen set 4, $k\equiv 24 \pmod {25}$} \\ 
\hline 
\noalign {\vskip 1mm} 
$g_1$ &   ( 0 & 0 & 0 & 0 & 0 & 1 ) & & & & ( 0 & 0 & 0 & 0 & 0 & 1 ) \\
$g_2$ &  ( 1 & -6 & 18 & -40 & 43 & -16 ) & /10 & & & ( 1 & -10 & 18 & -62 & 21 & -20 ) & /10 \\
$g_3$ &  ( 1 & 4 & -2 & 40 & -27 & 24 ) & /10 & & & ( 1 & -10 & 18 & -62 & 31 & -20 ) & /10 \\
$g_4$ &  ( 1 & -6 & 13 & -35 & 13 & -11 ) & /5 & & & ( 1 & -5 & 13 & -32 & 21 & -15 ) & /5 \\
$g_5$ &  ( 1 & -6 & 13 & -35 & 18 & -11 ) & /5 & & & ( 1 & 0 & 3 & 8 & -14 & 5 ) & /5 \\
\hline
\noalign {\vskip 1mm} 
&\multicolumn {16} {l} {\footnotesize $( c_5\:  c_4\:  c_3\:  c_2\:  c_1\:  c_0)/b=(c_5a^5+c_4a^4+c_3a^3+c_2a^2+c_1a+c_0)/b$ where $a=(4k+4)/5$} 
\end{tabular}
\label{table:104B}
\end{table}

\begin {table} [!htbp]
\small
\caption{Generating set 5 for degree 10 graphs of class 4 (diameter $k\equiv 4 \pmod 5$, part one.} 
\centering 
\setlength {\tabcolsep} {5pt}
\begin{tabular}{ @ { } c l c c c c r r l l c c c r r l} 
\noalign {\vskip 2mm} 
\hline\hline 
\noalign {\vskip 1mm} 
Generator & \multicolumn {8} {l} {Gen set 5, $k\equiv 4 \pmod {85}$} & \multicolumn {7} {l} {Gen set 5, $k\equiv 9 \pmod {85}$} \\ 
\hline 
\noalign {\vskip 1mm} 
$g_1$ &   ( 0 & 0 & 0 & 0 & 0 & 1 ) & & & ( 0 & 0 & 0 & 0 & 0 & 1 ) \\
$g_2$ &  ( 1 & 14 & -10 & 100 & -37 & 20 ) & /34 & & ( 1 & 10 & -14 & 90 & -69 & 40 ) & /34 \\
$g_3$ &  ( 2 & -6 & 14 & -21 & -23 & 23 ) & /17 & & ( 1 & -7 & 20 & -46 & 50 & -11 ) & /17 \\
$g_4$ &  (5 & 2 & 18 & 58 & -49 & 32 ) & /34 & & ( 3 & -4 & 26 & -19 & 48 & -33 ) & /17 \\
$g_5$ &  ( 7 & -4 & 32 & 20 & -55 & 4 ) & /34 & & ( 7 & 2 & 38 & 52 & -7 & 8 ) & /34 \\
\hline 
\noalign {\vskip 1mm} 
 & \multicolumn {8} {l} {Gen set 5, $k\equiv 14 \pmod {85}$} & \multicolumn {7} {l} {Gen set 5, $k\equiv 19 \pmod {85}$} \\ 
\hline 
\noalign {\vskip 1mm} 
$g_1$ &   ( 0 & 0 & 0 & 0 & 0 & 1 ) & & & ( 0 & 0 & 0 & 0 & 0 & 1 ) \\
\end{tabular}
\label{table:104C}
\end{table}

\begin {table} [!htbp]
\small
\caption{Generating set 5 for degree 10 graphs of class 4 (diameter $k\equiv 4 \pmod 5$, part two.} 
\centering 
\setlength {\tabcolsep} {5pt}
\begin{tabular}{ @ { } c l c c c c r r l l c c c r r l} 
\noalign {\vskip 2mm} 
\hline\hline 

\noalign {\vskip 1mm} 
Generator & \multicolumn {8} {l} {Gen set 5, $k\equiv 49 \pmod {85}$} & \multicolumn {7} {l} {Gen set 5, $k\equiv 54 \pmod {85}$} \\ 
\hline 
\noalign {\vskip 1mm} 
$g_1$ &   ( 0 & 0 & 0 & 0 & 0 & 1 ) & & & ( 0 & 0 & 0 & 0 & 0 & 1 ) \\
$g_2$ &  ( 2 & -10 & 27 & -65 & 47 & -25 ) & /17 & & ( 1 & 8 & -16 & 68 & -85 & 16 ) & /34 \\
$g_3$ &  ( 3 & 2 & 15 & 30 & -6 & 5 ) & /17 & & ( 3 & -10 & 20 & -34 & -51 & 48 ) & /34 \\
$g_4$ &  ( 7 & -18 & 52 & -66 & -31 & 40 ) & /34 & & ( 5 & 6 & 22 & 68 & -17 & 12 ) & /34 \\
$g_5$ &  ( 4 & -3 & 20 & 6 & -25 & 1 ) & /17 & & ( 4 & -2 & 21 & 17 & -17 & 13 ) & /17 \\
\hline
\noalign {\vskip 1mm} 
 & \multicolumn {8} {l} {Gen set 5, $k\equiv 59 \pmod {85}$} & \multicolumn {7} {l} {Gen set 5, $k\equiv 64 \pmod {85}$} \\ 
\hline 
\noalign {\vskip 1mm} 
$g_1$ &   ( 0 & 0 & 0 & 0 & 0 & 1 ) & & & ( 0 & 0 & 0 & 0 & 0 & 1 ) \\
\end{tabular}
\label{table:104D}
\end{table}


\newpage
Degree 11, diameter $k\equiv 1$, isomorphism class 1a

\begin {table} [!htbp]
\small
\caption{ Generating sets 1 \& 2 for degree 11 graphs of class 1a: diameter $k\equiv 1 \pmod 5$.} 
\centering 
\setlength {\tabcolsep} {6pt}
\begin{tabular}{ @ { } c l c c c c r r l l c c c r r l} 
\noalign {\vskip 2mm} 
\hline\hline 
\noalign {\vskip 1mm} 
Generator & \multicolumn {8} {l} {Gen set 1, $k\equiv 1 \pmod {25}$} & \multicolumn {7} {l} {Gen set 1, $k\equiv 6 \pmod {25}$} \\ 
\hline 
\noalign {\vskip 1mm} 
$g_1$ &   ( 0 & 0 & 0 & 0 & 0 & 1 ) & & & ( 0 & 0 & 0 & 0 & 0 & 1 ) \\
$g_2$ & ( 0 & 1 & 0 & 4 & 5 & -5 ) &/5 & & ( 1 & -5 & 9 & -25 & 0 & 5 ) & /5 \\
$g_3$ & ( 1 & -5 & 9 & -25 & 0 & 5 ) & /5 & & ( 1 & 1 & 5 & 10 & 0 & 0 ) & /5  \\
$g_4$ & ( 1 & 1 & 4 & 9 & -5 & -5 ) & /5 & &  ( 2 & -4 & 14 & -15 & 5 & 5 ) & /5  \\
$g_5$ & ( 2 & -4 & 13 & -16 & 0 & 0 ) & /5 & &  ( 2 & -1 & 12 & 0 & 5 & -5 ) & /5 \\
\hline
\noalign {\vskip 1mm} 
 & \multicolumn {8} {l} {Gen set 1, $k\equiv 11 \pmod {25}$} & \multicolumn {7} {l} {Gen set 1, $k\equiv 16 \pmod {25}$} \\ 
\hline 
\noalign {\vskip 1mm} 
$g_1$ &   ( 0 & 0 & 0 & 0 & 0 & 1 ) & & & ( 0 & 0 & 0 & 0 & 0 & 1 ) \\
\end{tabular}
\label{table:111aA}
\end{table}

\begin {table} [!htbp]
\small
\caption{One example of generating sets 3 \& 4 for degree 11 graphs of class 1a: diameter $k\equiv 1 \pmod 5$.} 
\centering 
\setlength {\tabcolsep} {5pt}
\begin{tabular}{ @ { } c l c c c c r r l l c c c r r l} 
\noalign {\vskip 2mm} 
\hline\hline 
\noalign {\vskip 1mm} 
Generator & \multicolumn {8} {l} {Gen set 3, $k\equiv 6 \pmod {65}$} & \multicolumn {7} {l} {Gen set 4, $k\equiv 6 \pmod {85}$} \\ 
\hline 
\noalign {\vskip 1mm} 
$g_1$ &   ( 0 & 0 & 0 & 0 & 0 & 1 ) & & & ( 0 & 0 & 0 & 0 & 0 & 1 ) \\
$g_2$ &  ( 1 & -11 & 14 & -63 & 13 & -13 ) & /13 & & (2 & -6 & 21 & -35 & 34 & -17 ) & /17 \\
$g_3$ &  ( 2 & 17 & 2 & 108 & 13 & 13 ) & /13 & & ( 3 & -9 & 23 & -44 & 0 & 0 ) & /17 \\
$g_4$ &  ( 3 & 19 & 3 & 123 & 0 & 13 ) & /13 & & ( 3 & -9 & 23 & -44 & 17 & -17 ) & /17 \\
$g_5$ &  ( 4 & 8 & 17 & 60 & 0 & 0 ) & /13 & & ( 5 & 2 & 27 & 23 & 17 & -17 ) & /17 \\
\hline
\noalign {\vskip 1mm} 
&\multicolumn {15} {l} {\footnotesize $( c_5\:  c_4\:  c_3\:  c_2\:  c_1\:  c_0)/b=(c_5a^5+c_4a^4+c_3a^3+c_2a^2+c_1a+c_0)/b$ where $a=(4k+1)/5$} 
\end{tabular}
\label{table:111aB}
\end{table}


\newpage
Degree 11, diameter $k\equiv 1$, isomorphism class 1b

\begin {table} [!htbp]
\small
\caption{Generating sets 1 \& 2 for degree 11 graphs of class 1b: diameter $k\equiv 1 \pmod 5$.} 
\centering 
\setlength {\tabcolsep} {7pt}
\begin{tabular}{ @ { } c l c c c c r r l l c c c r r l} 
\noalign {\vskip 2mm} 
\hline\hline 
\noalign {\vskip 1mm} 
Generator & \multicolumn {8} {l} {Gen set 1, $k\equiv 1 \pmod {5}$} & \multicolumn {7} {l} {Gen set 2, $k\equiv 1 \pmod {5}$} \\ 
\hline 
\noalign {\vskip 1mm} 
$g_1$ &   ( 0 & 0 & 0 & 0 & 0 & 1 ) & & & ( 0 & 0 & 0 & 0 & 0 & 1 ) \\
$g_2$ &  ( 0 & 2 & -1 & 11 & 2 & -1 ) &  & & (0 & 4 & -3 & 23 & -2 & -1 ) & \\
$g_3$ &  ( 0 & 5 & -3 & 29 & 1 & 1 ) &  & & ( 0 & 10 & -7 & 58 & -3 & -1 ) & \\
$g_4$ &  ( 0 & 8 & -5 & 46 & 1 & 1 ) &  & & ( 0 & 10 & -7 & 58 & -2 & 0 ) & \\
$g_5$ &  ( 0 & 8 & -5 & 46 & 2 & 0 ) &  & & ( 0 & 13 & -9 & 75 & -3 & -1 ) & \\
\hline
\noalign {\vskip 1mm} 
&\multicolumn {15} {l} {\footnotesize $( c_5\:  c_4\:  c_3\:  c_2\:  c_1\:  c_0)/b=(c_5a^5+c_4a^4+c_3a^3+c_2a^2+c_1a+c_0)/b$ where $a=(4k+1)/5$} 
\end{tabular}
\label{table:111bA}
\end{table}

\begin {table} [!htbp]
\small
\caption{ Generating set 3 for degree 11 graphs of class 1b: diameter $k\equiv 1 \pmod 5$.} 
\centering 
\setlength {\tabcolsep} {6pt}
\begin{tabular}{ @ { } c l c c c c r r l l c c c r r l} 
\noalign {\vskip 2mm} 
\hline\hline 
\noalign {\vskip 1mm} 
Generator & \multicolumn {8} {l} {Gen set 3, $k\equiv 1 \pmod {35}$} & \multicolumn {7} {l} {Gen set 3, $k\equiv 6 \pmod {35}$} \\ 
\hline 
\noalign {\vskip 1mm} 
$g_1$ &   ( 0 & 0 & 0 & 0 & 0 & 1 ) & & & ( 0 & 0 & 0 & 0 & 0 & 1 ) \\
$g_2$ & ( 1 & -7 & 7 & -36 & -14 & 0 ) &/7 & & ( 1 & 3 & 3 & 23 & 0 & 7 ) & /7 \\
$g_3$ & ( 2 & -7 & 14 & -30 & -7 & 7 ) & /7 & & ( 2 & -8 & 13 & -38 & -14 & 0 ) & /7  \\
$g_4$ & ( 2 & 0 & 14 & 5 & 21 & -7 ) & /7 & &  ( 3 & -5 & 16 & -15 & -21 & 7 ) & /7  \\
$g_5$ & ( 3 & -7 & 21 & -31 & 0 & -7 ) & /7 & &  ( 3 & 2 & 16 & 20 & 7 & -7 ) & /7 \\
\hline
\noalign {\vskip 1mm} 
 & \multicolumn {8} {l} {Gen set 3, $k\equiv 16 \pmod {35}$} & \multicolumn {7} {l} {Gen set 3, $k\equiv 21 \pmod {35}$} \\ 
\hline 
\noalign {\vskip 1mm} 
$g_1$ &   ( 0 & 0 & 0 & 0 & 0 & 1 ) & & & ( 0 & 0 & 0 & 0 & 0 & 1 ) \\
\end{tabular}
\label{table:111bB}
\end{table}


\newpage
Degree 11, diameter $k\equiv 2$

\begin {table} [!htbp]
\small
\caption{Generating sets 1 to 3 for degree 11 graphs of class 2: diameter $k\equiv 2 \pmod 5$.} 
\centering 
\setlength {\tabcolsep} {5pt}
\begin{tabular}{ @ { } c l c c c c r r l l l c c c c r l} 
\noalign {\vskip 2mm} 
\hline\hline 
\noalign {\vskip 1mm} 
Generator & \multicolumn {8} {l} {Gen set 1, $k\equiv 2 \pmod {5}$} && \multicolumn {7} {l} {Gen set 2, $k\equiv 2 \pmod {5}$} \\ 
\hline 
\noalign {\vskip 1mm} 
$g_1$ &  ( 0 & 0 & 0 & 0 & 0 & 1 ) & & & & ( 0 & 0 & 0 & 0 & 0 & 1 ) \\
$g_2$ &  ( 0 & 0 & 1 & -1 & 6 & -2 ) & /2 & & & ( 0 & 4 & -5 & 28 & -11 & 11 ) \\
$g_3$ &  ( 0 & 1 & -1 & 6 & 0 & -2 ) & /2 & & &  ( 0 & 10 & -13 & 71 & -31 & 29 ) \\
$g_4$ &  ( 1 & -4 & 12 & -23 & 20 & -6 ) & /2 & & &   ( 0 & 10 & -13 & 71 & -30 & 29 )  \\
$g_5$ &  ( 1 & -3 & 10 & -16 & 12 & -6 ) & /2 & & & ( 0 & 13 & -17 & 92 & -40 & 37 )  \\
\hline
\noalign {\vskip 1mm} 
 & \multicolumn {8} {l} {Gen set 3, $k\equiv 2 \pmod {15}$} && \multicolumn {7} {l} {Gen set 3, $k\equiv 7 \pmod {15}$} \\ 
\hline 
\noalign {\vskip 1mm} 
$g_1$ &   ( 0 & 0 & 0 & 0 & 0 & 1 ) & & & & ( 0 & 0 & 0 & 0 & 0 & 1 ) \\
\end{tabular}
\label{table:112A1}
\end{table}

\begin {table} [!htbp]
\small
\caption{Generating set 4 for degree 11 graphs of class 2: diameter $k\equiv 2 \pmod 5$.} 
\centering 
\setlength {\tabcolsep} {5pt}
\begin{tabular}{ @ { } c l c c c c r r l l l c c c c r l} 
\noalign {\vskip 2mm} 
\hline\hline 
\noalign {\vskip 1mm} 
Generator & \multicolumn {8} {l} {Gen set 4, $k\equiv 7 \pmod {25}$} && \multicolumn {7} {l} {Gen set 4, $k\equiv 12 \pmod {25}$} \\ 
\hline 
\noalign {\vskip 1mm} 
$g_1$ &   ( 0 & 0 & 0 & 0 & 0 & 1 ) & & & & ( 0 & 0 & 0 & 0 & 0 & 1 ) \\
$g_2$ &  ( 1 & -6 & 13 & -35 & 13 & -11 ) & /5 & & & ( 1 & -5 & 13 & -32 & 21 & -15 ) & /5 \\
$g_3$ &  ( 1 & -6 & 13 & -35 & 18 & -11  ) & /5 & & & ( 1 & 0 & 3 & 8 & -14 & 5 ) & /5 \\
$g_4$ &  ( 2 & -7 & 21 & -40 & 26 & -17 )& /5 & & & ( 2 & 0 & 11 & 11 & -3 & 5 ) & /5 \\
$g_5$ &  ( 2 & -2 & 11 & 0 & -9 & 3 )& /5 & & & ( 2 & 0 & 11 & 11 & 2 & 5 ) & /5 \\
\hline
\noalign {\vskip 1mm} 
 & \multicolumn {8} {l} {Gen set 4, $k\equiv 17 \pmod {25}$} && \multicolumn {7} {l} {Gen set 4, $k\equiv 22 \pmod {25}$} \\ 
\hline 
\noalign {\vskip 1mm} 
$g_1$ &   ( 0 & 0 & 0 & 0 & 0 & 1 ) & & & & ( 0 & 0 & 0 & 0 & 0 & 1 ) \\
$g_2$ &  ( 1 & -4 & 13 & -24 & 24 & -9 ) & /5 & & & ( 1 & 2 & 3 & 19 & -8 & 7 ) & /5 \\
$g_3$ &  ( 1 & 1 & 3 & 16 & -11 & 11  ) & /5 & & & ( 1 & 2 & 3 & 19 & -3 & 7 ) & /5 \\
$g_4$ &  ( 2 & -8 & 21 & -43 & 18 & -13 )& /5 & & & ( 2 & -6 & 21 & -32 & 29 & -11 ) & /5 \\
$g_5$ &  ( 2 & -8 & 21 & -43 & 23 & -13 )& /5 & & & ( 2 & -1 & 11 & 8 & -6 & 9 ) & /5 \\
\hline
\noalign {\vskip 1mm} 
&\multicolumn {16} {l} {\footnotesize $( c_5\:  c_4\:  c_3\:  c_2\:  c_1\:  c_0)/b=(c_5a^5+c_4a^4+c_3a^3+c_2a^2+c_1a+c_0)/b$ where $a=(4k+2)/5$} 
\end{tabular}
\label{table:112A2}
\end{table}

\begin {table} [!htbp]
\small
\caption{Generating set 5 for degree 11 graphs of class 2: diameter $k\equiv 2 \pmod 5$, part one.} 
\centering 
\setlength {\tabcolsep} {5pt}
\begin{tabular}{ @ { } c l c c c c r r l l c c c r r l} 
\noalign {\vskip 2mm} 
\hline\hline 
\noalign {\vskip 1mm} 
Generator & \multicolumn {8} {l} {Gen set 5, $k\equiv 2 \pmod {85}$} & \multicolumn {7} {l} {Gen set 5, $k\equiv 7 \pmod {85}$} \\ 
\hline 
\noalign {\vskip 1mm} 
$g_1$ &   ( 0 & 0 & 0 & 0 & 0 & 1 ) & & & ( 0 & 0 & 0 & 0 & 0 & 1 ) \\
$g_2$ &  ( 2 & -2 & 18 & -11 & 43 & -31 ) & /17 & & ( 2 & -10 & 27 & -65 & 47 & -25 ) & /17 \\
$g_3$ &  ( 5 & -5 & 28 & -2 & -20 & -1 ) & /17 & & ( 3 & 2 & 15 & 30 & -6 & 5 ) & /17 \\
$g_4$ &  ( 6 & -6 & 37 & 1 &-7 & 9 ) & /17 & & ( 4 & -3 & 20 & 6 & -25 & 1 ) & /17 \\
$g_5$ &  ( 8 & -8 & 55 & -10 & 19 & -5 ) & /17 & & ( 5 & -8 & 42 & -35 & 58 & -37 ) & /17 \\
\hline 
\noalign {\vskip 1mm} 
 & \multicolumn {8} {l} {Gen set 5, $k\equiv 12 \pmod {85}$} & \multicolumn {7} {l} {Gen set 5, $k\equiv 17 \pmod {85}$} \\ 
\hline 
\noalign {\vskip 1mm} 
$g_1$ &   ( 0 & 0 & 0 & 0 & 0 & 1 ) & & & ( 0 & 0 & 0 & 0 & 0 & 1 ) \\
\end{tabular}
\label{table:112B}
\end{table}

\begin {table} [!htbp]
\small
\caption{Generating set 5 for degree 11 graphs of class 2: diameter $k\equiv 2 \pmod 5$, part two.} 
\centering 
\setlength {\tabcolsep} {5pt}
\begin{tabular}{ @ { } c l c c c c r r l l c c c r r l} 
\noalign {\vskip 2mm} 
\hline\hline 
\noalign {\vskip 1mm} 
Generator & \multicolumn {8} {l} {Gen set 5, $k\equiv 62 \pmod {85}$} & \multicolumn {7} {l} {Gen set 5, $k\equiv 72 \pmod {85}$} \\ 
\hline 
\noalign {\vskip 1mm} 
$g_1$ &   ( 0 & 0 & 0 & 0 & 0 & 1 ) & & & ( 0 & 0 & 0 & 0 & 0 & 1 ) & \\
$g_2$ &  ( 2 & 4 & 7 & 38 & -11 & 7 ) & /17 & & ( 2 & -12 & 25 & -70 & 31 & -15 ) & /17 \\
$g_3$ &  ( 3 & -11 & 36 & -62 & 60 & -15 ) & /17 & & (3 & -1 & 12 & 14 & -30 & 3 ) & /17 \\
$g_4$ &  ( 7 & -20 & 67 & -105 & 72 & -35 ) & /17 & & ( 7 & -8 & 45 & -7 & -2 & 7 ) & /17 \\
$g_5$ &  ( 8 & -18 & 62 & -69 & 7 & 11 ) & /17 & & ( 8 & -14 & 66 & -59 & 73 & -43 ) & /17 \\
\hline
\noalign {\vskip 1mm} 
 & \multicolumn {8} {l} {Gen set 5, $k\equiv 77 \pmod {85}$} & \multicolumn {7} {l} {Gen set 5, $k\equiv 82 \pmod {85}$} \\ 
\hline 
\noalign {\vskip 1mm} 
$g_1$ &   ( 0 & 0 & 0 & 0 & 0 & 1 ) & & & ( 0 & 0 & 0 & 0 & 0 & 1 ) \\
$g_2$ &  ( 1 & -10 & 17 & -62 & 26 & -13 ) & /17 & & ( 1 & 3 & -4 & 30 & -40 & 7 ) & /17 \\
$g_3$ &  ( 4 & -6 & 34 & -27 & 53 & -35 ) & /17 & & ( 3 & -8 & 22 & -29 & -18 & 21 ) & /17 \\
$g_4$ &  ( 5 & -16 & 51 & -89 & 62 & -31 ) & /17 & & ( 5 & -2 & 31 & 14 & 4 & 1 ) & /17 \\
$g_5$ &  ( 7 & -19 & 68 & -94 & 80 & -23 ) & /17 & & ( 8 & -10 & 53 & -15 & 3 & 5 ) & /17 \\
\hline
\noalign {\vskip 1mm} 
&\multicolumn {15} {l} {\footnotesize $( c_5\:  c_4\:  c_3\:  c_2\:  c_1\:  c_0)/b=(c_5a^5+c_4a^4+c_3a^3+c_2a^2+c_1a+c_0)/b$ where $a=(4k+2)/5$} 
\end{tabular}
\label{table:112C}
\end{table}


\newpage
Degree 11, diameter $k\equiv 3$

\begin {table} [!htbp]
\small
\caption{Generating sets 1 to 3 for degree 11 graphs of class 3: diameter $k\equiv 3 \pmod 5$.} 
\centering 
\setlength {\tabcolsep} {6pt}
\begin{tabular}{ @ { } c l c c c c r r l l l c c c c r l} 
\noalign {\vskip 2mm} 
\hline\hline 
\noalign {\vskip 1mm} 
Generator & \multicolumn {8} {l} {Gen set 1, $k\equiv 3 \pmod {5}$} && \multicolumn {7} {l} {Gen set 2, $k\equiv 3 \pmod {5}$} \\ 
\hline 
\noalign {\vskip 1mm} 
$g_1$ &  ( 0 & 0 & 0 & 0 & 0 & 1 ) & & & & ( 0 & 0 & 0 & 0 & 0 & 1 ) \\
$g_2$ &  ( 0 & 0 & 1 & 1 & 6 & 2 ) & /2 & & & ( 0 & 4 & 5 & 28 & 11 & 11 ) \\
$g_3$ &  ( 0 & 1 & 1 & 6 & 0 & -2 ) & /2 & & &  ( 0 & 10 & 13 & 71 & 30 & 29 ) \\
$g_4$ &  ( 1 & 3 & 10 & 16 & 12 & 6 ) & /2 & & &   ( 0 & 10 & 13 & 71 & 31 & 29 )  \\
$g_5$ &  ( 1 & 4 & 12 & 23 & 20 & 6 ) & /2 & & & ( 0 & 13 & 17 & 92 & 40 & 37 )  \\
\hline
\noalign {\vskip 1mm} 
 & \multicolumn {8} {l} {Gen set 3, $k\equiv 3 \pmod {15}$} && \multicolumn {7} {l} {Gen set 3, $k\equiv 8 \pmod {15}$} \\ 
\hline 
\noalign {\vskip 1mm} 
$g_1$ &   ( 0 & 0 & 0 & 0 & 0 & 1 ) & & & & ( 0 & 0 & 0 & 0 & 0 & 1 ) \\
\end{tabular}
\label{table:113A1}
\end{table}

\begin {table} [!htbp]
\small
\caption{Generating set 4 for degree 11 graphs of class 3: diameter $k\equiv 3 \pmod 5$.} 
\centering 
\setlength {\tabcolsep} {5pt}
\begin{tabular}{ @ { } c l c c c c r r l l l c c c c r l} 
\noalign {\vskip 2mm} 
\hline\hline 
\noalign {\vskip 1mm} 
Generator & \multicolumn {8} {l} {Gen set 4, $k\equiv 3 \pmod {25}$} && \multicolumn {7} {l} {Gen set 4, $k\equiv 8 \pmod {25}$} \\ 
\hline 
\noalign {\vskip 1mm} 
$g_1$ &   ( 0 & 0 & 0 & 0 & 0 & 1 ) & & & & ( 0 & 0 & 0 & 0 & 0 & 1 ) \\
$g_2$ &  ( 1 & -2 & 3 & -19 & -8 & -7 ) & /5 & & & ( 1 & -1 & 3 & -16 & -11 & -11 ) & /5 \\
$g_3$ &  ( 1 & -2 & 3 & -19 & -3 & -7  ) & /5 & & & ( 1 & 4 & 13 & 24 & 24 & 9 ) & /5 \\
$g_4$ &  ( 2 & 1 & 11 & -8 & -6 & -9 )& /5 & & & ( 2 & 8 & 21 & 43 & 18 & 13 ) & /5 \\
$g_5$ &  ( 2 & 6 & 21 & 32 & 29 & 11 )& /5 & & & ( 2 & 8 & 21 & 43 & 23 & 13 ) & /5 \\
\hline
\noalign {\vskip 1mm} 
 & \multicolumn {8} {l} {Gen set 4, $k\equiv 13 \pmod {25}$} && \multicolumn {7} {l} {Gen set 4, $k\equiv 18 \pmod {25}$} \\ 
\hline 
\noalign {\vskip 1mm} 
$g_1$ &   ( 0 & 0 & 0 & 0 & 0 & 1 ) & & & & ( 0 & 0 & 0 & 0 & 0 & 1 ) \\
$g_2$ &  ( 1 & 0 & 3 & -8 & -14 & -5 ) & /5 & & & ( 1 & 6 & 13 & 35 & 13 & 11 ) & /5 \\
$g_3$ &  ( 1 & 5 & 13 & 32 & 21 & 15 ) & /5 & & & ( 1 & 6 & 13 & 35 & 18 & 11 ) & /5 \\
$g_4$ &  ( 2 & 0 & 11 & -11 & -3 & -5 )& /5 & & & ( 2 & 2 & 11 & 0 & -9 & -3 ) & /5 \\
$g_5$ &  ( 2 & 0 & 11 & -11 & 2 & -5 )& /5 & & & ( 2 & 7 & 21 & 40 & 26 & 17 ) & /5 \\
\hline
\noalign {\vskip 1mm} 
&\multicolumn {16} {l} {\footnotesize $( c_5\:  c_4\:  c_3\:  c_2\:  c_1\:  c_0)/b=(c_5a^5+c_4a^4+c_3a^3+c_2a^2+c_1a+c_0)/b$ where $a=(4k-2)/5$} 
\end{tabular}
\label{table:113A2}
\end{table}

\begin {table} [!htbp]
\small
\caption{Generating set 5 for degree 11 graphs of class 3: diameter $k\equiv 3 \pmod 5$, part one.} 
\centering 
\setlength {\tabcolsep} {5pt}
\begin{tabular}{ @ { } c l c c c c r r l l c c c r r l} 
\noalign {\vskip 2mm} 
\hline\hline 
\noalign {\vskip 1mm} 
Generator & \multicolumn {8} {l} {Gen set 5, $k\equiv 3 \pmod {85}$} & \multicolumn {7} {l} {Gen set 5, $k\equiv 8 \pmod {85}$} \\ 
\hline 
\noalign {\vskip 1mm} 
$g_1$ &   ( 0 & 0 & 0 & 0 & 0 & 1 ) & & & ( 0 & 0 & 0 & 0 & 0 & 1 ) \\
$g_2$ &  ( 1 & -3 & -4 & -30 & -40 & -7 ) & /17 & & ( 1 & 10 & 17 & 62 & 26 & 13 ) & /17 \\
$g_3$ &  ( 3 & 8 & 22 & 29 & -18 & -21 ) & /17 & & ( 4 & 6 & 34 & 27 & 53 & 35 ) & /17 \\
$g_4$ &  ( 5 & 2 & 31 & -14 & 4 & -1 ) & /17 & & ( 5 & 16 & 51 & 89 & 62 & 31 ) & /17 \\
$g_5$ &  ( 8 & 10 & 53 & 15 & 3 & -5 ) & /17 & & ( 7 & 19 & 68 & 94 & 80 & 23 ) & /17 \\
\hline 
\noalign {\vskip 1mm} 
 & \multicolumn {8} {l} {Gen set 5, $k\equiv 13 \pmod {85}$} & \multicolumn {7} {l} {Gen set 5, $k\equiv 23 \pmod {85}$} \\ 
\hline 
\noalign {\vskip 1mm} 
$g_1$ &   ( 0 & 0 & 0 & 0 & 0 & 1 ) & & & ( 0 & 0 & 0 & 0 & 0 & 1 ) \\
\end{tabular}
\label{table:113B}
\end{table}

\begin {table} [!htbp]
\small
\caption{Generating set 5 for degree 11 graphs of class 3: diameter $k\equiv 3 \pmod 5$, part two.} 
\centering 
\setlength {\tabcolsep} {5pt}
\begin{tabular}{ @ { } c l c c c c r r l l c c c r r l} 
\noalign {\vskip 2mm} 
\hline\hline 

\noalign {\vskip 1mm} 
Generator & \multicolumn {8} {l} {Gen set 5, $k\equiv 48 \pmod {85}$} & \multicolumn {7} {l} {Gen set 5, $k\equiv 53 \pmod {85}$} \\ 
\hline 
\noalign {\vskip 1mm} 
$g_1$ &   ( 0 & 0 & 0 & 0 & 0 & 1 ) & & & ( 0 & 0 & 0 & 0 & 0 & 1 ) \\
$g_2$ &  ( 4 & 14 & 43 & 81 & 57 & 29 ) & /17 & & ( 1 & 8 & 19 & 57 & 42 & 23 ) & /17 \\
$g_3$ &  ( 6 & 4 & 39 & -6 & 9 & 1 ) & /17 & & ( 2 & -1 & 4 & -22 & -35 & -5 ) & /17 \\
$g_4$ &  ( 7 & 16 & 54 & 61 & 2 & -13 ) & /17 & & ( 6 & 14 & 46 & 53 & -3 & -15 ) & /17 \\
$g_5$ &  ( 8 & 11 & 52 & 26 & -5 & 7 ) & /17 & & ( 7 & 22 & 65 & 110 & 56 & 25 ) & /17 \\
\hline
\noalign {\vskip 1mm} 
 & \multicolumn {8} {l} {Gen set 5, $k\equiv 58 \pmod {85}$} & \multicolumn {7} {l} {Gen set 5, $k\equiv 63 \pmod {85}$} \\ 
\hline 
\noalign {\vskip 1mm} 
$g_1$ &   ( 0 & 0 & 0 & 0 & 0 & 1 ) & & & ( 0 & 0 & 0 & 0 & 0 & 1 ) \\
\end{tabular}
\label{table:113C}
\end{table}


\newpage
Degree 11, diameter $k\equiv 4$

\begin {table} [!htbp]
\small
\caption{Generating sets 1 and 2 for degree 11 graphs of class 4: diameter $k\equiv 4 \pmod 5$.} 
\centering 
\setlength {\tabcolsep} {6pt}
\begin{tabular}{ @ { } c l c c c c r r l l l c c c r r r} 
\noalign {\vskip 4mm} 
\hline\hline 
\noalign {\vskip 1mm} 
Generator & \multicolumn {7} {l} {Gen set 1, $k\equiv 4 \pmod {5}$} &&& \multicolumn {7} {l} {Gen set 2, $k\equiv 4 \pmod {5}$} \\ 
\hline 
\noalign {\vskip 1mm} 
$g_1$ &  ( 0 & 0 & 0 & 0 & 0 & 1 ) & & & & ( 0 & 0 & 0 & 0 & 0 & 1 ) \\
$g_2$ &  ( 0 & 0 & 1 & 0 & 5 & 0 ) & /2 & & & ( 0 & 2 & 3 & 13 & 18 & 11 ) \\
$g_3$ &  ( 0 & 1 & 1 & 5 & 7 & -4 ) & /2 & & & ( 0 & 5 & 7 & 33 & 43 & 27 ) \\
$g_4$ &  ( 1 & -1 & 6 & -5 & 1 & 0 ) & /2 & & & ( 0 & 8 & 11 & 52 & 67 & 41 ) \\
$g_5$ &  ( 1 & 0 & 6 & 0 & 1 & -2 ) & /2 & & & ( 0 & 8 & 11 & 52 & 68 & 42 ) \\
\hline
\noalign {\vskip 1mm} 
&\multicolumn {16} {l} {\footnotesize $( c_5\:  c_4\:  c_3\:  c_2\:  c_1\:  c_0)/b=(c_5a^5+c_4a^4+c_3a^3+c_2a^2+c_1a+c_0)/b$ where $a=(4k-1)/5$} 
\end{tabular}
\label{table:114A}
\end{table}

\begin {table} [!htbp]
\small
\caption{Generating set 3 for degree 11 graphs of class 4: diameter $k\equiv 4 \pmod 5$, part one.} 
\centering 
\setlength {\tabcolsep} {5pt}
\begin{tabular}{ @ { } c l c c c c r r l l c c c r r l} 
\noalign {\vskip 2mm} 
\hline\hline 
\noalign {\vskip 1mm} 
Generator & \multicolumn {8} {l} {Gen set 3, $k\equiv 9 \pmod {85}$} & \multicolumn {7} {l} {Gen set 3, $k\equiv 14 \pmod {85}$} \\ 
\hline 
\noalign {\vskip 1mm} 
$g_1$ &   ( 0 & 0 & 0 & 0 & 0 & 1 ) & & & ( 0 & 0 & 0 & 0 & 0 & 1 ) \\
$g_2$ &  ( 1 & -3 & 7 & -15 & 0 & 1 ) & /17 & & ( 3 & -4 & 7 & -23 & -57 & -39 ) & /17 \\
$g_3$ &  ( 6 & -1 & 25 & -5 & -51 & -45 ) & /17 & & ( 5 & -1 & 23 & -10 & -44 & -48 ) & /17 \\
$g_4$ &  ( 7 & 13 & 49 & 82 & 51 & 7 ) & /17 & & ( 5 & -1 & 23 & -10 & -27 & -31 ) & /17 \\
$g_5$ &  ( 7 & 13 & 49 & 82 & 68 & 24 ) & /17 & & ( 8 & 12 & 47 & 69 & 18 & -19 ) & /17 \\
\hline 
\noalign {\vskip 1mm} 
 & \multicolumn {8} {l} {Gen set 3, $k\equiv 19 \pmod {85}$} & \multicolumn {7} {l} {Gen set 3, $k\equiv 24 \pmod {85}$} \\ 
\hline 
\noalign {\vskip 1mm} 
$g_1$ &   ( 0 & 0 & 0 & 0 & 0 & 1 ) & & & ( 0 & 0 & 0 & 0 & 0 & 1 ) \\
\end{tabular}
\label{table:114B}
\end{table}

\begin {table} [!htbp]
\small
\caption{Generating set 3 for degree 11 graphs of class 4: diameter $k\equiv 4 \pmod 5$, part two.} 
\centering 
\setlength {\tabcolsep} {5pt}
\begin{tabular}{ @ { } c l c c c c r r l l c c c c r l} 
\noalign {\vskip 2mm} 
\hline\hline 
\noalign {\vskip 1mm} 
Generator & \multicolumn {8} {l} {Gen set 3, $k\equiv 49 \pmod {85}$} & \multicolumn {7} {l} {Gen set 3, $k\equiv 54 \pmod {85}$} \\ 
\hline 
\noalign {\vskip 1mm} 
$g_1$ &   ( 0 & 0 & 0 & 0 & 0 & 1 ) & & & ( 0 & 0 & 0 & 0 & 0 & 1 ) \\
$g_2$ &  ( 2 & -2 & 13 & -9 & 2 & -1 ) & /17 & & ( 1 & -5 & -1 & -34 & -52 & -40 ) & /17 \\
$g_3$ &  ( 3 & -3 & 11 & -22 & -48 & -44 ) & /17 & & ( 1 & -5 & -1 & -34 & -35 & -23 ) & /17 \\
$g_4$ &  ( 3 & -3 & 11 & -22 & -31 & -27 ) & /17 & & ( 4 & -3 & 13 & -17 & -55 & -41 ) & /17 \\
$g_5$ &  ( 5 & 12 & 41 & 71 & 73 & 23 ) & /17 & & ( 5 & 9 & 29 & 51 & 12 & -13 ) & /17 \\
\hline
\noalign {\vskip 1mm} 
 & \multicolumn {8} {l} {Gen set 3, $k\equiv 59 \pmod {85}$} & \multicolumn {7} {l} {Gen set 3, $k\equiv 64 \pmod {85}$} \\ 
\hline 
\noalign {\vskip 1mm} 
$g_1$ &   ( 0 & 0 & 0 & 0 & 0 & 1 ) & & & ( 0 & 0 & 0 & 0 & 0 & 1 ) & \\
\end{tabular}
\label{table:114C}
\end{table}

\begin {table} [!htbp]
\small
\caption{One example of generating sets 4 and 5 for degree 11 graphs of class 4: diameter $k\equiv 4 \pmod 5$.} 
\centering 
\setlength {\tabcolsep} {5pt}
\begin{tabular}{ @ { } c l c c c c r r l l c c c r r l} 
\noalign {\vskip 2mm} 
\hline\hline 
\noalign {\vskip 1mm} 
Generator & \multicolumn {8} {l} {Gen set 4, $k\equiv 4 \pmod {115}$} & \multicolumn {7} {l} {Gen set 5, $k\equiv 4 \pmod {235}$} \\ 
\hline 
\noalign {\vskip 1mm} 
$g_1$ &   ( 0 & 0 & 0 & 0 & 0 & 1 ) & & & ( 0 & 0 & 0 & 0 & 0 & 1 ) \\
$g_2$ &  ( 4 & 2 & 17 & 6 & -26 & -28 ) & /23 & & ( 7 & 1 & 50 & 3 & 69 & -29 ) & /47 \\
$g_3$ &  ( 7 & 15 & 47 & 91 & 58 & -3 ) & /23 & & ( 17 & 36 & 108 & 202 & 87 & -57 ) & /47 \\
$g_4$ &  ( 10 & 5 & 54 & 38 & -19 & -1 ) & /23 & & ( 18 & 16 & 95 & 95 & -24 & -41 ) & /47 \\
$g_5$ &  ( 11 & 17 & 64 & 97 & 9 & -29 ) & /23 & & ( 22 & 30 & 137 & 184 & 96 & -24 ) & /47 \\
\hline
\noalign {\vskip 1mm} 
&\multicolumn {15} {l} {\footnotesize $( c_5\:  c_4\:  c_3\:  c_2\:  c_1\:  c_0)/b=(c_5a^5+c_4a^4+c_3a^3+c_2a^2+c_1a+c_0)/b$ where $a=(4k-1)/5$} 
\end{tabular}
\label{table:114D}
\end{table}

\end{document}